\documentclass[11pt,a4paper,reqno]{amsart}
\usepackage{bm}

\usepackage{setspace}
\usepackage{indentfirst}

\usepackage{tabls}
\usepackage{url}
\usepackage{color}
\definecolor{refkey}{rgb}{1,0,0}
\definecolor{labelkey}{rgb}{0,0,1}
\definecolor{labelkey}{rgb}{1,1,1}
\usepackage[linktocpage = true]{hyperref}
\hypersetup{
 colorlinks = true,
 citecolor = blue,
}

\usepackage{times}
\usepackage{amsthm}
\usepackage{amsmath}
\usepackage{amsfonts}
\usepackage{amssymb}
\usepackage{amsrefs}
\usepackage{mathrsfs}
\usepackage{xypic}
\usepackage{enumerate}
\usepackage{extarrows}
\usepackage{graphicx}
\usepackage{subfigure}
\usepackage[toc,page]{appendix}
\usepackage[all,cmtip]{xy}

\allowdisplaybreaks

\hypersetup{pdfpagemode=UseOutlines,colorlinks=false,pdfpagelayout=SinglePage,pdfstartview=FitH,bookmarksopen=true}

\theoremstyle{plain}% default
\newtheorem{theorem}{Theorem}[section]
\newtheorem{lemma}[theorem]{Lemma}
\newtheorem{corollary}[theorem]{Corollary}
\newtheorem{definition}[theorem]{Definition}
\newtheorem{proposition}[theorem]{Proposition}

\newtheorem{example}[theorem]{Example}
\newtheorem{remark}[theorem]{Remark}

\theoremstyle{definition}

\newtheorem*{note*}{Note}

\theoremstyle{remark}

\DeclareMathOperator{\lin}{lin}
\DeclareMathOperator{\ad}{ad} \DeclareMathOperator{\Ad}{Ad}
 
\DeclareMathOperator{\id}{id}

\DeclareMathOperator{\End}{End} 
\DeclareMathOperator{\Hom}{Hom}

 \DeclareMathOperator{\mult}{mult}
\DeclareMathOperator{\bmult}{bmult}

\newcommand{\sections}[1]{ { \Gamma} (#1)}
\newcommand{\XX}{\mathfrak{X}} % vector fields
\newcommand{\abs}[1]{\left\vert#1\right\vert}
\newcommand{\crossedmoduletriple}[3]{(#1\stackrel{#2}{\rightarrow}#3)}

\newcommand{\thetaalgebradegone}{\mathrm{\vartheta}}

\newcommand{\galgebradegzero}{\mathfrak{g}}

\newcommand{\bracketmapbracket}[2]{[#1,#2]_2}%{\derivesbracket}
\newcommand{\tobefilledin}{\,\stackrel{\centerdot}{}\,}

\newcommand{\jet}{\mathfrak{J}}

\newcommand{\liftingd}{\jmath^1}

\newcommand{\g}{\mathfrak{L}}

\newcommand{\Gpd}{\mathcal{G}}

\newcommand{\CinfM}{C^\infty(M)}
\newcommand{\Der}{\mathrm{Der}}

\newcommand{\XkmultG}[1]{\mathfrak{X}^{#1}_{\mathrm{mult}}(\Gpd)}
\newcommand{\OmegakmultG}[1]{\Omega^{#1}_{\mathrm{mult}}(\Gpd)}

\newcommand{\inputvariable}{~\cdot~}
\newcommand{\degree}[1]{\abs{#1}}
\newcommand{\Z}{\mathbb{Z}}

\newcommand{\dstar}{d_*}
\begin{document}

\title[The weak Lie $2$-algebra  of  multiplicative forms]{
The weak Lie $2$-algebra  of  multiplicative forms\\ on a quasi-Poisson groupoid
}
%\thanks{Research partially supported by NSFC grants 12071241(Chen \& Liu) and the National Key Research and Development Program of China (No. 2021YFA1002000) (Lang).}

\author{Zhuo Chen}%{}%\\
\address{Department of Mathematics, Tsinghua University, Beijing 100084, China}%\\\vspace{1mm} 
%\\
\email{\href{mailto:chenzhuo@tsinghua.edu.cn}{chenzhuo@tsinghua.edu.cn}}

\author{Honglei Lang$^\diamond$} 
\address{College of Science, China Agricultural University, Beijing 100083, China}
%%\vspace{1mm}
\email{\href{mailto:hllang@cau.edu.cn}{hllang@cau.edu.cn}(corresponding author)}

\author{Zhangju Liu}
\address{School of Mathematical and Statistical Sciences, Henan University, Kaifeng 475004, China}
\email{\href{mailto: liuzj@pku.edu.cn}{zhangju@henu.edu.cn}}

\makeatother

\begin{abstract} 
	Berwick-Evens and Lerman recently showed that the category of  vector fields on a geometric stack has the structure of a   Lie $2$-algebra. Motivated by   this work, we present a construction of graded weak Lie $2$-algebras associated with quasi-Poisson groupoids  based on the space of multiplicative forms on the groupoid and differential forms on the base manifold. We also establish a morphism between the Lie $2$-algebra of multiplicative multivector fields and the weak Lie $2$-algebra of multiplicative forms, allowing us to compare and relate different aspects of Lie $2$-algebra theory within the context of quasi-Poisson geometry. As an infinitesimal analogy,   we explicitly determine the associated weak Lie $2$-algebra structure of IM $1$-forms along with differential $1$-forms on the base manifold for any quasi-Lie bialgebroid.

%	Quasi-Poisson groupoids are $(+1)$-shifted Poisson objects in the category of  {differentiable} stacks. 	It is known that the space of multiplicative multivector fields on a Lie groupoid is a ($\Z$-)graded Lie $2$-algebra. 	
%		 In   this paper  we find a natural    ($\Z$-)graded weak Lie $2$-algebra structure which is  primarily composed of  multiplicative      forms on a  quasi-Poisson groupoid. Moreover, a morphism between the two Lie $2$-algebras is established.
	
  	\vskip0.2cm
 
   
\end{abstract}

%\textcolor{red}{53D17 Poisson manifolds, Poisson groupoids and algebroids; 16E45 differential graded algebras and applications; 58H05 pseudogroups and differentiable groupoids; 18H55 homotopical algebra; 58C50 Analysis on supermanifolds or graded manifolds}

\maketitle
\tableofcontents

\section{Introduction}
 
 A quasi-Poisson groupoid is a Lie groupoid $\Gpd$ equipped with a multiplicative $2$-vector field $P$ and some datum controlling $[P,P]$ (i.e., homotopic to zero). These structures are generalizations of Poisson groupoids \cite{W1}, which were initiated from Poisson Lie groups \cite{LuW} and symplectic groupoids \cites{CDW,Weinstein1987}. From the perspective proposed in \cite{BCLX}, quasi-Poisson groupoids can be viewed as $(+1)$-shifted differentiable Poisson stacks. Quasi-Poisson groupoids are essential since they represent the core constituents of objects with either multiplicative (multi-)vector fields or multiplicative differential forms. General multiplicative structures on Lie groupoids have widespread applications in various contexts, as demonstrated in \cites{DIRAC,ILX,MOMENT,MORITA}.

 Let us review some works related to multiplicative vector fields and forms. Berwick-Evans and Lerman \cite{BEL} demonstrated that vector fields on a differentiable stack $X$ can be understood in terms of a Lie $2$-algebra. This Lie $2$-algebra comprises the multiplicative vector fields on a Lie groupoid that presents $X$, along with the sections of the Lie algebroid $A$ associated with the Lie groupoid. The Lie $2$-algebra also appeared in \cite{MehtaPhd}. Furthermore, \cite{BCLX} established that every Lie groupoid $\Gamma$ corresponds to a ($\Z$-)graded Lie $2$-algebra. Recent research has focused on multiplicative differential forms on Lie groupoids due to their connection to infinitesimal multiplicative (IM-) forms and Spencer operators on the Lie algebroid level \cites{BC, CMS1, C}. In a recent work \cite{CLL}, the authors find that if $\Gpd$ is a Poisson Lie groupoid, then the space $\OmegakmultG{\bullet}$ of multiplicative forms on $\Gpd$ has a differential graded Lie algebra (DGLA) structure. Furthermore, when combined with $\Omega^\bullet(M)$, which is the space of forms on the base manifold $M$, $\OmegakmultG{\bullet}$ forms a canonical DGLA crossed module. This supplements the  previously known fact \cites{BCLX,BEL} that multiplicative multivector fields on $\Gpd$ form a DGLA crossed module with the Schouten algebra $\Gamma(\wedge^\bullet A)$ stemming from the tangent Lie algebroid $A$.

  Building on the aforementioned works \cites{MehtaPhd, BEL, BCLX,CLL}, our paper aims to investigate algebraic structures for multiplicative forms on quasi-Poisson groupoids. Specifically, we aim to establish (graded) weak Lie $2$-algebras, cubic $L_\infty$-algebras, and other higher objects.  {To ensure completeness and facilitate understanding for readers from different fields, we start  by introducing the basic definition   of   multiplicative forms on Lie groupoids.}

  \emph{      $\bullet$   Multiplicative forms.}  
 For general theory of  Lie groupoids and Lie algebroids, we refer to the standard text \cites{Mackenzie}. In this paper, we follow conventions of our previous work \cites{CL1,CLL}: $\Gpd\rightrightarrows M$ denotes a Lie groupoid over $M$ whose source and target maps are $s$ and $t$ (both mapping $\Gpd$ to $M$). 
 The tangent Lie algebroid of $\Gpd$ is standard: $A=\ker(s_*)|_{M}$. The letter $A$ could also refer to   a general Lie algebroid over $M$ with the Lie bracket  $[~\cdot~,~\cdot~]$ on $\Gamma(A)$  and anchor map $\rho: A\to TM$.

 For $u\in \Gamma(\wedge^k A)$, denote by $\overleftarrow{u}\in \Gamma(\wedge^k T\Gpd)$ the left-invariant $k$-vector field on $\Gpd$ associated to $u$.  In the meantime, for all $\omega\in \Omega^l(M)$, we have the pullback $s^*\omega\in \Omega^l(\Gpd)$ along the source map $s: \Gpd\to M$. 
 
 Further,  we recall the definitions of multiplicative forms and tensors on a Lie groupoid $\Gpd$ over $M$. Denote by $\Gpd^{(2)}$ the set of composable elements, i.e., $(g,r)\in \Gpd\times \Gpd$, satisfying  $s(g)=t(r)$.
 Denote by $m: \Gpd^{(2)}\to \Gpd$ the groupoid multiplication. 
 \begin{definition}\cite{Weinstein1987} 	A $k$-form $\Theta\in \Omega^k(\Gpd)$  is called {\bf multiplicative} if it satisfies the relation
 	\[m^*\Theta=\mathrm{pr}_1^*\Theta+\mathrm{pr}_2^* \Theta,\]
 	where $\mathrm{pr}_1,\mathrm{pr}_2: \Gpd^{(2)}\to \Gpd$ are the obvious projections.
 \end{definition}   
 Moreover, a function $F\in C^\infty(\Gpd)$  is multiplicative if it is a multiplicative $0$-form. Namely, it satisfies $F(gr)=F(g)+F(r)$ for all $(g,r)\in \Gpd^{(2)}$.
 
 \emph{  $\bullet$  Multiplicative  tensors.} 
The notion of multiplicative tensors is introduced in \cite{BD} by using of the tangent and cotangent Lie groupoids of a given Lie groupoid $\mathcal{G}$.
 \begin{definition}
 	Consider the Lie groupoid
 	\[\mathbb{G}^{(k,l)}: (\oplus^k T^*\Gpd)\oplus (\oplus^l T\Gpd)\rightrightarrows \oplus^k A^*\oplus (\oplus^l TM).\]
 	A $(k,l)$-tensor $T\in \mathcal{T}^{k,l}(\Gpd)$ on $\Gpd$  is called {\bf multiplicative} if it is a multiplicative function on $\mathbb{G}^{(k,l)}$.
 \end{definition}

  \emph{      $\bullet$  Quasi-Poisson groupoids.} Let us also recall the notion of quasi-Poisson groupoids.
 
 \begin{definition}\label{DEF:ILX}\cite{ILX}
 	A {\bf quasi-Poisson groupoid} is a triple $(\Gpd, P,\Phi)$, where  $\Gpd$ is a groupoid  whose Lie algebroid is $A$,  $P\in \mathfrak{X}_{\mult}^2(\Gpd)$,    $\Phi$   $\in \Gamma(\wedge^3A)$,  and they are compatible in the sense that
 	\begin{eqnarray*}\label{quasipoisson1}
 	\frac{1}{2}[P,P]&=&\overrightarrow{\Phi}-\overleftarrow{\Phi},\\ \mbox{~and~}\qquad [P,\overrightarrow{\Phi}]&=&0.\label{quasipoisson2}
 	\end{eqnarray*}
 	
 \end{definition}
 
 We shall show in Section \ref{Sec:fromQPG} how a quasi-Poisson groupoid gives rise to a weak Lie $2$-algebra and a graded weak Lie $2$-algebra. See below for a summary of our main results and  Section \ref{Sec:PreAlg} for   precise definitions of the related algebraic objects.

  \emph{      $\bullet$ The two main results.} 
  In this paper, our focus is on the study of multiplicative forms on quasi-Poisson groupoids and their interactions with the given quasi-Poisson structure. We start by associating a canonical weak Lie $2$-algebra (Theorem \ref{Thm:case1}) with any quasi-Poisson groupoid $(\mathcal{G},P,\Phi)$. This triple consists of  
   \[\Omega^1(M)\xrightarrow{J}\Omega^1_{\mult}(\mathcal{G}),\qquad J(\gamma):=s^*\gamma-t^*\gamma.\]  
  Here $\Omega^1(M)$ is the space  of   $1$-forms on the base manifold $M$, and $\Omega^1_{\mult}(\mathcal{G})$ is the space of multiplicative  $1$-forms on the groupoid $\Gpd$.
  
  One important feature of our result is the construction of a homotopy map ($3$-bracket) $$[~\cdot~,~\cdot~,~\cdot~]_3: \wedge^3 \Omega^1_{\mult}(\mathcal{G})\to \Omega^1(M).$$ This is not immediately evident, but can be expressed explicitly in Equation \eqref{3br}. Furthermore, we extend the above weak Lie $2$-algebra to a graded weak Lie $2$-algebra (Theorem \ref{Thm:secondmain} ~(i)) --- a triple of graded objects $\Omega^\bullet(M)\xrightarrow{J} \Omega^\bullet_{\mult}(\mathcal{G})$ where $J$ is as defined   in the same fashion previously. In this case, the homotopy map takes the form $$[~\cdot~,~\cdot~,~\cdot~]: \Omega^p_{\mult}(\mathcal{G})\wedge \Omega^q_{\mult}(\mathcal{G})\wedge \Omega^s_{\mult}(\mathcal{G})\to \Omega^{p+q+s-2}(M)$$ and has a more intricate construction.

  \emph{  $\bullet$  Passing to IM $1$-forms.} 
 The infinitesimal counterpart of a multiplicative $k$-form on $\Gpd$ is the notion of IM $k$-form of the tangent Lie algebroid $A$ of $\Gpd$; see \cites{BC}. Quasi-Lie bialgebroids, on the other hand, are infinitesimal replacements of quasi-Poisson groupoids \cite{ILX}. This suggests a natural expectation for an analogy of our main Theorem \ref{Thm:case1} --- a weak Lie $2$-algebra underlying IM $1$-forms associated with a quasi-Lie bialgebroid.  In Section \ref{Sec:IM1forms}, we explicitly construct a weak Lie $2$-algebra underlying IM $1$-forms associated with a quasi-Lie bialgebroid. Furthermore, we demonstrate the compatibility of this structure with the groupoid-level objects.

   \emph{  $\bullet$  Future work.}   In this paper, our focus does not include an examination of how the Morita equivalence class of a quasi-Poisson groupoid affects weak Lie $2$-algebras. However, given that quasi-Poisson groupoids are $1$-shifted Poisson stacks, it is reasonable to anticipate that the weak Lie $2$-algebras we are analyzing give rise to a stacky object. Investigating this possibility is  one of the future research.     Moreover, we are intrigued by quasi-symplectic groupoids, which have an interesting connection with quasi-Poisson groupoids \cite{0801}. We owe Henrique Bursztyn thanks for bringing this relationship to our attention.

 \emph{  $\bullet$  Structure of the paper.}    In Section \ref{Sec:PreAlg} we recall  definitions of
curved DGLAs, cubic $L_\infty$-algebras, and weak Lie  $2$-algebras.    There we also define the notion of a ($\Z$-)graded weak Lie  $2$-algebra. The next Section \ref{Sec:fromQPG} is
devoted to stating and proving our main results, namely Theorems \ref{Thm:case1} and \ref{Thm:secondmain},  through a series of identities, and we have dedicated considerable effort towards establishing a number of lemmas and propositions.  In this
section we also establish morphisms between the many different algebraic structures, and study the special case of quasi-Poisson groups.   Section \ref{Sec:linearQP2GP} describes   {a demonstration model}, namely the linear quasi-Poisson  $2$-group arising from a Lie $2$-algebra.    This model looks  {easy} but is actually very informative. We calculate the corresponding various higher algebraic structures. Finally, 
in Section \ref{Sec:IM1forms}, we analyze the weak Lie $2$-algebra structure on IM $1$-forms of a quasi-Lie bialgebroid, and explore its relationship with the objects introduced in Section \ref{Sec:fromQPG}.

% \emph{  $\bullet$  Acknowledgments.}     Lang is grateful to Henrique Bursztyn, Marco Zambon, and Chenchang Zhu for useful discussions and  remarks on this work. We would like to express special appreciation to Ping Xu for providing valuable comments and suggestions that helped us to refine our ideas and improve the quality of this paper.

\section{Preliminaries of algebraic objects}\label{Sec:PreAlg}

\subsection{Curved DGLAs, cubic $L_\infty$-algebras, and weak Lie  $2$-algebras} 
Throughout the paper,  graded means $\mathbb{Z}$-graded.

 \begin{definition}\label{Def:Linfty-algebra} \cites{G, S,V}
 	A \textbf{curved $L_\infty$-algebra}  is a graded vector space $\g$ equipped with a collection of skew-symmetric multilinear maps $[~\cdots~]_k: \Lambda^k \g \to \g $ of degree $(2-k)$, for all $k\geqslant  $0$ $, such that the (higher) Jacobi identities
 	\begin{equation}\label{Jacobi identities}
 	\sum \limits_{i=0}^n \sum \limits_{\sigma \in {\rm Sh}(i,n-i) }(-1)^{i(n-i)}\chi(\sigma;x_1,~\cdots~,x_n) [[x_{\sigma(1)},~\cdots~,x_{\sigma(i)}]_i,x_{\sigma(i+1)},~\cdots~, x_{\sigma(n)}]_{n-i+1}=0,
 	\end{equation}
 	hold for all homogeneous elements $x_1,~\cdots~,x_n \in V$  and $n\geqslant  0$. If the $0$-bracket $[]_0$ (an element in $\g_2$) vanishes, the curved $L_\infty$-structure is called flat, or uncurved, and we simply call $\g$ an {\bf $L_\infty$-algebra}.
 \end{definition}

 Here the symbol $\mathrm{Sh}(p, \, q)$ denotes the set of $(p, \, q)$-unshuffles. 
 	Note that in the literature there are different conventions about the sign $(\pm 1)$ in Equation \eqref{Jacobi identities}. 
  
 \textbf{Notation:}  It is common to write the unary bracket $[\inputvariable]_1$ as $d$, which is a degree $1$ endomorphism on $\g$. We also prefer to use the symbol $c  $ to denote the $0$-bracket, which is an element in   $\g_2$.
 
 In the current paper, we will encounter some particular cases of curved $L_\infty$-algebras.
\subsubsection*{ $\bullet$ Curved DGLA} 
	  \emph{
	If a curved $L_\infty$-algebra $\g$ whose $k$-brackets vanish for all $k\geqslant 3$, then $\g$  is known as a {\bf curved  DGLA}. In this situation, the Jacobi identities are the following:
	\begin{itemize}
		\item[-]  $d(c)=0$;
		\item[-]  $d^2(x)=-[c,x]_2$;
		\item[-]  $d[x_1,x_2]_2=[d x_1,x_2]_2+(-1)^{\degree{x_1}\degree{x_2}}[dx_2, x_1]_2$; 
		\item[-]  $[[x_1,x_2]_2,x_3]_2+(-1)^{1+\degree{x_2}~\cdot~\degree{x_3}}[[x_1,x_3]_2,x_2]_2
		+(-1)^{\degree{x_1}(\degree{x_2}+\degree{x_3})}[[x_2,x_3]_2,x_1]_2=0$.
	\end{itemize}}

The following  example is well-known. \begin{example}\label{Example:NP}
Let $N$ be a manifold and  $P\in \mathfrak{X}^2(N)$ a   bivector field. Then the space of multivector fields on $N$ forms a curved DGLA: $(\mathfrak{X}^\bullet(N)[1],c,d_P,[~\cdot~,~\cdot~]_2)$,  where  $c=\frac{1}{2}[P,P] \in \mathfrak{X}^3(N)$, $d_P:=[P,~\cdot~]$, and $[~\cdot~,~\cdot~]_2$ is the Schouten bracket.	 Here we emphasize that the convention of degree on $ \mathfrak{X}^\bullet(N)[1]$ is by setting $\mathrm{deg} (\mathfrak{X}^k(N)[1]):=k-1$.  

\end{example}

\subsubsection*{ $\bullet$ Curved cubic $L_\infty$-algebra} \emph{ If a curved $L_\infty$-algebra $\g$ whose $k$-brackets vanish for all $k\geqslant 4$, then $\g$  is known as a {\bf curved cubic $L_\infty$-algebra}.}
For example,  a curved cubic $L_\infty$-algebra arises from any split Courant algebroid \cite{Costa}.

	\subsubsection*{ $\bullet$ Cubic $L_{\infty}$-algebra} \emph{
	 When an $L_\infty$-algebra has all trivial brackets except  $[\inputvariable]_1 =d$, $[~\cdot~,~\cdot~]_2$, and $[~\cdot~,~\cdot~,~\cdot~]_3$,  it is called  a \textbf{cubic $L_{\infty}$-algebra}. The $3$-bracket $[~\cdot~,~\cdot~,~\cdot~]_3$ is also called the \textit{homotopy} map.} 
	 
	 \begin{example}\cite{G}*{Theorem 5.2} \label{Example:NP2}
	 We now recall a construction of a cubic $L_{\infty}$-algebra associated to the aforementioned $P\in \mathfrak{X}^2(N)$. Indeed, on the space $\Omega^1(N)$ of $1$-forms, there is a skew-symmetric bracket, called the $P$-bracket: 
	  \begin{equation}\label{Eqt:Pbracket1forms}[\alpha,\beta]_P=\mathcal{L}_{P^\sharp \alpha} \beta -\mathcal{L}_{P^\sharp \beta} \alpha-dP(\alpha,\beta)
	  %=dP(\alpha,\beta) + \iota_{P^\sharp \alpha}  d\beta-\iota_{P^\sharp \beta}  d\alpha,
	  \qquad \forall \alpha,\beta\in \Omega^1(N),\end{equation}
	  where   $P^\sharp:T^*N\to TN$  sends $\alpha\in \Omega^1(N)$  to $ \iota_{\alpha}P$. The $P$-bracket can be defined on  forms of all degrees by the Leibniz rule.
	  Then the quadruple $(\Omega^\bullet(N)[1],d, [~\cdot~,~\cdot~]_P, [~\cdot~,~\cdot~,~\cdot~]_3)$ constitutes  a cubic $L_\infty$-algebra, where $d$ is the de Rham differential, $[~\cdot~,~\cdot~,~\cdot~]_3: \Omega^p(N)\wedge \Omega^q(N)\wedge \Omega^s(N)\to \Omega^{p+q+s-3}(N)$ is defined by \[[\Theta_1,\Theta_2,\Theta_3]_3 =\iota_{\frac{1}{2}[P,P]} (\Theta_1\wedge \Theta_2\wedge\Theta_3),\qquad \Theta_i\in \Omega^1(N)\]
	  on $1$-forms and extended to all forms by requiring the Leibniz rule on each argument. 	 \end{example}
	 
	  { The two examples \ref{Example:NP} and \ref{Example:NP2} are adapted to the case of multiplicative forms on a Lie groupoid --- see Proposition \ref{Lin}. }

\subsubsection*{ $\bullet$ Weak Lie $2$-algebra}  	   Following the terminology of \cite{Baez}, \emph{
 a {\bf weak Lie $2$-algebra}  is a  $2$-term $L_\infty$-algebra concentrated in degrees $(-1)$ and $0$, i.e., $\g=\thetaalgebradegone\oplus \galgebradegzero$ where $\thetaalgebradegone=\g_{-1}$ and  $\galgebradegzero= \g_0$.}
 In this case, we have three structure maps, namely $d: \thetaalgebradegone\to \galgebradegzero$,   $[~\cdot~,~\cdot~]_2: \galgebradegzero\wedge \galgebradegzero\to \galgebradegzero$ and $\galgebradegzero\wedge \thetaalgebradegone\to \thetaalgebradegone$, and homotopy map $[~\cdot~,~\cdot~,~\cdot~]_3: \wedge^3 \galgebradegzero\to \thetaalgebradegone$; and they satisfy    the following compatibility conditions:
	for all $w,x,y,z \in \galgebradegzero$ and $u,v\in
	\thetaalgebradegone$,
	\begin{enumerate}
		\item[]  \begin{equation}\label{lie21}\bracketmapbracket{
				\bracketmapbracket{x}{y}}{z}+ \bracketmapbracket{
				\bracketmapbracket{y}{z}}{x}+\bracketmapbracket{
				\bracketmapbracket{z}{x}}{y}+d[x,y,z]_3=0;
		\end{equation}
		\item[]
		\begin{equation}\label{lie22}
			\bracketmapbracket{
				\bracketmapbracket{x}{y}}{u}+ \bracketmapbracket{
				\bracketmapbracket{y}{u}}{x}+\bracketmapbracket{
				\bracketmapbracket{u}{x}}{y}+[x,y,du]_3=0;
		\end{equation}
		\item[]
		\begin{equation}\label{lie23}
			\bracketmapbracket{du}{v}+\bracketmapbracket{dv}{u}=0,\qquad d\bracketmapbracket{x}{u}=\bracketmapbracket{x}{du};
		\end{equation}
		\item[]\begin{eqnarray}\label{lie25}
		\nonumber
			&&
			-\bracketmapbracket{w}{[x,y,z]_3}
			-\bracketmapbracket{y}{[x,z,w]_3}
			+\bracketmapbracket{z}{[x,y,w]_3}
			+\bracketmapbracket{x}{[y,z,w]_3}
			\\\nonumber &=&
			[\bracketmapbracket{x}{y},z,w]_3-[\bracketmapbracket{x}{z},y,w]_3
			+[\bracketmapbracket{x}{w},y,z]_3+[\bracketmapbracket{y}{z},x,w]_3\\
			&&
			-[\bracketmapbracket{y}{w},x,z]_3+[\bracketmapbracket{z}{w},x,y]_3.
		\end{eqnarray}
		\end{enumerate}		
So, a weak Lie $2$-algebra  is  a particular instance  of cubic   $L_{\infty}$-algebras. 
Moreover, \emph{ if $[~\cdot~,~\cdot~,~\cdot~]_3=0$, then $\g$ is called a {\bf strict Lie $2$-algebra}, or simply a {\bf   Lie $2$-algebra}, or a {\bf Lie algebra crossed module}.} In this case, $\galgebradegzero$ is an ordinary Lie algebra and  it acts on $\thetaalgebradegone$ by setting $x\triangleright u:=[x,u]_2$. Moreover, $\thetaalgebradegone$ is equipped with an Lie bracket  $[u,v]:=[du,v]_2$, $\forall u,v\in \thetaalgebradegone$.
			
			For example, associated to a Lie algebra $\g$ the map $\g\to \Der(\g),x\mapsto [x,\cdot]$ forms a strict Lie  $2$-algebra. Here $\Der(\g)$ stands for derivations of $\g$. For any   real semi-simple Lie algebra  $\g$  with the Killing form $\langle\cdot,\cdot\rangle$, the datum 
			$\mathbb{R}\xrightarrow{0} \g$ is a weak Lie  $2$-algebra for which the $3$-bracket is defined by $[x,y,z]_3=\langle x,[y,z]\rangle$ for all $x,y,z\in \g$.

  For a   Lie algebroid $A$ over $M$, denote by $\Der(A)$ the set of derivations of $A$, i.e. $\mathbb{R}$-linear operators $\sigma: \Gamma(A)\to \Gamma(A)$ (with its symbol $X\in \mathfrak{X}(M)$) satisfying the following properties:  $\forall u,v\in \Gamma(A)$, $f\in \CinfM$
  \begin{itemize}
  	\item $\sigma(fu)=f\sigma(u)+X(f)u$;
  	\item $[X,\rho(u)](f)=\rho(\sigma(u))(f)$;
  	\item $\sigma[u,v]=[\sigma(u),v]+[u,\sigma(v)]$.
  \end{itemize}
  It can be easily verified that $\Gamma(A)\xrightarrow{}\Der(A), u\mapsto [u,~\cdot~]$ is a strict Lie  $2$-algebra.
\subsection{Graded weak Lie  $2$-algebras} 
Next, we generalize the notion of weak Lie $2$-algebras.  
\begin{definition}
A {\bf graded weak Lie  $2$-algebra} is a cubic  $L_\infty$-algebra $\g$ which is the direct sum of two graded subspaces $\galgebradegzero $ and $\thetaalgebradegone$ such that the structure maps $d$, $[~\cdot~,~\cdot~]_2$, and $[~\cdot~,~\cdot~,~\cdot~]_3$ of $\g$ are subject to the following conditions:
\begin{itemize}
	\item $d$ maps $\thetaalgebradegone$ to $\galgebradegzero$ and is trivial on $\galgebradegzero$;
	\item $[~\cdot~,~\cdot~]_2$ maps $\wedge^2 \galgebradegzero$ to $ \galgebradegzero$ and $\galgebradegzero\wedge \thetaalgebradegone$ to $ \thetaalgebradegone$;
	\item $[~\cdot~,~\cdot~,~\cdot~]_3$ maps $\wedge^3 \galgebradegzero$ to  $\thetaalgebradegone$.  
\end{itemize}
 \end{definition}

So,  weak Lie $2$-algebras are special graded weak Lie $2$-algebras, and the later are special 
cubic  $L_\infty$-algebras. 
 In the sequel, we denote   a graded weak Lie $2$-algebra by 
$\g=\crossedmoduletriple{\thetaalgebradegone}{d}{\galgebradegzero}$ to emphasize the key ingredients of $\g$. The bracket $[~\cdot~,~\cdot~]_2$  as a map $\galgebradegzero\wedge \thetaalgebradegone$ to $ \thetaalgebradegone$ would be referred to as the \textit{action} of  $\galgebradegzero$ on $\thetaalgebradegone$, and  we use the more implicit  notation $\triangleright$, although it is not an honest action of Lie algebras.  Again, the $3$-bracket $[~\cdot~,~\cdot~,~\cdot~]_3$ is also called the \textit{homotopy} map. %and denoted by $h(\cdot~,~\cdot~,~\cdot)$ if there is no risk of confusion. 

Moreover, if   $ [~\cdot~,~\cdot~,~\cdot~]_3=0$, then we call $\g$   a   {\bf strict graded   Lie  $2$-algebra}, or simply a  {\bf graded   Lie  $2$-algebra}, or {\bf a graded Lie algebra crossed module}; and in this case,   $\galgebradegzero$ is a graded Lie algebra,   $\triangleright$ is an action indeed, and   $\thetaalgebradegone$ admits an induced graded Lie algebra structure.

An interesting instance of graded Lie  $2$-algebra is the following.
 
\begin{proposition}\label{exBCLX}\cite{BCLX} Let $\Gpd$ be a Lie groupoid.
	{The space $\XkmultG{\bullet}$ of multiplicative multivector fields on $\Gpd$ is  a graded Lie algebra after degree shifts, denoted by $\XkmultG{\bullet}[1]$,   the Schouten bracket being its structure map. Moreover, the map
		\[\Gamma(\wedge^\bullet A)[1]\xrightarrow{T} \XkmultG{\bullet}[1],\qquad u\mapsto \overleftarrow{u}-\overrightarrow{u} \]
	together with the action $\triangleright$ of $\XkmultG{\bullet}[1] $ on $\Gamma(\wedge^\bullet A[1])$ given by 
		\[\overleftarrow{X\triangleright u}=[X,\overleftarrow{u}] \qquad (\mathrm{or } ~ \overrightarrow{X\triangleright u}=[X,\overrightarrow{u}]),\qquad X\in \XkmultG{k}, u\in \Gamma(\wedge^l A)\]
	gives rise to a graded Lie $2$-algebra.}
When  concentrated in degree $0$ parts, it becomes the Lie $2$-algebra $\Gamma(A)\xrightarrow{T} \mathfrak{X}^1_{\mult}(\Gpd)$.
\end{proposition}

\begin{definition}  A {\bf morphism} of graded weak Lie $2$-algebras from  $\crossedmoduletriple{\thetaalgebradegone}{d}{\galgebradegzero}$ to  $\crossedmoduletriple{\thetaalgebradegone'}{d'}{\galgebradegzero'}$  consists of 
	\begin{itemize}
		\item a degree $0$ chain map $F_1=(F_\galgebradegzero,F_\thetaalgebradegone)$, namely, $F_\galgebradegzero: \galgebradegzero \to \galgebradegzero' $ and $F_\thetaalgebradegone:\thetaalgebradegone \to \thetaalgebradegone' $ such that $F_\galgebradegzero \circ d=d'\circ F_\thetaalgebradegone$,
		\item a degree ($-1$) graded skew-symmetric bilinear map $F_2 : \galgebradegzero\wedge \galgebradegzero\to \thetaalgebradegone'$,
		such that the following equations hold for $x,y,z\in \galgebradegzero$ and $u\in \thetaalgebradegone$:
	\end{itemize}
	\begin{itemize}
		\item[\rm {(1)}] $F_\galgebradegzero[x,y]_2-[F_\galgebradegzero(x),F_\galgebradegzero(y)]'_2=d' F_2 (x,y)$,
		\item[\rm {(2)}] $F_\thetaalgebradegone [x,u]_2-[F_\galgebradegzero(x),F_\thetaalgebradegone(u)]'_2=(-1)^{|x|}F_2(x,d(u))$,
		\item[\rm {(3)}] $F_\thetaalgebradegone [x,y,z]_3-[F_\galgebradegzero(x),F_\galgebradegzero(y),F_\galgebradegzero(z)]'_3=[F_\galgebradegzero(x),F_2 (y,z)]'_2-F_2 ([x,y]_2,z)+c.p.$.
	\end{itemize}
	%When $\crossedmoduletriple{\thetaalgebradegone}{\phi}{\galgebradegzero}$ and $\crossedmoduletriple{\thetaalgebradegone'}{\phi}{\galgebradegzero'}$ are weak Lie  $2$-algebras,  we recover the definition of morphisms of weak Lie  $2$-algebras.

	\end{definition} We can express the morphism as described above  more vividly with a diagram: 
\begin{equation*}
	\xymatrix{
		\thetaalgebradegone \ar@{^{}->}[d]_{d} \ar[rr]^{F_\thetaalgebradegone} && \thetaalgebradegone' \ar@{^{}->}[d]^{d'} \\
		\galgebradegzero \ar[rr]^{F_\galgebradegzero}\ar@{.>}[rru]^{F_2} && \galgebradegzero'.}
\end{equation*}

\section{Multiplicative forms on quasi-Poisson groupoids}\label{Sec:fromQPG}

  In this part, we study  higher structures stemming from a smooth manifold $N$ and a bivector field $P\in \mathfrak{X}^2 (N)$.  Recall the skew-symmetric $P$-bracket $[{\tobefilledin},{\tobefilledin}]_P$ on $\Omega^1(N)$   defined by \eqref{Eqt:Pbracket1forms}.
 %\begin{equation}\label{Eqt:Pbracket1forms}[\alpha,\beta]_P=\mathcal{L}_{P^\sharp \alpha} \beta -\mathcal{L}_{P^\sharp \beta} \alpha-dP(\alpha,\beta)
 %=dP(\alpha,\beta) + \iota_{P^\sharp \alpha}  d\beta-\iota_{P^\sharp \beta}  d\alpha,\qquad \forall \alpha,\beta\in \Omega^1(N),\end{equation}
% and a  map   $P^\sharp:T^*N\to TN$ called anchor, by sending $\alpha\in \Omega^1(N)$  to $ \iota_{\alpha}P$. 
We have two key formulas  \cite{Kosmann2}:
 \begin{eqnarray}\label{quasiim}
 [\alpha_1,[\alpha_2,\alpha_3]_P]_P+c.p.=-\frac{1}{2} L_{[P,P](\alpha_1,\alpha_2,\tobefilledin)} \alpha_3+c.p.+d([P,P](\alpha_1,\alpha_2,\alpha_3)),\qquad \forall \alpha_i\in \Omega^1(N),
 \end{eqnarray} 
 and  
 \begin{equation}\label{Eqt:PsharpwrtPbracket}
 P^\sharp[\alpha_1,\alpha_2]_P-[P^\sharp \alpha_1,P^\sharp \alpha_2]=\frac{1}{2}[P,P](\alpha_1,\alpha_2),\qquad \forall \alpha_i\in \Omega^1(N).\end{equation}
 Note that the bracket $[~\cdot~,~\cdot~]_P$ extends to all forms by using the Leibniz rule.

\subsection{The weak Lie $2$-algebra arising from a quasi-Poisson groupoid}~We now  turn to a general Lie groupoid $\Gpd$ with base manifold $M$. As usual, $A:=\ker(s_*)|_{M}$ stands for the tangent Lie algebroid of $\Gpd$. 

Recall from Proposition \ref{exBCLX} that 
the triple
\[\Gamma(A)\xrightarrow{T} \mathfrak{X}^1_{\mult}(\Gpd),\qquad T(u):=\overleftarrow{u}-\overrightarrow{u}\] 
forms a Lie $2$-algebra, where the Lie bracket on $\mathfrak{X}^1_{\mult}(\Gpd)$  is the Schouten bracket $[~\cdot~,~\cdot~]$ and the action $\triangleright: \mathfrak{X}^1_{\mult}(\Gpd)\wedge \Gamma(A)\to \Gamma(A)$ is determined by $\overleftarrow{X\triangleright u}=[X,\overleftarrow{u}]$ for $X\in \mathfrak{X}^1_{\mult}(\Gpd)$ and $u\in \Gamma(A)$. 

  We shift our focus to multiplicative $1$-forms on $\Gpd$, and we have  a parallel result explained below    ---   To any quasi-Poisson groupoid  is associated a canonical weak Lie $2$-algebra.

\begin{theorem}\label{Thm:case1}
Let $(\mathcal{G},P,\Phi)$ be a quasi-Poisson groupoid. Then the triple  
\[\Omega^1(M)\xrightarrow{J}\Omega^1_{\mult}(\mathcal{G}),\qquad J(\gamma):=s^*\gamma-t^*\gamma,\] forms a weak Lie $2$-algebra, where the bracket on $\Omega^1_{\mult}(\mathcal{G})$ is $[~\cdot~,~\cdot~]_P$, the action 
\[\triangleright: \Omega^1_{\mult}(\mathcal{G})\wedge \Omega^1(M)\to \Omega^1(M)\]
and the homotopy map 
\[ [~\cdot~,~\cdot~,~\cdot~]_3: \wedge^3 \Omega^1_{\mult}(\Gpd)\to \Omega^1(M)\] are determined by 
\begin{eqnarray}\label{brform}
s^*(\Theta\triangleright \gamma)&=&[\Theta,s^*\gamma]_P,
\end{eqnarray}
and \begin{eqnarray} s^*[\Theta_1,\Theta_2,\Theta_3]_3&=&L_{\overleftarrow{\Phi}\nonumber(\Theta_1,\Theta_2,~\cdot~)} \Theta_3+c.p.-2d\overleftarrow{\Phi}(\Theta_1,\Theta_2,\Theta_3)\\ &=&d\overleftarrow{\Phi}(\Theta_1,\Theta_2,\Theta_3)+\big(\iota_{\overleftarrow{\Phi}(\Theta_1,\Theta_2)} d\Theta_3+c.p.)\label{3br}
\end{eqnarray}
respectively.
\end{theorem}
 \begin{proof}  
  We first show that the homotopy map $[~\cdot~,~\cdot~,~\cdot~]_3$ given by Equation \eqref{3br} is well-defined. In fact, for $\Theta_i\in \OmegakmultG{1}$,  by \cite{CLL}*{Lemmas 3.5 and 3.8}, we have the following equalities:
\begin{eqnarray}
\label{eqeq1}\overleftarrow{\Phi}(\Theta_1,\Theta_2,\Theta_3)&=&s^*\Phi(\theta_1,\theta_2,\theta_3),\qquad \overrightarrow{\Phi}(\Theta_1,\Theta_2,\Theta_3)=t^*\Phi(\theta_1,\theta_2,\theta_3),\\ \label{eqeq2} \overleftarrow{\Phi}(\Theta_1,\Theta_2,\tobefilledin)&=&\overleftarrow{\Phi(\theta_1,\theta_2,\tobefilledin)},\qquad \overrightarrow{\Phi}(\Theta_1,\Theta_2,\tobefilledin)=\overrightarrow{\Phi(\theta_1,\theta_2,\tobefilledin)},
\end{eqnarray}
where $\theta_i=\mathrm{pr}_{A^*} \Theta_i|_M \in \Gamma(A^*)$.  Also for $u\in \Gamma(A)$ and $\alpha\in \Omega^k_{\mult}(\Gpd)$,  we have $\iota_{\overleftarrow{u}} \alpha=s^*\gamma$ for some $\gamma\in \Omega^{k-1} (M)$. So we see that the right hand side of \eqref{3br} must be of the form $s^*\mu$ where $\mu\in \Omega^1(M)$ is uniquely determined; and hence we simply define   $[\Theta_1,\Theta_2,\Theta_3]_3:=\mu$.   {Moreover, by applying $\mathrm{inv}^*$ on both sides of \eqref{3br}, we  obtain a parallel formula:}
\begin{eqnarray} t^*[\Theta_1,\Theta_2,\Theta_3]_3&=&L_{\overrightarrow{\Phi}\nonumber(\Theta_1,\Theta_2,~\cdot~)} \Theta_3+c.p.-2d\overrightarrow{\Phi}(\Theta_1,\Theta_2,\Theta_3)\\ &=&d\overrightarrow{\Phi}(\Theta_1,\Theta_2,\Theta_3)+\big(\iota_{\overrightarrow{\Phi}(\Theta_1,\Theta_2)} d\Theta_3+c.p.)\label{3brsecond}
\end{eqnarray}
For simplicity, we write $\Phi(\theta_1,\theta_2):=\Phi(\theta_1,\theta_2,~\cdot~)\in \Gamma(A)$
in the sequel.

Next, we verify one by one that what the theorem states  satisfies the axioms   \eqref{lie21} $\sim$ \eqref{lie25} of a weak Lie $2$-algebra:

\begin{itemize}
	\item  To see \eqref{lie21}, we use Equation \eqref{quasiim}, the fact $\frac{1}{2}[P,P]=\overrightarrow{\Phi}-\overleftarrow{\Phi}$, and Equations   \eqref{eqeq1} $\sim$ \eqref{3brsecond} to get
\begin{eqnarray*}
[\Theta_1,[\Theta_2,\Theta_3]_P]_P+c.p.&=&L_{(\overleftarrow{\Phi}-\overrightarrow{\Phi})(\Theta_1,\Theta_2)} \Theta_3+c.p.-2d (\overleftarrow{\Phi}-\overrightarrow{\Phi})(\Theta_1,\Theta_2,\Theta_3)\\ 
&=&d(\overleftarrow{\Phi}-\overrightarrow{\Phi})(\Theta_1,\Theta_2,\Theta_3)+\big(\iota_{(\overleftarrow{\Phi(\theta_1,\theta_2)}-\overrightarrow{\Phi(\theta_1,\theta_2)})}d \Theta_3+c.p.\big)\\ &=&(s^*-t^*)[\Theta_1,\Theta_2,\Theta_3]_3.
\end{eqnarray*}
 This is identically the desired relation.

\item To see \eqref{lie22}, we need   the following formula --- for any $\Theta_1,\Theta_2\in \Omega^1_{\mult}(\Gpd)$ and $\gamma\in \Omega^1(M)$, one has
\begin{eqnarray}\label{foraction}
[\Theta_1,[\Theta_2,s^*\gamma]_P]_P+[\Theta_2,[s^*\gamma,\Theta_1]_P]_P+[s^*\gamma,[\Theta_1,\Theta_2]_P]_P=s^*[\Theta_1,\Theta_2,s^*\gamma-t^*\gamma]_3.
\end{eqnarray}In fact, similar to the way to verify the equation above, we can turn the left hand side of Equation \eqref{foraction}   to 
\begin{eqnarray*}
&&-\frac{1}{2}d[P,P](\Theta_1,\Theta_2,s^*\gamma)-\frac{1}{2}\iota_{[P,P](\Theta_1,\Theta_2)} ds^*\gamma-\frac{1}{2}\iota_{[P,P](\Theta_2,s^*\gamma)} d\Theta_1-\frac{1}{2}\iota_{[P,P](s^*\gamma,\Theta_1)} d\Theta_2
\\ &=&d(\overleftarrow{\Phi}-\overrightarrow{\Phi})(\Theta_1,\Theta_2,s^*\gamma)+\iota_{(\overleftarrow{\Phi}-\overrightarrow{\Phi})(\Theta_1,\Theta_2)} ds^*\gamma+\iota_{(\overleftarrow{\Phi}-\overrightarrow{\Phi})(\Theta_2,s^*\gamma)} d\Theta_1+\iota_{(\overleftarrow{\Phi}-\overrightarrow{\Phi})(s^*\gamma,\Theta_1)} d\Theta_2\\
&=&-ds^*\Phi(\theta_1,\theta_2,\rho^*\gamma)-s^*\iota_{\rho\Phi(\theta_1,\theta_2)} d\gamma-\iota_{\overleftarrow{\Phi(\theta_2,\rho^*\gamma)}}d\Theta_1-\iota_{\overleftarrow{\Phi(\rho^*\gamma,\theta_1)}}d\Theta_2.
\end{eqnarray*}
Here we used \eqref{eqeq1}-\eqref{eqeq2} and the facts 
\begin{eqnarray}\label{strho}
s_*(\overleftarrow{u}-\overrightarrow{u})=s_*(\overleftarrow{u})=-\rho u,\qquad \mathrm{pr}_{ A^* } (s^*\gamma-t^*\gamma)|_M=-\rho^*\gamma\in \Gamma(A^*).
\end{eqnarray}
On the other hand, we have
\begin{eqnarray*}
&&s^*[\Theta_1,\Theta_2,s^*\gamma-t^*\gamma]_3\\ &=&d\overleftarrow{\Phi}(\Theta_1,\Theta_2,s^*\gamma-t^*\gamma)+\iota_{\overleftarrow{\Phi}(\Theta_1,\Theta_2)} d(s^*\gamma-t^*\gamma)+\iota_{\overleftarrow{\Phi}(\Theta_2,s^*\gamma-t^*\gamma)}d\Theta_1+\iota_{\overleftarrow{\Phi}(s^*\gamma-t^*\gamma,\Theta_1)} d\Theta_2\\ &=&
-ds^*\Phi(\theta_1,\theta_2,\rho^*\gamma)-s^*\iota_{\rho\Phi(\theta_1,\theta_2)} d\gamma-\iota_{\overleftarrow{\Phi(\theta_2,\rho^*\gamma)}}d\Theta_1-\iota_{\overleftarrow{\Phi(\rho^*\gamma,\theta_1)}}d\Theta_2.
\end{eqnarray*}
This verifies the desired \eqref{foraction}. By the definition of $\Theta\triangleright \gamma$ in \eqref{brform} and since $s^*$ is injective, \eqref{foraction} implies that
\[\Theta_1\triangleright(\Theta_2\triangleright \gamma)-\Theta_2\triangleright (\Theta_1\triangleright \gamma)-[\Theta_1,\Theta_2]_P\triangleright \gamma=[\Theta_1,\Theta_2,J\gamma]_3.\]
Hence one gets \eqref{lie22}. 

\item The axiom \eqref{lie23} can be verified directly. %implied by Theorem \ref{Thm:LiealgebracrossedmodulePoissonLiegroupoid}. 

\item  It is left to show \eqref{lie25}, namely, 
\begin{eqnarray}\label{23com}
\Theta_1\triangleright [\Theta_2,\Theta_3,\Theta_4]_3+c.p.-\big([[\Theta_1,\Theta_2]_P,\Theta_3,\Theta_4]_3+c.p.\big)=0,\qquad \Theta_i\in \Omega^1_{\mult}(\Gpd).
\end{eqnarray}
Indeed,  it follows from the relation $[P,\overleftarrow{\Phi}]=0$. Let us elaborate on this fact.  On the one hand,  for all $\Theta_i\in \Omega^1(\Gpd)$ (not necessarily multiplicative), we have 
\begin{eqnarray}\label{important}
\nonumber&&[P,\overleftarrow{\Phi}](\Theta_1,\Theta_2,\Theta_3,\Theta_4)=P\lrcorner d(\overleftarrow{\Phi}\lrcorner \Theta)-\overleftarrow{\Phi}\lrcorner d(P\lrcorner \Theta)+(P\wedge \overleftarrow{\Phi})\lrcorner d\Theta\\ \nonumber&=&
\big(\overleftarrow{\Phi}(\Theta_1,\Theta_2,\Theta_3)P(d\Theta_4)+P(d\overleftarrow{\Phi}(\Theta_1,\Theta_2,\Theta_3),\Theta_4)+c.p.(4)\big)\\ \nonumber&&-\big(P(\Theta_1,\Theta_2)\big(\overleftarrow{\Phi}(d\Theta_3,\Theta_4)-\overleftarrow{\Phi}(\Theta_3,d\Theta_4)\big)+\overleftarrow{\Phi}(dP(\Theta_1,\Theta_2),\Theta_3,\Theta_4)+c.p.(6)\big)
\\ \nonumber&&-\big(P(d\Theta_4)\overleftarrow{\Phi}(\Theta_1,\Theta_2,\Theta_3)+c.p.(4)\big)-\big((P^\sharp \Theta_1\wedge \overleftarrow{\Phi}(\Theta_2,\Theta_3))(d\Theta_4)+c.p.(12)\big)\\ \nonumber&&+\big((\overleftarrow{\Phi}(d\Theta_3,\Theta_4)-\overleftarrow{\Phi}(\Theta_3,d\Theta_4))P(\Theta_1,\Theta_2)+c.p.(6)\big)\\ \nonumber&=& {\big(P(d\overleftarrow{\Phi}(\Theta_1,\Theta_2,\Theta_3),\Theta_4)+c.p.(4)\big)}- {\big(\overleftarrow{\Phi}(dP(\Theta_1,\Theta_2),\Theta_3,\Theta_4)+c.p.(6)\big)}\\ &&-
 {(P^\sharp \Theta_1\wedge \overleftarrow{\Phi}(\Theta_2,\Theta_3))(d\Theta_4)+c.p.(12)},
\end{eqnarray}
where $c.p.(4)$ and $c.p.(6)$ stand for the $(3,1)$ and $(2,2)$-unshuffles respectively, and $c.p.(12)$ is the product of $(3,1)$ and $(2,1)$-unshuffles.
By straightforward computation, one can rewrite   Equation \eqref{important}  into a more concise form 
\begin{eqnarray}\label{important2}
\nonumber [P,\overleftarrow{\Phi}](\Theta_1,\Theta_2,\Theta_3,~\cdot~)&=&[P^\sharp(\Theta_3),\overleftarrow{\Phi}(\Theta_1,\Theta_2)]-\overleftarrow{\Phi}([\Theta_1,\Theta_2]_P,\Theta_3)+c.p.\\ &&
+P^\sharp(d\overleftarrow{\Phi}(\Theta_1,\Theta_2,\Theta_3))+P^\sharp\big(\iota_{\overleftarrow{\Phi}(\Theta_1,\Theta_2)} d\Theta_3+c.p.\big).
\end{eqnarray}
On the other hand, by applying 
 $s^*$   on the left hand side of Equation \eqref{23com} we get
\begin{eqnarray*}
&&[\Theta_1,d\overleftarrow{\Phi}(\Theta_2,\Theta_3,\Theta_4)+\big(\iota_{\overleftarrow{\Phi}(\Theta_2,\Theta_3)} d\Theta_4+c.p.(3)\big)]_P+c.p.(4)\\ &&-\big(d\overleftarrow{\Phi}([\Theta_1,\Theta_2]_P,\Theta_3,\Theta_4)+\iota_{\overleftarrow{\Phi}([\Theta_1,\Theta_2]_P,\Theta_3)} d\Theta_4+\iota_{\overleftarrow{\Phi}(\Theta_3,\Theta_4)} d[\Theta_1,\Theta_2]_P+\iota_{\overleftarrow{\Phi}(\Theta_4,[\Theta_1,\Theta_2]_P)} d\Theta_3+c.p.(6)\big)\\ &=&
\big( {dP(\Theta_1,d\overleftarrow{\Phi}(\Theta_2,\Theta_3,\Theta_4))}-\iota_{P^\sharp d\overleftarrow{\Phi}(\Theta_2,\Theta_3,\Theta_4)} d\Theta_1+c.p.(4)\big)\\ &&+\big(L_{P^\sharp \Theta_1} \iota_{\overleftarrow{\Phi}(\Theta_2,\Theta_3)} d\Theta_4-\iota_{P^\sharp \iota_{\overleftarrow{\Phi}(\Theta_2,\Theta_3)}d\Theta_4} d\Theta_1+c.p.(12)\big)
\\ &&-\big(d\overleftarrow{\Phi}( {\iota_{P^\sharp\Theta_1} d\Theta_2-\iota_{P^\sharp\Theta_2} d\Theta_1}+ {dP(\Theta_1,\Theta_2)}, \Theta_3,\Theta_4)\\ &&+\iota_{\overleftarrow{\Phi}([\Theta_1,\Theta_2]_P,\Theta_3)} d\Theta_4+\iota_{\overleftarrow{\Phi}(\Theta_3,\Theta_4)} (L_{P^\sharp \Theta_1} d\Theta_2-L_{P^\sharp \Theta_2} d\Theta_1)+\iota_{\overleftarrow{\Phi}(\Theta_4,[\Theta_1,\Theta_2]_P)} d\Theta_3+c.p.(6)\big)\\ &=&
d[P,\overleftarrow{\Phi}](\Theta_1,\Theta_2,\Theta_3,\Theta_4)+\big(\iota_{[P,\overleftarrow{\Phi}](\Theta_1,\Theta_2,\Theta_3,~\cdot~)} d\Theta_4+c.p.(4)\big),%\\ &=&0,
\end{eqnarray*}
where we have applied Equations \eqref{important}, \eqref{important2} and the Cartan formulas \[d\circ L_X=L_X\circ d,\qquad L_X\circ \iota_Y-\iota_Y\circ L_X=\iota_{[X,Y]}.\]
So if $[P,\overleftarrow{\Phi}]=0$ then \eqref{23com} holds and we complete the proof.\end{itemize}
\end{proof}

 \begin{remark} 
We remark that the   Lie $2$-algebra claimed by Theorem \ref{Thm:case1} can not be drawn directly from the construction as shown in  Example \ref{Example:NP2}.  
\end{remark}
 
\begin{proposition}\label{Lie2homo} Regarding the weak Lie $2$-algebra given by Theorem \ref{Thm:case1} and the one by Proposition \ref{exBCLX}, there is a weak Lie $2$-algebra morphism $(P^\sharp,p^\sharp,\nu)$ between them:  
\begin{equation*}
 		\xymatrix{
 			\Omega^1(M)\ar@{^{}->}[d]_{J} \ar[r]^{p^\sharp} & \Gamma(A)\ar@{^{}->}[d]^{T} \\
 			\Omega^1_{\mult}(\Gpd)\ar[r]^{P^\sharp}\ar@{.>}[ru]^{\nu} & \mathfrak{X}^1_{\mult}(\Gpd)},
 	\end{equation*}
 	where  $p=\mathrm{pr}_{ TM\otimes A } (P|_M)\in \Gamma(TM\otimes A)$ and $\nu: \wedge^2\Omega^1_{\mult}(\Gpd)  \to \Gamma(A)$ is defined by
	\[\nu(\Theta_1,\Theta_2)=-\Phi(\theta_1,\theta_2,~\cdot~),
	\qquad \mbox{where }~ \theta_i=\mathrm{pr}_{  A^*} (\Theta_i|_M)\in \Gamma(A^*).\]
	\end{proposition}
	\begin{proof}   The fact that $T\circ p^\sharp=P^\sharp\circ J$ has been shown in \cite{CLL}*{Proposition 4.8}. We check all the other conditions.  First, by Equations \eqref{Eqt:PsharpwrtPbracket}, \eqref{eqeq1}, \eqref{eqeq2} and \eqref{strho}, we  obtain: 
\[P^\sharp[\Theta_1,\Theta_2]_P-[P^\sharp \Theta_1,P^\sharp \Theta_2]=\overrightarrow{\Phi(\theta_1,\theta_2)}-\overleftarrow{\Phi(\theta_1,\theta_2)}=T\nu (\theta_1,\theta_2),\]
and, for $\Theta\in \Omega^1_{\mult}(\Gpd)$ and $\gamma\in \Omega^1(M)$, 
\[P^\sharp[\Theta,s^*\gamma]_P-[P^\sharp \Theta,P^\sharp s^*\gamma]=(\overrightarrow{\Phi}-\overleftarrow{\Phi})(\Theta,s^*\gamma)=\overleftarrow{\Phi(\theta,\rho^*\gamma)}=\overleftarrow{\nu(\Theta, s^*\gamma-t^*\gamma)}.\]
Second, by the definition of $\Theta\triangleright \gamma$, the relations $P^\sharp s^*(\mu)=\overleftarrow{p^\sharp(\mu)}$ and $[P^\sharp \Theta, \overleftarrow{p^\sharp \gamma}]=\overleftarrow{(P^\sharp \Theta)\triangleright (p^\sharp \gamma)}$ for any $\mu, \gamma\in \Omega^1(M)$, we further have
\[\overleftarrow{p^\sharp (\Theta\triangleright \gamma)}-\overleftarrow{(P^\sharp\Theta) \triangleright  (p^\sharp \gamma)}=\overleftarrow{\nu(\Theta,J\gamma)},\]
which implies that \[p^\sharp (\Theta\triangleright \gamma)-(P^\sharp\Theta) \triangleright  (p^\sharp \gamma)=\nu(\Theta,J\gamma).\]
Finally, we check the third condition
\begin{eqnarray}\label{homo3}
-P^\sharp(\Theta_3) \triangleright \nu (\Theta_1,\Theta_2)+\nu ([\Theta_1,\Theta_2]_P,\Theta_3)+c.p.+p^\sharp([\Theta_1,\Theta_2,\Theta_3]_3)=0.
\end{eqnarray}
In fact, applying the left translation $\overleftarrow{~\cdot~}$ to the left hand side  of \eqref{homo3}, we get
\begin{eqnarray*}
&&\big([P^\sharp(\Theta_3),\overleftarrow{\Phi}(\Theta_1,\Theta_2)]-\overleftarrow{\Phi}([\Theta_1,\Theta_2]_P,\Theta_3)+c.p.\big)\\ &&
+P^\sharp(d\overleftarrow{\Phi}(\Theta_1,\Theta_2,\Theta_3))+P^\sharp\big(\iota_{\overleftarrow{\Phi}(\Theta_1,\Theta_2)} d\Theta_3+c.p.\big)=[P,\overleftarrow{\Phi}](\Theta_1,\Theta_2,\Theta_3,~\cdot~)=0,
\end{eqnarray*}
where we have used \eqref{important2}.
Hence we proved \eqref{homo3} and finished the verification of $(P^\sharp,p^\sharp,\nu)$ being a morphism of the two weak Lie $2$-algebras in question.
\end{proof}

\subsection{The cubic $L_\infty$-algebra of multiplicative forms}
In this part, we  investigate higher degree multiplicative tensors on the Lie groupoid $\Gpd$ whose tangent Lie algebroid is $A$ (all over the base manifold $M$). Let us first make convention of contractions:    For any tensor field $R\in \mathcal{T}^{k,l}(N):=\Gamma(\wedge^k TN \otimes \wedge^l T^*N)$   and $\Theta\in \mathcal{T}^{0,p}(N)=\Omega^p(N)$ on general manifold $N$, define $\iota_R  \Theta \in \mathcal{T}^{k-1,l+p-1}(N)$ as follows:
 \begin{eqnarray}\label{contraction}
\iota_R  \Theta  &=&\sum_i {(-1)^{k-i}} X_1\wedge ~\cdots~ \widehat{X_i}~\cdots~\wedge X_k \otimes (\beta\wedge \iota_{X_i}\Theta),  
\end{eqnarray}
where we have assumed $R=X_1\wedge~\cdots~ \wedge X_k\otimes \beta$.

We also adopt an operator first introduced in \cite{BD}:
\begin{eqnarray}\label{Eqn:leftliftingoperator}
\mathcal{S}:&&~\Gamma(\wedge^k A\otimes \wedge^l T^*M) \to \Gamma(\wedge^k T\Gpd\otimes \wedge^l T^*\Gpd)\\\nonumber&&~\qquad\quad\quad\quad~
u\otimes \omega \quad \mapsto\quad  \overleftarrow{u}\otimes s^*\omega.	 
\end{eqnarray}
Roughly speaking, the operator $\mathcal{S}$ lifts   $u\otimes \omega$ to a left-invariant tensor field on $\Gpd$. 

\begin{lemma}\label{preparation}~ 
\begin{itemize}
\item[\rm (i)] For all $R\in \mathcal{T}^{k,l}_{\mult}(\Gpd)$  and $\Theta\in \Omega^p_{\mult}(\Gpd)$, we have $\iota_R  \Theta \in \mathcal{T}^{k-1,l+p-1}_{\mult}(\Gpd)$;
\item[\rm (ii)] For any $u\in \Gamma(\wedge^k A), \gamma\in \Omega^l(M)$ and $\Theta\in \Omega^p_{\mult}(\Gpd)$, we have 
\begin{equation*}
 \iota_{\mathcal{S}(u \otimes  \gamma)} \Theta=\mathcal{S}(\iota_{u\otimes \gamma} \theta).
\end{equation*}
Here $\theta:=$ $\mathrm{pr}_{ A^*\otimes (\wedge^{p-1}T^*M)}$ $(\Theta|_M)$ is {the leading term\footnote{From $\Theta\in \Omega^p_{\mult}(\Gpd)$ we define $\theta:=\mathrm{pr}_{A^*\otimes (\wedge^{p-1}T^*M)} \Theta|_M\in \Gamma(A^*\otimes (\wedge^{p-1} T^*M))$, and call it \textbf{the leading term} of $\Theta$, which completely determines the restriction of  $\Theta$ on $M$; see \cite{CLL} for details.}} of $\Theta$ and 
$\iota_{u\otimes \gamma} \theta\in \Gamma(\wedge^{k-1} A\otimes \wedge^{l+p-1} T^*M)$ is defined in the same fashion as in \eqref{contraction}.  For   the   operator   $\mathcal{S}$, see \eqref{Eqn:leftliftingoperator}.  \end{itemize}
\end{lemma}

\begin{proof}   $\mathrm{(i)}$  Since $R\in \mathcal{T}^{k,l}_{\mult}(\Gpd)$  and $\Theta\in \Omega^p_{\mult}(\Gpd)$ are multiplicative, we know that the maps
\[\Theta^\sharp: \oplus^{p-1} T\Gpd\to T^*\Gpd,~\mbox{ and }\quad R:\oplus^k T^* \Gpd \oplus \oplus^{l} T\Gpd \to \mathbb{R}\] are groupoid morphisms.
For $(g,h)\in \Gpd^{(2)}$, $Y_i\in T_g\Gpd, Y'_i\in T_h \Gpd, \alpha_j\in T^*_g\Gpd$ and $\alpha'_j\in T^*_h\Gpd$ such that  $(Y_i,Y'_i)\in (T \Gpd)^{(2)},(\alpha_j,\alpha'_j)\in (T^*\Gpd)^{(2)}$ are composable, we have
\begin{eqnarray*}
&&\iota_R  \Theta(\alpha_1~\cdot~ \alpha'_1,~\cdots~,\alpha_{k-1}~\cdot~ \alpha'_{k-1},Y_1~\cdot~ Y'_1,~\cdots~, Y_{l+p-1}~\cdot~ Y'_{l+p-1})
\\ &=&\pm \sum_{\sigma}(-1)^\sigma R(\Theta^\sharp(Y_{\sigma_1}~\cdot~ Y'_{\sigma_1},~\cdots~, Y_{\sigma_{p-1}}~\cdot~ Y'_{\sigma_{p-1}}),\\
&&\qquad\qquad\qquad\qquad\qquad\quad \alpha_1~\cdot~ \alpha'_1,~\cdots~,\alpha_{k-1}~\cdot~ \alpha'_{k-1},Y_{\sigma_{p}}~\cdot~ Y'_{\sigma_p},~\cdots~, Y_{\sigma_{l+p-1}}~\cdot~ Y'_{\sigma_{l+p-1}})\\ &=&\pm
\sum_{\sigma}(-1)^\sigma R(\Theta^\sharp(Y_{\sigma_1},~\cdots~, Y_{\sigma_{p-1}})~\cdot~ \Theta^\sharp(Y'_{\sigma_1},~\cdots~,  Y'_{\sigma_{p-1}}), ~\cdots~,Y_{\sigma_{p}}~\cdot~ Y'_{\sigma_p},~\cdots~, Y_{\sigma_{l+p-1}}~\cdot~ Y'_{\sigma_{l+p-1}})\\ &=&\pm \sum_{\sigma}(-1)^\sigma \big(R(\Theta^\sharp(Y_{\sigma_1},~\cdots~, Y_{\sigma_{p-1}}), \alpha_1,~\cdots~,\alpha_{k-1},Y_{\sigma_{p}},~\cdots~, Y_{\sigma_{l+p-1}})\\ &&\qquad\quad+R(\Theta^\sharp(Y'_{\sigma_1},~\cdots~, Y'_{\sigma_{p-1}}), \alpha'_1,~\cdots~,\alpha'_{k-1},Y'_{\sigma_{p}},~\cdots~, Y'_{\sigma_{l+p-1}})\big)\\ &=&
\iota_R  \Theta(\alpha_1,~\cdots~,\alpha_{k-1},Y_1,~\cdots~, Y_{l+p-1})+
\iota_R  \Theta(\alpha'_1,~\cdots~,\alpha'_{k-1},Y'_1,~\cdots~ Y'_{l+p-1}).\end{eqnarray*}
This fact confirms that $\iota_R  \Theta$ is a multiplicative $(k-1,l+p-1)$-tensor field.

$\mathrm{(ii)}$ It suffices to check that 
\[(\iota_{\overleftarrow{u}\otimes s^*\gamma} \Theta) (\alpha_1,~\cdots~,\alpha_{k-1},Y_1,~\cdots~, Y_{l+p-1})=0,   \]
holds for $Y_1\in \ker s_{T\Gpd}=\ker s_*$ or $\alpha_1\in \ker s_{T^*\Gpd}$, and $Y_i\in \mathfrak{X}^1(\Gpd)$, $\alpha_j\in \Omega^1(\Gpd)$, $i,j\geqslant 2$. In fact, as $\alpha_1\in \ker s_{T^*\Gpd}$, we have 
\[\langle \overleftarrow{w},\alpha_1\rangle=\langle w,s_{T^*\Gpd}\alpha_1\rangle=0,\qquad \forall w\in \Gamma(A),\]
and thus
\begin{eqnarray*}
&&(\iota_{\overleftarrow{u}\otimes s^*\gamma} \Theta) (\alpha_1,~\cdots~,\alpha_{k-1},Y_1,~\cdots~, Y_{l+p-1})\\ &=&
\pm \sum_{\sigma} (-1)^\sigma \overleftarrow{u}(\Theta^\sharp(Y_{\sigma_1},~\cdots~,Y_{\sigma_{p-1}}),\alpha_1,~\cdots~,\alpha_{k-1})(s^*\gamma)(Y_{\sigma_{p}},~\cdots~,Y_{\sigma_{l+p-1}})=0.
\end{eqnarray*}
Meanwhile, for $Y_1\in \ker s_*$, one has
\begin{eqnarray*}
&&(\iota_{\overleftarrow{u}\otimes s^*\gamma} \Theta) (\alpha_1,~\cdots~,\alpha_{k-1},Y_1,~\cdots~, Y_{l+p-1})\\ &=&\pm 
\sum_{\tau} (-1)^\tau\overleftarrow{u}(\Theta^\sharp(Y_1,Y_{\tau_1},~\cdots~,Y_{\tau_{p-2}}),\alpha_1,~\cdots~,\alpha_{k-1})(s^*\gamma)(Y_{\tau_{p-1}},~\cdots~,Y_{\tau_{l+p-2}})\\ &=&\pm\sum_{\tau} (-1)^\tau u(s_{T^*\Gpd}\Theta^\sharp(Y_1,Y_{\tau_1},~\cdots~,Y_{\tau_{p-2}}),s_{T^*\Gpd}\alpha_1,~\cdots~,s_{T^*\Gpd}\alpha_{k-1})(s^*\gamma)(Y_{\tau_{p-1}},~\cdots~,Y_{\tau_{l+p-2}})\\ &=& \pm 
\sum_{\tau} (-1)^\tau u(\Theta^\sharp(s_*Y_1,s_*Y_{\tau_1},~\cdots~,s_*Y_{\tau_{p-2}}),s_{T^*\Gpd}\alpha_1,~\cdots~,s_{T^*\Gpd}\alpha_{k-1})(s^*\gamma)(Y_{\tau_{p-1}},~\cdots~,Y_{\tau_{l+p-2}})\\ &=&0,
\end{eqnarray*}
where in the second last equation we have used the identity $s_{T^*\Gpd} \circ \Theta^\sharp=\Theta^\sharp\circ s_*$ since $\Theta$ is multiplicative.  
\end{proof}

Applying Examples \ref{Example:NP} and \ref{Example:NP2} to the case of  a Lie groupoid $\Gpd$ with a  bivector field $P\in \mathfrak{X}^2(\Gpd)$, we obtain a cubic $L_\infty$-algebra on forms $\Omega^\bullet(\Gpd)$ and a curved DGLA on multivector fields $\mathfrak{X}^\bullet(\Gpd)$ of $\Gpd$. %, and a morphism between them. 
  Concerning the groupoid structure, it is certainly interesting to consider the case that   $P$ is a multiplicative bivector  field on $\Gpd$. Then we shall obtain a  sub
cubic $L_\infty$-algebra   and a sub curved DGLA. 
\begin{proposition}\label{Lin}
Let $\Gpd$ be a Lie groupoid, and $P$ a multiplicative bivector field on $\Gpd$. The following statements are true:
\begin{itemize}
\item[\rm (i)] The quadruple $(\Omega^\bullet_{\mult}(\Gpd)[1],d, [~\cdot~,~\cdot~]_P, [~\cdot~,~\cdot~,~\cdot~]_3)$ is a cubic $L_\infty$-algebra, where $d$ is the de Rham differential and $[~\cdot~,~\cdot~,~\cdot~]_3: \Omega^p_{\mult}(\Gpd)\wedge \Omega^q_{\mult}(\Gpd)\wedge \Omega^s_{\mult}(\Gpd)\to \Omega^{p+q+s-3}_{\mult}(\Gpd)$ is defined by \[[\Theta_1,\Theta_2,\Theta_3]_3=\iota_{(\iota_{(\iota_{\frac{1}{2}[P,P]} \Theta_1)} \Theta_2)} \Theta_3,\qquad \Theta_i\in \Omega^\bullet_{\mult}(\Gpd).\]
(For convention of the contraction $\iota$, see Equation \eqref{contraction}.)
\item [\rm (ii)] The quadruple  $(\mathfrak{X}^\bullet_{\mult}(\Gpd)[1],c,d_P=[P,~\cdot~],[~\cdot~,~\cdot~])$ is a   curved DGLA,  where  $c=\frac{1}{2}[P,P] \in \mathfrak{X}^3_{\mult}(\Gpd)$.
%\item[\rm  (iii)] ???!!! We have a morphism $(F_0,F_1,F_2): \Omega^\bullet_{\mult}(\Gpd)[1]\to \mathfrak{X}^\bullet_{\mult}(\Gpd)[1]$ of curved $L_\infty$-algebras, where $F_0=-P\in \mathfrak{X}^2_{\mult}(\Gpd)$, $F_1:=\wedge^\bullet P^\sharp: \Omega^\bullet_{\mult}(\Gpd)\to \mathfrak{X}^\bullet_{\mult}(\Gpd)$
%and \[F_2:\Omega^p_{\mult}(\Gpd)\wedge \Omega^q_{\mult}(\Gpd)\to \mathfrak{X}^{p+q-2}_{\mult}(\Gpd),\qquad (\Theta_1,\Theta_2)\mapsto \iota_{P}(\Theta_1\wedge \Theta_2),\]
%and all the higher maps $F_k$ vanish for $k\geqslant 3$.  
\end{itemize}
 \end{proposition}
 \begin{proof}   For (i),   we only need to show that   multiplicative forms are closed under the bracket  $[~\cdot~,~\cdot~]_P$ and the $3$-bracket $[~\cdot~,~\cdot~,~\cdot~]_3$. The former was proved in our previous work \cite{CLL}*{Theorem 4.14}. For the latter, since $[P,P]\in \mathfrak{X}^3_{\mult}(\Gpd)$ is  multiplicative, and by applying $\rm(i)$ of Lemma \ref{preparation} repeatedly, 
we see that \[[\Theta_1,\Theta_2,\Theta_3]_3\in \mathcal{T}_{\mult}^{(0,p+q+s-3)}(\Gpd)=\Omega^{p+q+s-3}_{\mult}(\Gpd).\] Thus  $\Omega^\bullet_{\mult}(\Gpd) $ is a sub cubic $L_\infty$-algebra in $ \Omega^\bullet(\Gpd)$. 

For (ii), it is well-known that multiplicative multivector fields are closed under the Schouten bracket and $P$ is multiplicative. So $\mathfrak{X}^\bullet_{\mult}(\Gpd) $ is a sub curved DGLA of $\mathfrak{X}^\bullet(\Gpd)$.
%For (iii), based on \cite[Lemma 4.17]{CLL} and Lemma \ref{preparation},  we see that $F_0,F_1,F_2$ are well-defined. Namely, they map multiplicative forms to multiplicative multivector fields. So the conclusion is clear.
 \end{proof}

Note that all structure maps in (i) are (multi-)derivations   in each argument. For this reason, we also call $(\Omega^\bullet_{\mult}(\Gpd)[1],d, [~\cdot~,~\cdot~]_P, [~\cdot~,~\cdot~,~\cdot~]_3)$ a derived Poisson algebra \cite{BCSX}.

\subsection{The graded weak Lie $2$-algebra arising from a quasi-Poisson groupoid} We are ready to state our second main result.
\begin{theorem}\label{Thm:secondmain}
Let $(\Gpd,P,\Phi)$ be a quasi-Poisson groupoid as in Definition \ref{DEF:ILX}. Then the following statements are true:
\begin{itemize}
\item[\rm (i)] The triple 
$\Omega^\bullet(M)[1]\xrightarrow{J} \Omega^\bullet_{\mult}(\Gpd)[1]$ is a graded weak Lie $2$-algebra, where $J$ is given by $\gamma\mapsto s^*\gamma-t^*\gamma$, 
the bracket on $\Omega^\bullet_{\mult}(\Gpd)$ is $[~\cdot~,~\cdot~]_P$, the action $\triangleright: \Omega^p_{\mult}(\Gpd)\times \Omega^q(M)\to \Omega^{p+q-1}(M)$  and the $3$-bracket $[~\cdot~,~\cdot~,~\cdot~]_3: \Omega^p_{\mult}(\Gpd)\wedge \Omega^q_{\mult}(\Gpd)\wedge \Omega^s_{\mult}(\Gpd)\to \Omega^{p+q+s-2}(M)$ are determined by 
\begin{eqnarray*}
s^*(\Theta\triangleright \gamma)&=&[\Theta,s^*\gamma]_P,\\ s^*[\Theta_1,\Theta_2,\Theta_3]_3&=&d\iota_{(\iota_{(\iota_{\overleftarrow{\Phi}}\Theta_1)}\Theta_2)} \Theta_3+\big(\iota_{(\iota_{(\iota_{\overleftarrow{\Phi}} \Theta_1)} \Theta_2)}d\Theta_3+c.p.),%\\ &=&d\iota_{\overleftarrow{\Phi}}(\Theta_1\wedge \Theta_2\wedge \Theta_3)+\big(\iota_{\overleftarrow{\Phi}(\Theta_1,\Theta_2,~\cdot~)} d\Theta_3+c.p.),
\end{eqnarray*}respectively, 
for $\Theta, \Theta_i\in \Omega^\bullet_{\mult}(\Gpd)$ and $\gamma\in \Omega^\bullet (M)$.
\item[\rm (ii)] The triple 
 $\Gamma(\wedge^\bullet A)[1]\xrightarrow{T} \mathfrak{X}^\bullet_{\mult}(\Gpd)[1]$ with $T(u)= \overleftarrow{u}-\overrightarrow{u}$ is a  graded Lie $2$-algebra with  the action $\triangleright: \mathfrak{X}^p_{\mult}(\Gpd)\times \Gamma(\wedge^q A)\to \Gamma(\wedge^{p+q-1} A)$   defined by 
 $\overleftarrow{X\triangleright u}=[X,\overleftarrow{u}]$.
 \item[\rm (iii)] There is a  morphism of graded weak Lie $2$-algebras
\begin{equation*}
\xymatrix{
 			\Omega^\bullet(M)[1]\ar@{^{}->}[d]_{J} \ar[r]^{\wedge^\bullet p^\sharp} & \Gamma(\wedge^\bullet A)[1] \ar@{^{}->}[d]^{T} \\
 			\Omega^\bullet _{\mult}(\Gpd)[1]\ar[r]^{\wedge^\bullet P^\sharp} \ar@{.>}[ru]^{\nu} & \mathfrak{X}^\bullet_{\mult}(\Gpd)[1]},
 	\end{equation*} formed by  $(\wedge^ \bullet P^\sharp,\wedge^\bullet p^\sharp,\nu)$, 
 	where  $p=\mathrm{pr}_{ TM\otimes A } (P|_M)\in \Gamma(TM\otimes A)$ and $\nu: \Omega^p_{\mult}(\Gpd) \wedge \Omega^q_{\mult}(\Gpd)\to \Gamma(\wedge^{p+q-1} A)$ is defined by
	\begin{eqnarray}\label{nnu}
	\nu(\Theta_1,\Theta_2)=-(\mathrm{id}\otimes (\wedge^{p+q-2} p^\sharp))(\iota_{\iota_{\Phi}\theta_1} \theta_2),
	\end{eqnarray}
	with $\theta_1=\mathrm{pr}_{ A^*\otimes (\wedge^{p-1}T^*M )} (\Theta_1|_M)\in \Gamma(A^*\otimes (\wedge^{p-1} T^*M))$ and $\theta_2$  defined similarly.  {The contraction in the right hand side of \eqref{nnu} is defined in the same manner as that of \eqref{contraction}.}
	\end{itemize}
\end{theorem}	
\begin{proof}  Statement $\mathrm{(ii)}$ is well-known (e.g.  see \cite{BCLX}). We only prove the other two.

	$\mathrm{(i)}$  We first show that $ [ \tobefilledin,\tobefilledin,\tobefilledin ]_3$ is well-defined. Namely, to every triple $(\Theta_1\in \Omega^{p}_{\mult}(\Gpd),\Theta_2\in \Omega^{q}_{\mult}(\Gpd),\Theta_3\in \Omega^{s}_{\mult}(\Gpd))$ there exists a unique element $\mu\in \Omega^{p+q+s-2}(M)$ such that
\begin{eqnarray}\label{hwell}
d\iota_{(\iota_{(\iota_{\overleftarrow{\Phi}}\Theta_1)}\Theta_2)} \Theta_3+\big(\iota_{(\iota_{(\iota_{\overleftarrow{\Phi}} \Theta_1)} \Theta_2)}d\Theta_3+c.p.)=s^*\mu.
\end{eqnarray}
In fact, we have $d\Theta_3\in \Omega^{s+1}_{\mult} (\Gpd)$ and $\overleftarrow{\Phi}(\Theta_1,\Theta_2,~\cdot~)= \overleftarrow{\Phi(\theta_1,\theta_2,~\cdot~)}$. Using
 $\mathrm{(ii)}$ of Lemma \ref{preparation} repeatedly and the fact that $s^*$ is injective, we obtain the $\mu$ in \eqref{hwell}.

  { Further,  we note that $s^*(\cdot \triangleright  \cdot)$ and $s^*[\cdot,\cdot,\cdot]_3$ are subject to the Leibniz rules, namely 
 \begin{eqnarray*}
 s^*((\Theta_1\wedge \Theta_2)\triangleright \gamma)&=&\Theta_1\wedge s^*(\Theta_2\triangleright \gamma)+(-1)^{|\Theta_2|(|\gamma|-1)} s^*(\Theta_1\triangleright \gamma)\wedge \Theta_2,\\
 s^*(\Theta\triangleright (\gamma_1\wedge \gamma_2))&=&s^*(\Theta\triangleright \gamma_1)\wedge (s^*\gamma_2)+(-1)^{(|\Theta|-1)|\gamma_1|}(s^*\gamma_1)\wedge s^*(\Theta\triangleright \gamma_2),\\ \mbox{ and }
 {s^*[\Theta_1\wedge \Theta_2,\Theta_3,\Theta_4]_3}&=&\Theta_1\wedge s^*[\Theta_2,\Theta_3,\Theta_4]_3+(-1)^{|\Theta_2|(|\Theta_3|+|\Theta_4|)}s^*[\Theta_1,\Theta_3,\Theta_4]_3\wedge \Theta_2.
 \end{eqnarray*}
% and the relation
% \[(s^*-t^*)[\Theta_1,\Theta_2,\Theta_3]_3=\iota_{\frac{1}{2}[P,P]}(\Theta_1\wedge \Theta_2 \wedge \Theta_3)+\big(\iota_{\frac{1}{2}[P,P](\Theta_1,\Theta_2,\cdot)}d\Theta_3+c.p.\big),\qquad \forall \Theta_i\in \Omega^1_{\mult}(\Gpd).\]
  Based on Theorem \ref{Thm:case1}, the Leibniz rules of $s^*(\cdot \triangleright  \cdot)$ and  $s^*[\cdot,\cdot,\cdot]_3$, and the fact that $s^*,t^*$ are injective maps, we can verify the desired graded weak Lie $2$-algebra. }

$\mathrm{(iii)}$ In what follows,  $\wedge^\bullet P^\sharp$ is abbreviated to $P^\sharp$, and similarly,  $\wedge^\bullet p^\sharp$ to  $p^\sharp$.     Formula \eqref{Eqt:PsharpwrtPbracket} can be extended by the Leibniz rule to higher degree differential forms: 
 \begin{eqnarray*}
P^\sharp[\Theta_1,\Theta_2]_P-[P^\sharp \Theta_1,P^\sharp \Theta_2]=(\mathrm{id}\otimes P^\sharp)(\iota_{\iota_{\frac{1}{2}[P,P]}\Theta_1} \Theta_2)
\end{eqnarray*}
for all $\Theta_1, \Theta_2\in \Omega^\bullet_{\mult}(\Gpd)$.
Using $\frac{1}{2}[P,P]=\overrightarrow{\Phi}-\overleftarrow{\Phi}$, $\rm(ii)$ of Lemma \ref{preparation}, and the relations
\[(\mathrm{id}\otimes P^\sharp)(\overleftarrow{v}\otimes s^*\mu)=\overleftarrow{v\otimes p^\sharp(\mu)},\qquad (\mathrm{id}\otimes P^\sharp)(\overrightarrow{v}\otimes t^*\mu)=\overrightarrow{v\otimes p^\sharp(\mu)}\qquad \forall v\in \Gamma(\wedge^\bullet A),\mu\in \Omega^\bullet(M),\]
we further obtain
\begin{eqnarray*}
P^\sharp[\Theta_1,\Theta_2]_P-[P^\sharp \Theta_1,P^\sharp \Theta_2]&=&(\mathrm{id}\otimes P^\sharp)(\iota_{\iota_{\overrightarrow{\Phi}-\overleftarrow{\Phi}} \Theta_1} \Theta_2)
\\ &=&\overrightarrow{(\mathrm{id}\otimes p^\sharp)(\iota_{\iota_\Phi \theta_1}  \theta_2)}-\overleftarrow{(\mathrm{id}\otimes p^\sharp)(\iota_{\iota_\Phi \theta_1}  \theta_2)}\\ &=&T(\nu(\Theta_1,\Theta_2)).
\end{eqnarray*}
Taking advantage of these relationships, what remains is some direct verification of the said morphism of graded weak Lie $2$-algebras. We omit the details.
\end{proof}	
\begin{remark}
	
 {If the quasi-Poisson groupoid $(\Gpd,P,\Phi)$ degenerates to a Poisson groupoid, namely $\Phi=0$, then what we obtain from Theorem \ref{Thm:secondmain} are two \emph{strict} graded Lie $2$-algebras together with a \emph{strict} graded Lie $2$-algebra homomorphism between them, i.e., those given by \cite{CLL}*{Theorem 4.14}.}
 
\end{remark}

\subsection{The special case of quasi-Poisson Lie groups}
 In this part, we study a relatively easy situation of quasi-Poisson groupoids, called quasi-Poisson Lie groups,  i.e., when the base manifold $M$ of the groupoid $\Gpd$ in question is a single point. For clarity of notations, we use $G$ to denote such a group instead of $\Gpd$, and the Lie algebra of $G$ is denoted by $\g=T_{e}G$. 
 
\begin{corollary}\label{Prop:quasiPoissonLieGroup}
Let $(G,P,\Phi)$ be a quasi-Poisson Lie group. The following statements are true.
\begin{itemize}
\item[\rm  (1)] The Lie algebra $(\Omega^1_{\mult}(G),[~\cdot~,~\cdot~]_P)$   {  is isomorphic to the Lie algebra $((\g^*)^G,[\cdot,\cdot]_{\g^*})$} ($G$-invariant $1$-forms); 
\item[\rm  (2)] The triple   $\g\xrightarrow{T=\overleftarrow{(~\cdot~)}-\overrightarrow{(~\cdot~})}\mathfrak{X}^1_{\mult}(G)$ %with the Lie bracket on $\g$ and Schouten bracket on $\mathfrak{X}^1_{\mult}(G)$ 
constitutes a Lie $2$-algebra;
\item[\rm (3)] There is a weak Lie $2$-algebra morphism formed by $(P^\sharp,0,\nu)$:
\begin{equation*}
 		\xymatrix{
 			0\ar@{^{}->}[d]_{0} \ar[r]^{0} & \g \ar@{^{}->}[d]^{T} \\
 			\Omega^1_{\mult}(G)\ar[r]^{P^\sharp}  \ar@{.>}[ru]^{\nu} & \mathfrak{X}^1_{\mult}(G)},
 	\end{equation*}
 	where  $\nu: \wedge^2\Omega^1_{\mult}(G)  \to \g$ is defined by
	\[\nu(\Theta_1,\Theta_2)=-\Phi(\theta_1,\theta_2,~\cdot~),\]
	where $\theta_i\in \g^*$ is determined by $R_g^*\Theta_{i}(g)=\theta_i$.
\end{itemize}
	\end{corollary}
	\begin{proof} {The isomorphism between $\Omega^1_{\mult}(G)$ and $(\g^*)^G$  sends $\Theta\in \Omega^1_{\mult}(G)$ to $ \theta\in(\g^*)^G$ given by $ \theta:=R_g^*\Theta_g=L_g^*\Theta_g$, for any $g\in G$. This is due to $\Theta$ being multiplicative. Of course, one could simply set $\theta=\Theta|_e$.
			
			By \cite{CLL}*{Example 4.2},   $[\Theta_1,\Theta_2]_P$ is sent to $[\theta_1,\theta_2]_*$, which  proves Statement (1). 
	Statements (2) and (3) are direct consequences of Theorem \ref{Thm:case1} and Proposition \ref{Lie2homo}.	}
	\end{proof}
 
 \begin{remark}We claimed that $ \Omega^1_{\mult}(G)$ is a Lie algebra whose structure map is the $P$-bracket $[~\cdot~,~\cdot~]_P$. However, be aware that the large space $\Omega^1(G)$  is not a Lie algebra with respect to $[~\cdot~,~\cdot~]_P$. 
 	Please also compare with the previous result (Example \ref{Example:NP2}) that  {$\Omega^\bullet(G)$}   carries a cubic $L_\infty$-algebra structure. 
 \end{remark}

\begin{example}\label{abeliangroup}
Let $V$ be a finite dimensional vector space. Viewing it as an abelian group, we have the identifications
\[\mathfrak{X}_{\mult}^k(V)=\Hom(V, \wedge^k V) \qquad (\forall k\geqslant 1),\qquad \Omega_{\mult}^1(V)=V^*,\qquad\mbox{and ~}~ \Omega_{\mult}^l(V)=0 \qquad (\forall l\geqslant 2).\]
 
Now consider a Lie algebra $\g$   and the abelian group structure on the vector space $V:=\g^*$. We have a  Poisson structure on $V$ determined by $\{x,y\}_{P}=[x,y]_\g$, for all $x,y\in \g$ seen as linear functions on $\g^*$.  This Poisson structure is  widely known as the Kirillov-Kostant-Souriau (KKS) Poisson structure. It turns out that $(\g^*,P)$ forms a   Poisson Lie group which is particularly called the linear Poisson group associated to the given Lie algebra $\g$. %\textcolor{red}{a little too much}

Indeed,   
the Lie algebra $(\Omega^1_{\mult}(\g^*),[~\cdot~,~\cdot~]_{P})$ coincides with the Lie algebra $\g$; and the   Lie  $2$-algebra
associated with multiplicative vector field is of the form $\g^*\xrightarrow{0}\End(\g^*)$. Moreover, we have a Lie  $2$-algebra morphism
\begin{equation}\label{coad}
 		\xymatrix{
 			0\ar@{^{}->}[d]_{0} \ar[rr]^{0} && \g^* \ar@{^{}->}[d]^{T=0} \\
 			\g\ar[rr]^{P^\sharp} && \End(\g^*)} 
 	\end{equation}
 	where  $P^\sharp: \g\to \End(\g^*)$ is actually
	\[P^\sharp(x)=\ad^*_x,\qquad \forall x\in \g.\]
\end{example}

\section{The linear quasi-Poisson  $2$-group arising from a Lie $2$-algebra}
\label{Sec:linearQP2GP}
%Let $(\thetaalgebradegone\xrightarrow{d} \galgebradegzero,l_2,l_3)$ be a Lie  $2$-algebra as in \ref{??}. USE $\thetaalgebradegone$ and $\galgebradegzero $!!! \textcolor{red}{how?} 
%A $2$-term complex of vector spaces give an action Lie groupoid. 

This section focuses on linear quasi-Poisson $2$-groups, which serve as simple models for examination. These types of groups inherently involve an action Lie groupoid, leading us to first analyze multiplicative structures on general action Lie groupoids.

 %\textcolor{blue}{need to change}

\subsection{Multiplicative forms  and vector fields on an action Lie groupoid}
We recall the concept of an action Lie groupoid. For references, see \cites{Mackenzie}.  
Let $G$ be a Lie group, $M$ a manifold,  and $\sigma: G\times M\to M$  a Lie group action.
We adopt a particular notation $G\triangleright M\rightrightarrows M$ to denote   
the action Lie groupoid arising from the action  $\sigma$: the underlying space of $G\triangleright M$ is the Cartesian product  $G\times M$, the base manifold is $M$,    
the source and target maps are given, respectively, by
\[s(g,m)=m, \quad\mbox{ and }~ \qquad t(g,m)=gm, \]
and the multiplication in $G\triangleright M$ is computed by
$$(h,gm)(g,m)=(hg,m)$$
for all $g,h\in G$, $m\in M$. Here and in the sequel, for simplicity, $gm$ stands for the value   $\sigma(g,m)\in M$ of the group action. Also, we will use	  $\sigma_m:G\to M$ and $\sigma_g: M\to M$ to denote, respectively, the maps $g\mapsto gm$ and $m\mapsto gm$.

Next, we describe   multiplicative $1$-forms on action Lie groupoids.  
As before, we denote by $\g:=T_eG$  the tangent Lie algebra of $G$ and   by $\rho:\g\to \mathfrak{X}^1(M)$ the map given by $\rho(x)|_m:=(\sigma_m)_{*e}(x)$. 

\begin{proposition}\label{formaction}
	Let $\Theta\in \Omega^1(G\times M)$ be a $1$-form on the action Lie groupoid $G\triangleright M\rightrightarrows M$.  {Suppose that (under  the natural decomposition $T^*(G\triangleright M)=T^*G\times T^*M$)  }   the two components of $\Theta$ are given by $\Theta^G:G\times M\to T^*G$ and $\Theta^M:G\times M\to T^*M$, respectively. Then $\Theta$ is multiplicative if and only if there exists a map $\mu:M\to \g^*$ such that for any $g,h\in G$ and $m\in M$ the following equalities hold: 
	\begin{eqnarray*}
		\Ad_h^*\mu_{hm}-\mu_m&=&\rho^*\Theta^M_{(h,m)};\\
		\Theta^G_{(g,m)}&=&R_{g^{-1}}^* \mu_{gm};\\
		\Theta^M(hg)&=&\sigma_g^*\big(\Theta^M(h)\big)+\Theta^M(g).
	\end{eqnarray*}
	The last equation is indeed saying that $\Theta^M$ is  a $1$-cocycle when it is regarded as a map $G\to \Omega^1(M)$  {(with respect to the obvious   $G$-module $\Omega^1(M)$)}.

\end{proposition}

\begin{proof}  By definition, a  $1$-form $\Theta\in \Omega^1(G\triangleright M)$ is multiplicative if and only if 
	\begin{eqnarray}\label{multformac}
		\Theta(Y_h ~\cdot~ X_g, U_m)=\Theta(Y_h, X_g~\cdot~ U_m)+\Theta(X_g,U_m),\qquad \forall Y_h\in T_h G,X_g\in T_g G, U_m\in T_m M.
	\end{eqnarray}
	Identifying $T_{(g,m)}^* (G\times M)$ with $T_g^* G\times T_m^* M$, Equation \eqref{multformac} amounts to three conditions:
	  	\begin{eqnarray}
	\label{Gpart2'}\Theta^G_{(hg,m)}(R_{g*}Y_h)&=&\Theta^G_{(h,gm)}(Y_h),\\
	\label{Gpart'}\Theta^G_{(hg,m)}(L_{h_*}X_g)&=&\Theta^M_{(h,gm)}(\sigma_{m*} X_g)+\Theta^G_{(g,m)}(X_g),\\
	\label{Mpart'}\Theta^M_{(hg,m)}(U_m)&=&\Theta^M_{(h,gm)}(\sigma_{g*}U_m)+\Theta^M_{(g,m)}(U_m),
	\end{eqnarray}
	where $\sigma_m:G\to M$ and $\sigma_g: M\to M$ are, respectively, the maps $g\mapsto gm$ and $m\mapsto gm$.

	Suppose that $\Theta$ is multiplicative. Define $\mu_{m}=\Theta^G_{(e,m)}$ for all $m\in M$. Then by taking $h=e$ in \eqref{Gpart2'}, we have
	\[R_{g}^* \Theta^G_{(g,m)}=\Theta^G_{(e,gm)} =\mu_{gm}.\]
	Substituting this relation into \eqref{Gpart'}, we further obtain
	\[L_h^*R_{{(hg)}^{-1}}^*\mu_{hgm}=\sigma_m^*\Theta^M_{(h,gm)}+R_{g^{-1}}^*\mu_{gm},\]
	which implies that 
	\[\Ad_{h}^*\mu_{hm}-\mu_m=\rho^*\Theta^M_{(h,m)}.\]
	The identity $\Theta^M(hg)=\sigma_g^*\big(\Theta^M(h)\big)+\Theta^M(g)$ is just a variation of Equation \eqref{Mpart'}.
	The converse statement  is straightforward to verify.\end{proof}

Maintaining  the above assumptions, we have two other statements.

\begin{corollary}\label{transitiveform}
	\begin{itemize}
		\item[\rm (i)] Suppose that the action $\sigma$ is transitive (i.e., for any $m\in M$, the map $\sigma(-,m): G\to M$ is surjective). Then $\Theta=(\Theta^G, \Theta^M)$ is multiplicative if and only if  it is determined by a smooth map $\mu:M\to \g^*$  satisfying  $\Ad_h^*\mu_{hm}-\mu_m\in (\ker \rho)^\perp$ (for all $h\in G$)  such that  
		\[\Theta^G_{(g,m)}:=R_{g^{-1}}^* \mu_{gm}, \qquad \rho^*\Theta^M_{(h,m)}:=\Ad_h^*\mu_{hm}-\mu_m,\]
		for all $g,h\in G$ and $m\in M$.
		\item[\rm (ii)] Suppose that the action $\sigma$ is trivial,  i.e., $\sigma(g,m)=m$ for all $g\in G,m\in M$. Then $\Theta=(\Theta^G, \Theta^M)$ is multiplicative if and only if\begin{enumerate}
			\item 
			$\iota_m^*(\Theta^G)\in \Omega^1_{\mult}(G)$ for all $m\in M$, where $\iota_m: G\to G\times M$ is the embedding map $g\mapsto (g,m)$;
			\item  the map $\Theta^M:G\to \Omega^1(M)$  satisfies $\Theta^M(hg)=\Theta^M(h)+\Theta^M(g), \forall g,h\in G.$
		\end{enumerate}
	\end{itemize}
	
\end{corollary}
\begin{proof} (i) As   $\rho$ is transitive,   $\rho^*$ is injective. Then for any $\mu\in C^\infty(M,\g^*)$,    if it holds that $\Ad_h^*\mu-\mu\in (\ker \rho)^\perp=\mathrm{Im} \rho^*$, there exists a unique $\Theta_{(h,~\cdot~)}^M\in \Omega^1(M)$ such that $\Ad_h^*\mu-\mu=\rho^*\Theta^M_{(h,~\cdot~)}$. It is straightforward to check that $\Theta^G_{(g,m)}:=R_{g^{-1}}^*\mu_{gm}$ and $\Theta^M$ together define a multiplicative $1$-form by  the relation $\rho\circ \Ad_g=\sigma_{g*}\circ \rho: \g\to \mathfrak{X}^1(M)$.
	The converse is similar to check.
	
	(ii) If the action is trivial, then $\rho=\sigma_{m_*}=0$ and $\sigma_{g*}=\mathrm{\id}$. The conclusion  direct follows by   
	Proposition \ref{formaction}.
\end{proof}

\begin{example} %Again, consider the action Lie groupoid  $ G\triangleright M\rightrightarrows M$  arising from an action $\sigma:G\times M\to M$.
	\begin{itemize}
		\item[\rm  (1)]
		Given   $\gamma\in \Omega^1(M)$, %define $\mu: M\to \g^*$ by $\mu_m:=\rho^*\gamma_m$. Then 
		by setting 
		\[\Theta^G_{(h,m)}=\sigma_m^*\gamma_{hm},\qquad \Theta^M_{(h,m)}=\sigma_h^*\gamma_{hm}-\gamma_m,
		\quad \forall h\in G, m\in M\]
		we obtain a multiplicative $1$-form on the action Lie groupoid $G\triangleright M$. In fact, the said $1$-form is $\Theta=t^*\gamma-s^*\gamma$.
		%For any $\gamma\in \Omega^k(M)$, the $k$-form
		%\[s^*\gamma-t^*\gamma=\mathrm{pr}_M^*\gamma-\sigma^*\gamma\in \Omega^k(G\times M)\]
		%is a multiplicative $k$-form on $\Gpd$, where $\mathrm{pr}_M: G\times M\to M$ is the projection. In fact,
		%\[(s^*\gamma-t^*\gamma)_{(g,m)}=\gamma_m-\sum_{s=0}^k(\sigma_m^{*s}\otimes \sigma_g^{*k-s})\gamma_{gm}\in \oplus_{s=0}^k \big( \wedge^s T^*_g G\otimes (\wedge^{k-s} T^*_m M)\big).\]
		%Moreover, its $\wedge^k T^*G$ and $\wedge^k T^*M$-parts are respectively
		%\[(s^*\gamma-t^*\gamma)^G_{(g,m)}=-\sigma_m^*\gamma_{gm},\qquad (s^*\gamma-t^*\gamma)^M_{(g,m)}=\gamma_m-\sigma_g^*\gamma_{gm},\]
		%and its $\wedge^p T^*G\otimes \wedge^{k-p} T^*M$-part is ($1\leqslant p<k$)
		%\[(s^*\gamma-t^*\gamma)^{p,k-p}_{(g,m)}=-(\wedge^p \sigma_m\otimes (\wedge^{k-p}\sigma_g^*))\gamma_{gm}.\]
		\item[\rm  (2)]Let  $\alpha\in \Omega^1(G)$ be a multiplicative $1$-form  on the Lie group $G$. Then $\mathrm{pr}_G^*\alpha\in \Omega^1(G\times M)$ is a multiplicative $1$-form on the action Lie groupoid on $G\triangleright M$, where $\mathrm{pr}_G: G\times M\to G$ is the projection.
	\end{itemize}
\end{example}

 {Now we turn to  multiplicative vector fields on the action Lie groupoid $G\triangleright M\rightrightarrows M$. 
Consider a general vector field of the form $$X=(X^G,X^M)\in \mathfrak{X}^1(G\triangleright M) ,$$ where $X^G: G\times M\to TG$ and $X^M: G\times M\to TM$ are determined by the natural identification of $T(G\triangleright M)$ with $TG\times TM$. Then $X$ is multiplicative if and only if for all $g,h\in G, m\in M$, the following equation holds:
\begin{equation}
\label{Eqt:temp1}
\big(X^G_{(hg,m)},X^M_{(hg,m)}\big)=\big(X^G_{(h,gm)},X^M_{(h,gm)}\big)*\big(X^G_{(g,m)},X^M_{(g,m)}\big),\end{equation}
where $*$ is the multiplication on the tangent groupoid $TG\triangleright TM\rightrightarrows TM$, which is again an  action Lie groupoid (arising from the induced action $\sigma_*:TG\times TM\to TM$). Below we give more information about $X^G$ and $X^M$.}

\begin{proposition}\label{vvector fields}
Let $X=(X^G,X^M)\in \mathfrak{X}^1(G\times M)$ be a vector field on the action Lie groupoid $G\triangleright M\rightrightarrows M$. Then $X$ is multiplicative if and only if for all $g \in G$ and $ m\in M$, we have
\begin{itemize}
	\item $X^M(g,m)$   does not depend on $g$, and thus we treat   $X^M  \in   \mathfrak{X}^1(M)$;
	\item $\sigma_{m*} X_{(g,m)}^G =  X^M_{gm}-\sigma_{g*}X_m^M$;
	\item $X_{(hg,m)}^G = L_{h*} X_{(g,m)}^G+R_{g*} X^G_{(h,gm)}$ for all $h\in G$.
\end{itemize} 
 
\end{proposition}
\begin{proof}Equivalently, we can unravel Equation 
	\eqref{Eqt:temp1} as follows: 
	\begin{eqnarray*}
	\label{condition1}X_{(g,m)}^M&=&X_{(hg,m)}^M;\\ 
	\label{condition2}X_{(g,m)}^G\triangleright X_{(g,m)}^M&=&X_{(h,gm)}^M;\\ 
	\label{condition3}X_{(h,gm)}^G~\cdot~ X_{(g,m)}^G&=&X_{(hg,m)}^G.
	\end{eqnarray*}
	Now the statement follows directly from the above three conditions.% \eqref{condition1}-\eqref{condition3}. 
\end{proof}

\begin{corollary}\label{transitivevec}
	Maintaining the  assumptions as in Proposition \ref{vvector fields}, the following statements are true.
	\begin{itemize}
		\item[\rm (i)] If the action $\sigma  :G\times M\to M$   is transitive, then we have a natural isomorphism 
		\[\mathfrak{X}^1_{\mult}(G\triangleright M)\cong\mathfrak{X}^1(M)\oplus \Gamma_{\mult}(\sigma^*\ker \rho),\qquad (X^G,X^M)\mapsto (X^M,\gamma),\]
		such that \begin{equation}\label{Eqt:XGexpressed} X^G_{(g,m)}=R_{g*}u_{gm}-L_{g*}{u_m}+R_{g*}\mathcal{\gamma}(g,m)\end{equation}
		where $u\in \Gamma((\ker \rho)^0)$  is the unique element satisfying  $X^M=\rho(u)$, $\Gamma_{\mult}(\sigma^*\ker \rho)\subset \Gamma(\sigma^*\ker \rho)$ is composed of elements $\gamma: G\times M\to  \sigma^*\ker \rho$ satisfying  
	\begin{eqnarray*}\label{gamma}
		\gamma(hg,m)=\Ad_{h}\gamma(g,m)+\gamma(h,gm),\qquad \forall g,h\in G, m\in M,
	\end{eqnarray*}
and we fix a decomposition $M\times \g=\ker \rho\oplus (\ker \rho)^0$.
				\item[\rm (ii)] If the action is trivial, then $X=(X^G, X^M)$ is multiplicative if and only if   $X^M\in \mathfrak{X}^1(M)$ and 
		$X^G_{(~\cdot~,m)}\in \mathfrak{X}^1(G)$ for any $m\in M$ is a multiplicative vector field on the Lie group $G$.
	\end{itemize}
\end{corollary}
\begin{proof} The proof of Statement (ii) is easy and skipped. We only show (i). 
	
	Let $X=(X^G,X^M)$ be a multiplicative vector field. If $\rho$ is transitive, then there exists a unique element $u\in \Gamma((\ker \rho)^0)$ such that $\rho(u)=X^M$. Due to the second condition in Proposition \ref{vvector fields}, we have
	\begin{eqnarray*}
		\rho_{gm}(R_{g*}^{-1} X^G_{(g,m)})&=& \sigma_{m*}R_{g^*} R_{g^*}^{-1} X^G_{(g,m)}=\rho_{gm} ( u_{gm})-\sigma_{g*}\rho_{m}(u_m)\\ &=&\rho_{gm}(u_{gm}-\Ad_{g} u_m),
	\end{eqnarray*}
	where we have used the fact $\sigma_g\circ \rho_m=\rho_{gm}\circ \Ad_g$.
	So there exists $\mathcal{F}\in C^\infty(G\times M,\ker \rho)$ such that $X^G $ is expressed as in \eqref{Eqt:XGexpressed}. The converse fact follows in the same way. 
\end{proof}

\begin{example} 
	\begin{itemize}
		\item[\rm  (1)]
		For any $fu\in C^\infty(M)\otimes \g$, the elements  \[X^G=(\mathrm{pr_M}^*f)\overleftarrow{u}-(\sigma^*f)\overrightarrow{u},\qquad X^M=-f\rho(u),\]
		define  a multiplicative vector field on $G\triangleright M$. In fact, it is identically  $X=\overleftarrow{fu}-\overrightarrow{fu}$;
		\item[\rm  (2)]
		Let  $Y\in \mathfrak{X}^1(G)$ be a multiplicative vector field  on the Lie group $G$ satisfying that $\sigma_{m*} Y=0$ for all $m\in M$. Then $\tilde{Y}\in \mathfrak{X}^1(G\times M)$ defined by $\tilde{Y}_{(g,m)}=Y_g$  is a multiplicative vector field on  $G\triangleright M$.\end{itemize}
\end{example}

\subsection{Linear action groupoids and quasi-Poisson $2$-groups}
\subsubsection{Multiplicative forms  and vector fields on a linear action groupoid}

Given a linear map of vector spaces $\thetaalgebradegone\xrightarrow{d} \galgebradegzero$, we denote by $d^T: \galgebradegzero^*\to \thetaalgebradegone^*$ the dual map determined by
$$(d^Tg)(u)= -g(du), \quad\forall g\in \galgebradegzero^*,u\in \thetaalgebradegone.$$

%the action Lie groupoid $\galgebradegzero^*\triangleright \thetaalgebradegone^*\rightrightarrows \thetaalgebradegone^*$.
%Regarding the abelian group $G=\galgebradegzero^*$ and base manifold $M=\thetaalgebradegone^*$,  $ \galgebradegzero^*\triangleright \thetaalgebradegone^*\rightrightarrows \thetaalgebradegone^*$ is an action Lie groupoid and further a quasi-Poisson Lie groupoid. 
There is an associated action Lie groupoid   $\galgebradegzero^*\triangleright \thetaalgebradegone^*\rightrightarrows \thetaalgebradegone^*$. Here $\galgebradegzero^*\triangleright \thetaalgebradegone^*$ as a set is the direct product $\galgebradegzero^*\times \thetaalgebradegone^*$. The source map is given by $s:$ $(g,m)\mapsto m$, and the target map $t$ sends $(g,m)$ to $d^T g+m$, for all $(g,m)\in \galgebradegzero^*\triangleright \thetaalgebradegone^*$. For simplicity, we will write $gm$ for $d^T g+m$ from now on.

The groupoid multiplication in $\galgebradegzero^*\triangleright \thetaalgebradegone^*$ is also easy:
\begin{equation}
\label{Eqt:groupoidmultlinearLie2}
(h,gm)(g,m)=(h+g,m), \quad h,g\in \galgebradegzero^*, m\in \thetaalgebradegone^*.\end{equation}

%\textcolor{red}{We prove it directly, not using results in the general setting.}

  First, we  characterize   multiplicative $1$-forms on the   Lie groupoid $\galgebradegzero^*\triangleright \thetaalgebradegone^*$.

\begin{proposition}\label{Prop:easyexampleprop1} Fix a decomposition  $\galgebradegzero=\mathrm{Im} d\oplus (\mathrm{coker}d)$. 	We have an isomorphism
	\[\Omega^1_{\mult}(\galgebradegzero^*\triangleright \thetaalgebradegone^*)\cong C^\infty(\thetaalgebradegone^*,\mathrm{Im} d) \oplus C^\infty(\thetaalgebradegone^*,\mathrm{coker} d)^{\galgebradegzero^*} \oplus C^\infty_{\mult}(\galgebradegzero^*\triangleright \thetaalgebradegone^*,\ker d),\]
	where $C^\infty(\thetaalgebradegone^*,\mathrm{coker} d)^{\galgebradegzero^*}$ stands for $\mathrm{coker} d$-valued functions $f$ on $\thetaalgebradegone^*$ satisfying $f(gm)=f(m)$  and $C^\infty_{\mult}(\galgebradegzero^*\triangleright \thetaalgebradegone^*,\ker d)$ is the space of $\ker d$-valued multiplicative functions $\beta$ on  $ \galgebradegzero^*\triangleright \thetaalgebradegone^*$, i.e., they satisfy 
	\[\beta(h+g,m)=\beta(h,gm)+\beta(g,m).\]
	\end{proposition}
	\begin{proof}%This result can be seen as a consequence of Corollary \ref{Cor:regularcase}. We reprove it here by using the local coordinates.
	 Let us take a basis of $\thetaalgebradegone$: $$\{u_1,~\cdots~,u_r,u_{r+1},~\cdots~ u_q\}$$  such that $ du_1,~\cdots~, du_r   $ are linearly independent in $\galgebradegzero$ and $du_{r+1}=\cdots=du_q=0$ where $  q=\mathrm{dim} \thetaalgebradegone$. 
Then $\mathrm{Im} d$ is spanned by $du_i$ ($1\leqslant i \leqslant r$). 
Take the dual basis $$\{u^1,~\cdots~,u^r,u^{r+1},~\cdots~,u^q\}$$ of $\thetaalgebradegone^*$ and extend $\{du_1,~\cdots~, du_r\}$ to a 
basis of $\galgebradegzero$: $$\{x_1:=du_1,~\cdots~, x_r:=du_r, x_{r+1},~\cdots~ x_p\}.$$ Suppose that the corresponding dual basis of $\galgebradegzero^*$ is $$\{x^1,~\cdots~, x^r,x^{r+1},~\cdots~, x^p\}.$$ Here $p=\mathrm{dim} \galgebradegzero$. One can check that $d^Tx^i=-u^i$ for all $i=1,~\cdots~,r$. 
Then a $1$-form $\Theta=(\Theta^{\galgebradegzero^*},{\thetaalgebradegone^*})\in \Omega^1(\galgebradegzero^*\triangleright \thetaalgebradegone^*)$ takes the form
\begin{eqnarray*}
\Theta_{(g,m)}^{\galgebradegzero^*}&=&\sum_{i=1}^r A_i(g,m) du_i+\sum_{j=r+1}^p B_j(g,m) x_j,\\
	\Theta_{(g,m)}^{\thetaalgebradegone^*}&=&\sum_{i=1}^r C_i(g,m)u_i+\sum_{k=r+1}^q\beta_k(g,m) u_k,
\end{eqnarray*}
where $A_i,B_j,C_i,\beta_k\in C^\infty(\galgebradegzero^*\times  \thetaalgebradegone^*)$.

Recall Proposition \ref{formaction} where we considered  multiplicative $1$-forms on a general action Lie groupoid $G\triangleright M\rightrightarrows M$. 	 
For our case, we have $G= \galgebradegzero^*, M=\thetaalgebradegone^*$ and $R_{g*}=L_{h*}=\mathrm{id}, \sigma_{m*}=d^T, \sigma_{g*}=\mathrm{id}$. Applying \eqref{Gpart2'} to our $\Theta^{\galgebradegzero^*}$, we obtain 
\[A_i(hg,m)=A_i(h,gm),\qquad B_j(hg,m)=B_j(h,gm),\]
which implies that \[A_i(g,m)=A_i(0, gm)=:\mu_i(gm),\quad B_j(g,m)=B_j(0,gm)=:\alpha_j(gm), \qquad \forall \mu_i,\alpha_j\in C^\infty(\thetaalgebradegone^*).\] Then applying \eqref{Gpart'} to $(\Theta^{\galgebradegzero^*},
\Theta^{\thetaalgebradegone^*})$, we find
\[A_i(hg,m)=-C_i(h,gm)+A_i(g,m),\qquad B_j(hg,m)=B_j(g,m),\]
which further implies   
\[C_i(h,m)=A_i(0,m)-A_i(h,m)=\mu_i(m)-\mu_i(hm),\qquad \alpha_j(gm)=\alpha_j(m),\]
and thus   $\alpha_j\in C^\infty(\thetaalgebradegone^*)^{\galgebradegzero^*}$. Finally, applying \eqref{Mpart'} to $\Theta^{\thetaalgebradegone^*}$, we have
\[C_i(hg,m)=C_i(h,gm)+C_i(g,m),\qquad \beta_k(hg,m)=\beta_k(h,gm)+\beta_k(g,m).\]
Note that $C_i$ which is determined by $\mu_i$ automatically satisfies the first equation. So we have $\beta_k\in C^\infty_{\mult}(\galgebradegzero^*\triangleright \thetaalgebradegone^*)$. In summary, we have
\[A_i(g,m)=\mu_i(gm),\quad B_j(g,m)=\alpha_j(m),\qquad C_i(g,m)=\mu_i(m)-\mu_i(gm),\quad \beta_k\in C^\infty_{\mult}(\galgebradegzero^*\triangleright \thetaalgebradegone^*),\]
where $\mu_i,\alpha_j\in C^\infty(\thetaalgebradegone^*)$ and $\alpha_j\in C^\infty(\thetaalgebradegone^*)^{\galgebradegzero^*}$. 
Hence,  a $1$-form $\Theta=(\Theta^{\galgebradegzero^*},\Theta^{\thetaalgebradegone^*})\in \Omega^1(\galgebradegzero^*\triangleright \thetaalgebradegone^*)$ is multiplicative  if and only if it can be expressed in the form 
\begin{eqnarray*}
	\Theta_{(g,m)}^{\galgebradegzero^*}&=&\sum_{i=1}^r \mu_i(gm)du_i+\sum_{j=r+1}^p \alpha_j(m) x_j,\\
	\Theta_{(g,m)}^{\thetaalgebradegone^*}&=&\sum_{i=1}^r \big(\mu_i(m)-\mu_i(gm)\big)u_i+\sum_{k=r+1}^q\beta_k(g,m) u_k,
\end{eqnarray*}
where $\mu_i\in C^\infty(\thetaalgebradegone^*)$ and $\alpha_j\in C^\infty(\thetaalgebradegone^*)^{\galgebradegzero^*}$  satisfy $\alpha_j(gm)=\alpha_j(m)$, and $\beta_k\in C^\infty_{\mult}(\galgebradegzero^*\triangleright \thetaalgebradegone^*)$ are multiplicative functions on $\Gpd$. This completes the proof.
\end{proof}

 {Indeed, the    $C^\infty(\thetaalgebradegone^*,\mathrm{Im} d)$-component of  $(\Theta_{(g,m)}^{\galgebradegzero^*},\Theta_{(g,m)}^{\thetaalgebradegone^*})$, namely the sum of those terms related to $\mu_i\in C^\infty(\thetaalgebradegone^*)$, coincides with   the multiplicative form $s^*\gamma-t^*\gamma$, where $\gamma=\sum_{i=1}^{r}\mu_i u_i\in \Omega^1(\thetaalgebradegone^*)$.  And for the $1$-form  $\gamma'=\sum_{k=r+1}^q \nu_k u_k\in \Omega^1(\thetaalgebradegone^*)$, where $u_k\in \ker d$, the multiplicative form $s^*\gamma'-t^*\gamma'$ gives the $\beta_k$-part:
\[s^*\gamma'-t^*\gamma'=\sum_{k=r+1}^q (s^*\nu_k-t^*\nu_k)u_k\in \Omega^1_{\mult}(\galgebradegzero^*\triangleright \thetaalgebradegone^*),\qquad (s^*\nu_k-t^*\nu_k)(g,m)=\nu_k(m)-\nu_k(gm).\]}

\begin{corollary} 
	\begin{itemize}
		\item[\rm (1)] If $d$ is injective, then we have $\Omega^1_{\mult}(\galgebradegzero^*\triangleright \thetaalgebradegone^*)\cong C^\infty(\thetaalgebradegone^*, \mathrm{Im} d)\oplus \mathrm{coker} d$;
		\item[\rm (2)] If $d$ is surjective, then we have $\Omega^1_{\mult}(\galgebradegzero^*\triangleright \thetaalgebradegone^*)\cong C^\infty(\thetaalgebradegone^*,\galgebradegzero) \oplus C^\infty_{\mult}(\galgebradegzero^*\triangleright \thetaalgebradegone^*,\ker d)$;
		\item[\rm (3)] If $d=0$, then we have $\Omega^1_{\mult}(\galgebradegzero^*\triangleright \thetaalgebradegone^*)\cong C^\infty(\thetaalgebradegone^*,\galgebradegzero) \oplus C^\infty(\thetaalgebradegone^*, \galgebradegzero\otimes \thetaalgebradegone)$.
	\end{itemize}

\end{corollary}

%sofarhere

%\subsection{}

  Second, we turn to multiplicative vector fields on a linear action groupoid.   {The following fact follows from  Proposition \ref{vvector fields}. }
\begin{proposition}We have an isomorphism  
	\[\mathfrak{X}^1_{\mult}(\galgebradegzero^*\triangleright \thetaalgebradegone^*)=C^\infty(\thetaalgebradegone^*,\mathrm{Im} d^T) \oplus C^\infty(\thetaalgebradegone^*,\mathrm{coker} d^T)^{\galgebradegzero^*} \oplus C^\infty_{\mult}(\galgebradegzero^*\triangleright \thetaalgebradegone^*, \ker d^T),\]
	where $\thetaalgebradegone^*=\mathrm{Im} d^T\oplus \mathrm{coker}d^T$.
\end{proposition}

In fact, if we  continue using the notations introduced in the   proof of Proposition \ref{Prop:easyexampleprop1}, then	a multiplicative vector field $X=(X^{\galgebradegzero^*},X^{\thetaalgebradegone^*})$ can be written in  the form 
\begin{eqnarray}
\label{localvf1} X_{(g,m)}^{\galgebradegzero^*}&=&\sum_{i=1}^r \big(\mu_i(gm)-\mu_i(m)\big) x^i+\sum_{j=r+1}^p \beta_j(g,m) x^j,\\
\label{localvf2} X_{(g,m)}^{\thetaalgebradegone^*}&=&X_m^{\thetaalgebradegone^*}=\sum_{i=1}^r \mu_i(m) d^T x^i+\sum_{k=r+1}^q \alpha_k(m) u^k,
\end{eqnarray}
where $\mu_i\in C^\infty(\thetaalgebradegone^*)$ and $\alpha_k\in C^\infty(\thetaalgebradegone^*)^{\galgebradegzero^*}$ satisfy   $\alpha_k(gm)=\alpha_k(m)$, and $\beta_j\in C^\infty_{\mult}(\galgebradegzero^*\triangleright \thetaalgebradegone^*)$  are multiplicative functions on  $\galgebradegzero^*\triangleright \thetaalgebradegone^*$.

 {Further, we see that the   $C^\infty(\thetaalgebradegone^*,\mathrm{Im} d^T)$-part
 	of a multiplicative vector field $(X_{(g,m)}^{\galgebradegzero^*},X_{(g,m)}^{\thetaalgebradegone^*})$, namely the sum of terms related to
 	 $\mu_i$, is given by $\overrightarrow{e}-\overleftarrow{e}$ for $e=\sum_{i=1}^r\mu_i x^i\in \Gamma(\galgebradegzero^*\triangleright \thetaalgebradegone^*)$, a section of the Lie algebroid. While for $e'=\sum_{j=r+1}^p\nu_j x^j\in \Gamma(\galgebradegzero^*\triangleright \thetaalgebradegone^*)$ with $d^Tx^j=0$, the multiplicative vector field $\overrightarrow{e'}-\overleftarrow{e'}$ is expressed as 
\[\overrightarrow{e'}-\overleftarrow{e'}=\sum_{j={r+1}}^p(t^*\nu_j-s^*\nu_j)x^j,\qquad (t^*\nu_j-s^*\nu_j)(g,m)=\nu_j(gm)-\nu_j(m),\]}
which belongs to $C^\infty_{\mult}(\galgebradegzero^*\triangleright \thetaalgebradegone^*, \ker d^T)$.

In addition, we have the following facts:

\begin{corollary} 
	\begin{itemize}
		\item[\rm (1)] If $d$ is injective, then we have $\mathfrak{X}^1_{\mult}(\galgebradegzero^*\triangleright \thetaalgebradegone^*)\cong C^\infty(\thetaalgebradegone^*,\thetaalgebradegone^*) \oplus C^\infty_{\mult}(\galgebradegzero^*\triangleright \thetaalgebradegone^*,\ker d^T)$;
		\item[\rm (2)] If $d$ is surjective, then we have $\mathfrak{X}^1_{\mult}(\galgebradegzero^*\triangleright \thetaalgebradegone^*)\cong C^\infty(\thetaalgebradegone^*,\mathrm{Im} d^T) \oplus C^\infty(\thetaalgebradegone^*,\mathrm{coker} d^T)^{\galgebradegzero^*}$;
		\item[\rm (3)] If $d=0$, then we have $\mathfrak{X}^1_{\mult}(\galgebradegzero^*\triangleright \thetaalgebradegone^*)\cong C^\infty(\thetaalgebradegone^*,\thetaalgebradegone^*) \oplus  C^\infty(\thetaalgebradegone^*, \galgebradegzero\otimes \galgebradegzero^*)$.
	\end{itemize}

\end{corollary}

\subsubsection{Application to linear quasi-Poisson  $2$-groups} %as a toy model}

If the   $2$-term complex $\thetaalgebradegone\xrightarrow{d} \galgebradegzero$ we mentioned happens to come from a   Lie  $2$-algebra $(\thetaalgebradegone\xrightarrow{d} \galgebradegzero,[~\cdot~,~\cdot~]_2,[~\cdot~,~\cdot~,~\cdot~]_3)$,  then the action Lie groupoid	$ \galgebradegzero^*\triangleright \thetaalgebradegone^*\rightrightarrows \thetaalgebradegone^*$ can be enhanced to   a quasi-Poisson Lie groupoid   with the bivector field $P$ and the $3$-section $\Phi$  defined below:
\[P=[~\cdot~,~\cdot~]_2\in \wedge^2\galgebradegzero^* \otimes \galgebradegzero \oplus \galgebradegzero^*\wedge \thetaalgebradegone^*\otimes \thetaalgebradegone\oplus \wedge^2 \thetaalgebradegone^*\otimes \thetaalgebradegone,\qquad \Phi=[~\cdot~,~\cdot~,~\cdot~]_3\in \wedge^3 \galgebradegzero^*\otimes \thetaalgebradegone.\]
For details, see \cites{CSX, LSX}. 

Making use of Theorem \ref{Thm:case1} and Proposition \ref{Lie2homo}, we obtain two Lie  $2$-algebras and a weak Lie  $2$-algebra  morphism shown as in the following diagram:
\begin{equation}\label{caselie2}
\xymatrix{
	C^\infty(\thetaalgebradegone^*,\thetaalgebradegone)\ar@{^{}->}[d]_{d} \ar[r]^{p^\sharp} & C^\infty(\thetaalgebradegone^*,\galgebradegzero^*) \ar@{^{}->}[d]^{T} \\
	\Omega^1_{\mult}(\galgebradegzero^*\triangleright \thetaalgebradegone^*)\ar[r]^{ P^\sharp}\ar@{.>}[ru]^{\nu}  & \mathfrak{X}^1_{\mult}(\galgebradegzero^*\triangleright \thetaalgebradegone^*)}.
\end{equation}
where $\nu$ is defined as in Equation \eqref{nnu}. 
	
The  two Lie  $2$-algebras that appear in this diagram are both of infinite dimensions. We shall find two  finite dimensional sub Lie  $2$-algebras. Note that the quasi-Poisson Lie groupoid $ \galgebradegzero^*\triangleright \thetaalgebradegone^*\rightrightarrows \thetaalgebradegone^*$ is actually a   \textbf{quasi-Poisson $2$-group} whose Lie  $2$-bialgebra is $(\g^*,\g)$, where the Lie  $2$-algebra structure on $\g^*$ is trivial  \cite{CSX}. So the Lie group structure on $ \galgebradegzero^*\triangleright \thetaalgebradegone^*$ is indeed abelian, namely,
	\[ {(g,m)~\cdot~ (h,n)=(g+h,m+n)},\qquad \forall g,h\in \galgebradegzero^*,m,n\in \thetaalgebradegone^*.\]

	 %and the (abelian) Lie group structure on $\galgebradegzero^*$ is the addition.
	 %\begin{example}\label{coad2}
	 %Let $\g=(\thetaalgebradegone\xrightarrow{d} \galgebradegzero,l_2,l_3)$ be a Lie  $2$-algebra. Consider the quasi-Poisson  $2$-group $ \galgebradegzero^*\triangleright \thetaalgebradegone^* \rightrightarrows \thetaalgebradegone^*$.
         %Then $\Gpd:=\galgebradegzero^*\times \thetaalgebradegone^*\rightrightarrows \thetaalgebradegone^*$ is an action Lie groupoid and further a quasi-Poisson Lie groupoid. It is moreover a  quasi-Poisson  $2$-group whose Lie  $2$-bialgebra is $(\g^*,\g)$, where the Lie  $2$-algebra structure on $\g^*$ is trivial. (See \cite{CSX}).
		
		By saying a \textit{bi-multiplicative form}  on $\galgebradegzero^*\triangleright \thetaalgebradegone^*$, we mean a differential form  (of any degree)   that are multiplicative with respect to both the groupoid and group structures of the Lie $2$-group $\galgebradegzero^*\triangleright \thetaalgebradegone^*$. The notation of the space of bi-multiplicative forms is $\Omega^\bullet_{\bmult}(\galgebradegzero^*\triangleright \thetaalgebradegone^*)$. Similarly, we use $\mathfrak{X}^\bullet_{\bmult}(\galgebradegzero^*\triangleright \thetaalgebradegone^*)$ 
		to denote the space of \textit{bi-multiplicative vector fields} on on $\galgebradegzero^*\triangleright \thetaalgebradegone^*$ which are multiplicative with respect to both the groupoid and group structures.
	Indeed, one can give more concrete characterizations of these spaces. Our last proposition gives an illustration of  the $\bullet=1$ case.
		
		\begin{proposition}For bi-multiplicative $1$-forms and vector fields, we have
			\[\Omega^1_{\bmult}(\galgebradegzero^*\triangleright \thetaalgebradegone^*)=\galgebradegzero,\] \[\mbox{and}\qquad  
		\mathfrak{X}^1_{\bmult}(\galgebradegzero^*\triangleright \thetaalgebradegone^*)=\End_0(\g^*):=\{(A,B)\in \End(\galgebradegzero^*)\oplus \End(\thetaalgebradegone^*)|d^T \circ A=B\circ d^T\}.\]
		\end{proposition}
                 \begin{proof}
		The space of multiplicative $1$-forms on the abelian Lie group $\galgebradegzero^*\triangleright \thetaalgebradegone^*$ coincides with $\galgebradegzero\oplus \thetaalgebradegone$. According to Proposition \ref{Prop:easyexampleprop1}, we have $\Omega^1_{\bmult}(\galgebradegzero^*\triangleright \thetaalgebradegone^*)=\galgebradegzero$.	
		
		 By Example \ref{abeliangroup}, a vector field $X\in \mathfrak{X}^1(\galgebradegzero^*\triangleright \thetaalgebradegone^*)$ is multiplicative  with respect to the abelian group structure on $\galgebradegzero^*\oplus\thetaalgebradegone^*$ if and only if it is of the form
		        \[X=\begin{pmatrix} A & C \\ D& B\end{pmatrix}\in \End(\galgebradegzero^*\oplus\thetaalgebradegone^*),\qquad A\in \End(\galgebradegzero^*), B\in \End(\thetaalgebradegone^*), C\in \Hom(\thetaalgebradegone^*, \galgebradegzero^*), D\in \Hom(\galgebradegzero^*,\thetaalgebradegone^*).\]
			If $X$ is further multiplicative regarding the groupoid structure, then it takes the form in \eqref{localvf1} and \eqref{localvf2}.
			So the functions $\mu_i\in C^\infty(\thetaalgebradegone^*),\alpha_k\in C^\infty(\thetaalgebradegone^*)^{\galgebradegzero^*},\beta_j\in C^\infty_{\mult}(\galgebradegzero^*\triangleright \thetaalgebradegone^*)$ in the two formulas are indeed  linear functions, and hence  
			\[\mu_i,\alpha_k \in \thetaalgebradegone,\qquad   d\alpha_k=0,\qquad \beta_j\in \galgebradegzero.\]
			Then Equations \eqref{localvf1} and \eqref{localvf2} turn to
			\[X^{\galgebradegzero^*}_{(g,m)}=\sum_{i=1}^r \mu_i(d^T g) x^i+\sum_{j=r+1}^p \beta_j(g) x^j,\qquad X^{\thetaalgebradegone^*}_{(g,m)}=\sum_{i=1}^r \mu_i(m) d^Tx^i+\sum_{k=r+1}^q \alpha_k(m) u^k.\]
			(Here $x^i$ and $u^k$ are as in the proof of Proposition \ref{Prop:easyexampleprop1}.)
			From this fact we see that
			\[X_{(0,m)}^{\thetaalgebradegone^*}=B_m+C_m\in \thetaalgebradegone^*,\qquad X_{(g,0)}^{\galgebradegzero^*}=A_g+D_g\in \galgebradegzero^*,\qquad d^TX^{\galgebradegzero^*}_{(g,d^Tg)}=X^{\thetaalgebradegone^*}_{(g,d^Tg)},\]
			which implies that $C=0, D=0$ and $d^T \circ A=B\circ d^T$.
		\end{proof}

	             The Lie  $2$-algebra $\Omega^1_{\mult}(\thetaalgebradegone^*)\to \Omega^1_{\bmult}(\galgebradegzero^*\triangleright \thetaalgebradegone^*)$ is actually the original Lie  $2$-algebra $\thetaalgebradegone\to \galgebradegzero$, where we only consider   multiplicative $1$-forms on the abelian Lie group $\thetaalgebradegone^*$. Further, restricting on  linear sections of the Lie algebroid $\galgebradegzero^*\triangleright \thetaalgebradegone^*\to \thetaalgebradegone^*$, the  Lie  $2$-algebra $\Gamma_{\mathrm{linear}}(\galgebradegzero^*\triangleright \thetaalgebradegone^*)\to \mathfrak{X}^1_{\bmult}(\galgebradegzero^*\triangleright \thetaalgebradegone^*)$ turns out to match with $\End(\g^*)$, i.e., it is of the form: 
		\[\Hom(\thetaalgebradegone^*,\galgebradegzero^*)\xrightarrow{T} \End_0(\g^*),\qquad T(D)=(D\circ d^*,d^*\circ D).\]
		    Moreover, the Lie  $2$-algebra morphism in \eqref{caselie2} becomes the the coadjoint action $(\ad_0^*,\ad^*_1,\ad^*_2)$ of the  Lie  $2$-algebra $\g$ on its dual $\g^*$: 
		\begin{equation*}
		\xymatrix{
			\thetaalgebradegone\ar@{^{}->}[d]_{d} \ar[rr]^{\ad^*_1} && \Hom(\thetaalgebradegone^*,\galgebradegzero^*)  \ar@{^{}->}[d]^{T} \\
			\galgebradegzero\ar[rr]^{\ad_0^*}\ar@{.>}[rru]^{\nu} &&  \End_0( \g^*)},
		\end{equation*}
		where   $\nu:\wedge^2\galgebradegzero\to \Hom(\thetaalgebradegone^*,\galgebradegzero^*)$ is given by
		\[\nu(x,y)=-[x,y,~\cdot~]_3^*,\qquad \forall x,y\in \galgebradegzero.\]
		This is a Lie  $2$-algebra version of Diagram \eqref{coad}.
 	
 	  \section{Infinitesimal multiplicative (IM) $1$-forms on a quasi-Lie bialgebroid}
 	  \label{Sec:IM1forms}

 	  \subsection{IM  $1$-forms of a Lie algebroid}  
 	  Let $A$ be a Lie algebroid over $M$.  	  
 	  Recall from \cite{BC} that an {\bf  IM $1$-form} of the Lie algebroid $A$ is defined to be  	  
 	  a pair $(\nu,\theta)$ where $\nu:A\to T ^*M$ is a morphism of vector bundles, $\theta \in  \Gamma(A^*)$, and the following conditions are satisfied:
 	  \begin{eqnarray}
 	  \label{IM1}\theta[x,y]&=&\rho(x)\theta(y)-\rho(y)\theta(x)-\langle \rho(y),\nu(x)\rangle,  \\ 
 	  \label{IM2}\nu[x,y]&=&L_{\rho(x)} \nu(y)-\iota_{\rho(y)} d\nu(x),
 	  \end{eqnarray}
 	  for all $x,y\in \Gamma(A)$.  {Equation \eqref{IM1} is also formulated as $   (d_A \theta)(x,y)=\langle \rho(y),\nu(x)\rangle  $ where $d_A: \Gamma(A^*)\to \Gamma(\wedge^2 A^*)$ is  the differential associated with the Lie algebroid structure of $A$.} 
 	  
 	  Denote by $\mathrm{IM}^1(A)$ the set of IM $1$-forms. Indeed, for any $k\geqslant 0$, there is also the notion of IM $k$-forms on $A$ forming the set $\mathrm{IM}^k(A)$. For details, see \cite{BC}.
 	  
 	  To any $\gamma\in \Omega^1(M)$ is associated a pair $(\iota_{\rho(\cdot)} d\gamma,\rho^*\gamma)$, which is an example of IM $1$-form of $A$.
 	  
 	  	  We also recall an important fact \cite{BC}*{Theorem 2}. Let $\Gpd$ be a source-simply-connected Lie groupoid over $M$ with Lie algebroid $A\to M$. There exists a one-to-one correspondence between multiplicative $1$-forms on $\Gpd$ and IM $1$-forms. To be specific,   $\alpha\in \Omega^1_{\mult}(\Gpd)$  corresponds to $\sigma(\alpha):=(\nu,\theta)\in \mathrm{IM}^1(A)$ defined by
 	  	\begin{eqnarray}\label{sigma1}
 	  	 \langle\nu(x),U\rangle&=&d\alpha (x,U), \\\label{sigma2}
 	  	 \mbox{ and }\quad \theta(x)&=&\alpha(x)
 	  	 \end{eqnarray} 	  	
 	  	for $x\in \sections{A}$ and $U\in \mathfrak{X}^1(M)$. More generally, one has $\Omega^k_{\mult}(\Gpd)\cong \mathrm{IM}^k(A)$.
 	  	
 	   Now let $(\Gpd,P,\Phi)$ be a quasi-Poisson groupoid. By Theorem \ref{Thm:case1}, we have a   weak Lie $2$-algebra 
 	   $\Omega^1(M)\xrightarrow{J}\Omega^1_{\mult}(\mathcal{G}) $. Hence, if $\Gpd$ is source-simply-connected, then  $\Omega^1_{\mult}(\mathcal{G})$ can be identified with $\mathrm{IM}^1(A)$ and we also have a weak Lie $2$-algebra 
 	   $\Omega^1(M)\xrightarrow{j}  \mathrm{IM}^1(A)$. 
 	   
 	   Since quasi-Lie bialgebroids are infinitesimal replacements of quasi-Poisson groupoids \cite{ILX}, it is   natural to expect that  a weak Lie $2$-algebra $\Omega^1(M)\xrightarrow{j} \mathrm{IM}^1(A)$ is directly associated with a quasi-Lie bialgebroid $(A,\dstar ,\Phi)$.   
 	   In what follows, we demonstrate this fact. It is  worth noting that the results presented can be extended to the graded space of all degree IM forms $\mathrm{IM}^\bullet(A)$ of a quasi-Lie bialgebroid $A$, although for brevity, we limit our consideration to IM $1$-forms.

 \subsection{The weak Lie $2$-algebra of IM $1$-forms on a quasi-Lie bialgebroid}
 We start with recalling the definition of a quasi-Lie bialgebroid.
  \begin{definition}\cite{R} 
  A quasi-Lie bialgebroid is a triple  $(A,\dstar ,\Phi)$  consisting of a Lie algebroid $A$ (over the base manifold $M$), a section $\Phi\in \Gamma(\wedge^3 A)$, and an operator $\dstar :\Gamma(\wedge^\bullet A )\to \Gamma(\wedge^{\bullet+1} A )$ satisfying the following conditions
  \begin{itemize}
  	\item $\dstar $ is a derivation of degree $1$, i.e.,
  	$$\dstar (x \wedge y )=\dstar  x \wedge y  + (-1)^k x\wedge \dstar  y,\quad \forall  x \in \Gamma(\wedge^k A ),  y \in \Gamma(\wedge^\bullet A );$$
  	\item $\dstar $ is a derivation of the Schouten bracket, i.e.,
  	$$\dstar[x,y]=[\dstar x, y]+(-1)^{k-1}[x,\dstar y],\quad \forall  x \in \Gamma(\wedge^k A ),  y \in \Gamma(\wedge^\bullet A );$$
  	\item The square of $\dstar $ is controlled by $\Phi$ in the sense that $\dstar^2 =-[\Phi,\tobefilledin]$, as a map $\Gamma(\wedge^\bullet A )\to \Gamma(\wedge^{\bullet+2} A )$ and $d_*\Phi=0$.
  \end{itemize}
  \end{definition}
   {The operator $\dstar$ in a quasi-Lie bialgebroid gives rise to an anchor map $\rho_*:A^*\to TM$ and a bracket $[\cdot,\cdot]_*$ on $\Gamma(A^*)$ defined as follows:}
  \begin{eqnarray*}
  \rho_*(\xi)f&=&\langle d_*f, \xi\rangle;\\
  \langle [\xi,\xi']_*,x\rangle&=&\rho_*(\xi)\langle \xi',x\rangle-\rho_*(\xi')\langle \xi,x\rangle-\langle d_*x,\xi\wedge \xi'\rangle,
  \end{eqnarray*}
  for all $f\in C^\infty(M), x\in \Gamma(A)$ and $\xi,\xi'\in \Gamma(A^*)$. But note that $(A^*,[\cdot,\cdot]_*,\rho_*)$ does not form a Lie algebroid. 
  
Stemming from a quasi-Lie bialgebroid, we have an associated weak Lie $2$-algebra underlying IM $1$-forms; our main theorem below gives the details of this construction.
 \begin{theorem}\label{Thm:quasiLiebitoIM2algebra}Let $(A,\dstar ,\Phi)$ be a quasi-Lie bialgebroid as defined above. With the following structure maps, the 2-term complex  
 	\[ \Omega^1(M)\xrightarrow{j} \mathrm{IM}^1(A),\qquad j(\gamma)=(-\iota_{\rho(~\cdot~)} d\gamma, -\rho^*\gamma)\] 
 	 composes a weak Lie $2$-algebra.
 	 \begin{itemize}
 	 	\item The skew-symmetric bracket on $\mathrm{IM}^1(A)$ is defined by
 	 	\begin{eqnarray}\nonumber%\label{bracket}
 	 	[(\nu,\theta),(\nu', \theta')] &=&\nonumber \big(\nu\circ \rho_{*}^*\circ \nu'-\nu'\circ \rho_{*}^*\circ \nu+L_{(\rho_{*} \theta)}\nu'(\cdot)-\nu'(L_\theta(\cdot))-L_{(\rho_{*} \theta')} \nu(\cdot)+\nu(L_{\theta'}(\cdot)),\\ &&\label{bracket}[\theta,\theta']_{*}\big). 
 	 	\end{eqnarray}
 	% for all	$(\nu,\theta),(\nu',\theta')\in \mathrm{IM}^1(A)$.
 	 	\item The action of $\mathrm{IM}^1(A)$ on $\Omega^1(M)$ is defined by 
 	 	\[(\nu,\theta)\triangleright \gamma=\nu(\rho_{*}^*\gamma)+L_{\rho_* \theta} \gamma.\]
 	 	\item The $3$-bracket $[\cdot,\cdot,\cdot]_3:~ \otimes_{\mathbb{R}}^3 (\mathrm{IM}^1(A))\to \Omega^1(M)$ is defined by
 	 	\[[(\nu_1,\theta_1),(\nu_2,\theta_2),(\nu_3,\theta_3)]_3=d\Phi(\theta_1,\theta_2,\theta_3)+\nu_1(\Phi(\theta_2,\theta_3))+\nu_2(\Phi(\theta_3,\theta_1))+\nu_3(\Phi(\theta_1,\theta_2)).\]
 	 \end{itemize}
 \end{theorem}
 
 Recall that   $\Gamma(A)\xrightarrow{t}\Der(A)$ with $t(u)=[u,~\cdot~]$ is a strict Lie  $2$-algebra. It turns out that the weak  Lie  $2$-algebra we just constructed is connected to  $\Gamma(A)\xrightarrow{t}\Der(A)$ in a nice manner.

 \begin{proposition}\label{Prop:IMmorphism}Under the same assumptions as in the above theorem, there exists a weak Lie  $2$-algebra homomorphism $(\psi_0,\rho_*^*,\psi_2)$:
 	\begin{equation*}
 	\xymatrix{
 		\Omega^1(M)\ar@{^{}->}[d]_{j}  \ar[r]^{\rho_{*}^*} &  \Gamma(A) \ar@{^{}->}[d]^{t} \\
 		\mathrm{IM}^1(A)\ar[r]^{\psi_0} & \Der (A)},
 	\end{equation*}
 	where  $\psi_0(\nu,\theta)=\rho_{*}^* \nu(\cdot)+L_{\theta}(\cdot)$ and $\psi_2:\wedge^2 \mathrm{IM}^1(A)\to \Gamma(A)$ is given by 
 	\[\psi_2((\nu,\theta),(\nu',\theta'))=\Phi(\theta,\theta').\]
 
 \end{proposition}

 	The proofs of these results are quite involved and hence we divide them into several parts.
 \subsubsection{Well-definedness of the $2$-bracket} We verify that the resulting pair  $(\tilde{\nu},\tilde{\theta}):=[(\nu,\theta),(\nu', \theta')]$ given by Equation \eqref{bracket} satisfies \eqref{IM1} and \eqref{IM2}, {namely, $[(\nu,\theta),(\nu', \theta')]\in \mathrm{IM}^1(A)$. }

 Since $(A,\dstar ,\Phi)$ is a quasi-Lie bialgebroid, we have 
 \[d_A [\theta,\theta']_*=[d_A\theta,\theta']_*+[\theta,d_A \theta']_*,\qquad \forall \theta,\theta'\in \Gamma(A^*).\]
 Then using \eqref{IM1}  for $(\nu,\theta), (\nu',\theta')$ and the following relations due to \cite{MX1}:
 \begin{eqnarray}\label{LiP}
 L_{\rho^*_*\gamma} \theta=-[\rho^*\gamma,\theta]_*-\rho^*(\iota_{\rho_*\theta} d\gamma),\quad \qquad L_{\rho^*\gamma} x=-[\rho_*^*\gamma,x]-\rho_*^*(\iota_{\rho x} d\gamma)
 \end{eqnarray}
 for all $\gamma\in \Omega^1(M),\theta\in \Gamma(A^*), x\in \Gamma(A)$,
 we further obtain
 \begin{eqnarray*}
 	d_A [\theta,\theta']_*(x,y)&=&-L_{\theta'}(d_A \theta)(x,y)-c.p.\\ &=&
 	-\rho_*(\theta')d_A \theta(x,y)+d_A \theta(L_{\theta'}x,y)+d_A \theta(x,L_{\theta'}y)-c.p.\\ &=&
 	-\rho_*(\theta')\langle \rho(y),\nu(x)\rangle+\langle \rho(y),\nu(L_{\theta'} x)\rangle+\langle \rho(L_{\theta'} y),\nu(x)\rangle-c.p.
 	\\ &=&\langle y,[\rho^*\nu(x),\theta']_*\rangle+\langle \rho(y),\nu(L_{\theta'} x)\rangle-c.p.
 	\\ &=&\langle y, -L_{\rho_*^* \nu(x)} \theta'-\rho^*(\iota_{\rho_*\theta'} d\nu(x))\rangle+\langle \rho(y),\nu(L_{\theta'} x)\rangle-c.p.
 	\\ &=&\langle y, -\iota_{\rho_*^* \nu(x)} d_A \theta'-\rho^*d\langle \nu(x),\rho_* \theta' \rangle -\rho^*(\iota_{\rho_*\theta'} d\nu(x))\rangle+\langle \rho y,\nu(L_{\theta'} x)\rangle-c.p.
 	\\ &=&
 	\langle y, -\rho^*\nu'(\rho_*^* \nu(x))-\rho^*L_{\rho_*\theta'} \nu(x)\rangle+\langle \rho y,\nu(L_{\theta'} x)\rangle-c.p.\\ &=&\langle 
 	\rho y,-\nu'\rho_*^*\nu(x)-L_{\rho_*\theta'}\nu(x)+\nu(L_{\theta'} x)+\nu\rho_*^*\nu'(x)+L_{\rho_*\theta}\nu'(x)-\nu'(L_{\theta} x)\rangle.
 \end{eqnarray*}
 So we proved \eqref{IM1}.  Then it is left to check \eqref{IM2} for $(\tilde{\nu},\tilde{\theta})$. Using the formula 
 \begin{eqnarray}\label{Ld}
 L_{\theta}[x,y]=[L_{\theta}x,y]+[x,L_{\theta} y]-L_{\iota_x d_{A} \theta} y+\iota_{\iota_y d_{A} \theta} d_* x,
 \end{eqnarray}
 we have 
 \begin{eqnarray*}
 	\tilde{\nu}[x,y]&=&\nu\rho_*^*\nu'[x,y]+L_{\rho_*\theta}\nu'[x,y]-\nu'(L_{\theta}[x,y])-c.p.\\ &=&
 	\nu \rho_*^*\big(L_{\rho x} \nu'(y)-\iota_{\rho y}d\nu'(x)\big)+L_{\rho_*\theta} (L_{\rho x} \nu'(y)-\iota_{\rho y} d\nu'(x))\\ &&
 	-L_{\rho(L_{\theta} x)} \nu'(y)+\iota_{\rho y} d\nu'(L_{\theta} x)-L_{\rho x} \nu'(L_{\theta} y)+\iota_{\rho(L_{\theta} y)} d\nu'(x)+\nu'(L_{\iota_x d_{A} \theta} y-\iota_{\iota_y d_{A} \theta} d_*x)-c.p.,
 \end{eqnarray*}
 and 
 \begin{eqnarray*}
 	&&L_{\rho x} \tilde{\nu}(y)-\iota_{\rho y} d\tilde{\nu}(x)\\ &=&L_{\rho x}\big(\nu\rho_*^*\nu'(y)+L_{\rho_*\theta} \nu'(y)-\nu'(L_{\theta} y)\big)-\iota_{\rho y} d\big(\nu\rho_*^*\nu'(x)+L_{\rho_*\theta} \nu'(x)-\nu'(L_{\theta} x)\big)-c.p..
 \end{eqnarray*}
 According to  Equations \eqref{IM2}  and \eqref{LiP}, we have
 \begin{eqnarray*}
 	\nu\rho_*^*L_{\rho x} \nu'(y)&=&\nu\big(\rho_*^*\iota_{\rho x} d\nu'(y)+\rho_*^*d\iota_{\rho x} \nu'(y)\big)
 	\\ &=&\nu([x,\rho_*^*\nu'(y)]-L_{\rho^*\nu'(y)} x+d_{*}\iota_x \rho^*\nu'(y))\\ &=&
 	L_{\rho x} \nu \rho_*^*\nu'(y)-\iota_{\rho \rho_*^*\nu'(y)} d\nu(x)-\nu(\iota_{\rho^*\nu'(y)} d_* x),
 \end{eqnarray*}
 and
 \begin{eqnarray*}
 	-\nu\rho_*^*\iota_{\rho y}d\nu'(x)&=&\nu(L_{\rho^*\nu'(x)} y+[\rho_*^*\nu'(x),y])\\ &=&
 	\nu(L_{\rho^*\nu'(x)} y)+L_{\rho\rho_*^* \nu'(x)} \nu(y)-\iota_{\rho y} d\nu(\rho_*^*\nu'(x)).
 \end{eqnarray*}
 Utilizing the above  relations to $\tilde{\nu}[x,y]$, we obtain\begin{eqnarray*}
 	\tilde{\nu}[x,y]-L_{\rho x} \tilde{\nu}(y)-\iota_{\rho y} d\tilde{\nu}(x)&=&\big(L_{[\rho_* \theta,\rho x]} \nu'(y)-L_{\rho (L_{\theta} x)} \nu'(y)\big)+\big(\iota_{[\rho y,\rho_*\theta]} d\nu'(x)+\iota_{\rho(L_{\theta} y)} d\nu'(x)\big)\\ &&+L_{\rho\rho_*^* \nu'(x)} \nu(y)-\iota_{\rho \rho_*^*\nu'(y)} d\nu(x)-c.p.\\ 
 	&=&0,
 \end{eqnarray*}
 where we have used \eqref{IM1}, the Cartan  formulas 
 \[d\circ L_u=L_u\circ d,\qquad L_u\circ \iota_v-\iota_v \circ L_u=\iota_{[u,v]},\qquad \forall u,v\in \mathfrak{X}^1(M),\]
 and the equations
 \begin{eqnarray}\label{LiP2}
 [\rho_*\theta,\rho x]=\rho(L_{\theta} x)-\rho_*(\iota_x d_{A} \theta),\qquad \rho_*\circ  \rho^*=-\rho\circ \rho_*^*.
 \end{eqnarray}
 Hence we proved that $(\tilde{\nu},\tilde{\theta})$ satisfies \eqref{IM2}, and  verified that $[(\nu,\theta),(\nu',\theta')]\in \mathrm{IM}^1(A)$.
 \subsubsection{A key property of the $2$-bracket}

 \begin{lemma} Given $(\nu,\theta)\in \mathrm{IM}^1(A)$, for all \ $\gamma\in \Omega^1(M)$, define $\mu=\mu(\gamma):=\nu(\rho_*^*\gamma)+L_{\rho_*\theta} \gamma\in \Omega^1(M)$.
 	We have the following identity 
 	\begin{eqnarray}\label{ex}
 	[(\nu,\theta),(\iota_{\rho(\cdot)} d\gamma,\rho^*\gamma)]=(\iota_{\rho(\cdot)}d\mu,\rho^*\mu).
 	\end{eqnarray}
 	\end{lemma}
 \begin{proof}

 	To simplify notations, we denote $(\hat{\nu},\hat{\theta}):=[(\nu,\theta),(\iota_{\rho(\cdot)} d\gamma,\rho^*\gamma)]$. Then by  Equations \eqref{LiP} and \eqref{IM1}, we have 
 	\[\hat{\theta}=[\theta,\rho^*\gamma]_*=L_{\rho_*^*\gamma}\theta+\rho^*(\iota_{\rho_*\theta}d\gamma)=\iota_{\rho_*^*\gamma} d_A  \theta+d_A \langle \gamma, \rho_* \theta\rangle+\rho^*(\iota_{\rho_*\theta}d\gamma)=\rho^*\nu(\rho_*^*\gamma)+\rho^*L_{\rho_*\theta} \gamma,\]
 	which is exactly $\rho^*\mu$. Next we compute $\hat{\nu}$. When it is applied to $x\in \Gamma(A)$, and using \eqref{LiP},\eqref{LiP2},\eqref{IM1}, \eqref{IM2}, we can explicitly describe  $\hat{\nu}$:
 	\begin{eqnarray*}
 		\hat{\nu}(x)&=&\nu \rho_*^*\iota_{\rho(x)} d\gamma-\iota_{\rho(\rho_*^*\nu(x))}d\gamma+L_{\rho_* \theta} \iota_{\rho x}d\gamma-\iota_{\rho(L_{\theta} x)}d\gamma-L_{\rho_*\rho^*\gamma} \nu(x)+\nu(L_{\rho^*\gamma} x)\\ &=&\nu([x,\rho_*^*\gamma])-\iota_{[\rho_* \theta,\rho x]}d\gamma+L_{\rho_* \theta} \iota_{\rho x}d\gamma-L_{\rho_*\rho^*\gamma} \nu(x)\\ &=&-L_{\rho(\rho_*^* \gamma)} \nu(x)+\iota_{\rho x} d\nu(\rho_*^*\gamma)-(L_{\rho_* \theta}\iota_{\rho x}-\iota_{\rho x}L_{\rho_* \theta}) d\gamma+L_{\rho_* \theta} \iota_{\rho x}d\gamma-L_{\rho_*\rho^*\gamma} \nu(x)
 		\\ &=&\iota_{\rho x} d\nu(\rho_*^*\gamma)+\iota_{\rho x}L_{\rho_* \theta} d\gamma\\ &=&\iota_{\rho x} d\mu,
 	\end{eqnarray*}
 	where in the second-to-last calculation, we utilized $d\circ L_{\rho_*\theta}=L_{\rho_*\theta}\circ d$. Thus we proved \eqref{ex}.
 \end{proof}
 
 \subsubsection{Proof of Theorem \ref{Thm:quasiLiebitoIM2algebra}}
  \begin{itemize}
 	\item[~(1)~] \qquad We first show two relations:
 	\begin{eqnarray*}\label{ideal}
 		[(\nu,\theta),j\gamma]=j((\nu,\theta)\triangleright \gamma),\qquad (j\gamma)\triangleright \gamma'=-(j\gamma')\triangleright \gamma,
 	\end{eqnarray*}
 	for $(\nu,\theta)\in \mathrm{IM}^1(A)$ and $\gamma,\gamma'\in \Omega^1(M)$. 
 	
 	The first one follows directly from \eqref{ex}. 
 	
 	To see the second one, consider the map  $\pi^\sharp:=\rho\circ \rho^*_*: T^*M\to TM$. Since $\rho\circ \rho_*^*=-\rho_*\circ \rho^*$, $\pi$ is a bivector field on the base manifold $M$ and thus defines a skew-symmetric bracket (not necessarily Lie) $[\cdot,\cdot]_\pi$ on $\Omega^1(M)$. It follows that  
 	\[(j\gamma)\triangleright \gamma'=L_{\pi^\sharp \gamma} \gamma'-\iota_{\pi^\sharp \gamma'} d\gamma=[\gamma,\gamma']_\pi=-(j\gamma')\triangleright \gamma.\]
 	\item[~(2)~]  \qquad Next, we show that the $2$-bracket \eqref{bracket} satisfies a generalized type of Jacobi identity:
 	\begin{eqnarray}\label{Jacobi}
 	[[(\nu_1,\theta_1),(\nu_2,\theta_2)],(\nu_3,\theta_3)]+c.p.=-j[(\nu_1,\theta_1),(\nu_2,\theta_2),(\nu_3,\theta_3)]_3.
 	\end{eqnarray}
 To verify the identity proposed above that involves the $2$-bracket $[\cdot,\cdot]$, which is $\mathbb{R}$-bilinear, all possible combinations of $\nu_i$ and $\theta_i$ should be considered. For instance, when focusing solely on the pure entries of $\nu_i$, it is easy to see that they do not contribute to the left hand side of Equation \eqref{Jacobi}. This is due to the fact that by definition, we have $[[\nu_1,\nu_2],\nu_3]+c.p.=0$.

 Using the axioms of a quasi-Lie bialgebroid $(A, \dstar, \Phi)$ and Equation \eqref{IM1}, we can establish the following equality by considering only $\theta_i$ in the entries: 	\begin{eqnarray*}
 		[[\theta_1,\theta_2],\theta_3]+c.p.&=&d_A\Phi(\theta_1,\theta_2,\theta_3)+\Phi(d_A  \theta_1,\theta_2,\theta_3)-\Phi( \theta_1,d_A \theta_2,\theta_3)+\Phi(\theta_1,\theta_2,d_A \theta_3) \\ &=&\rho^*d\Phi(\theta_1,\theta_2,\theta_3)+\rho^*\nu_1(\Phi(\theta_2,\theta_3))+\rho^*\nu_2(\Phi(\theta_3,\theta_1))+\rho^*\nu_3(\Phi(\theta_1,\theta_2))\\ &=&
 		\rho^*[(\nu_1,\theta_2),(\nu_2,\theta_2),(\nu_3,\theta_3)]_3.
 	\end{eqnarray*}

 	In the meantime, we have the following mixed terms:
 	\begin{eqnarray*}
 		&&[[\nu_1,\nu_2],\theta_3]+[[\nu_2,\theta_3],\nu_1]+[[\theta_3,\nu_1],\nu_2]\\ &=&
 		[\nu_1,\nu_2](L_{\theta_3} (\cdot))-L_{\rho_*\theta_3}[\nu_1,\nu_2](\cdot)+\big(
 		[\nu_2,\theta_3]\rho_*^*\nu_1-\nu_1\rho_*^*[\nu_2,\theta_3]-c.p.(\nu_1,\nu_2)\big) \\ &=&
 		(\nu_1\rho_*^*\nu_2-\nu_2\rho_*^*\nu_1)(L_{\theta_3} (\cdot))-L_{\rho_*\theta_3}(\nu_1\rho_*^*\nu_2-\nu_2\rho_*^*\nu_1)\\ &&+\big(
 		\nu_2(L_{\theta_3} \rho_*^* \nu_1(\cdot))-L_{\rho_*\theta_3}(\nu_2\rho_*^*\nu_1(\cdot))-\nu_1\rho^*_*\nu_2(L_{\theta_3}(\cdot))+\nu_1\rho_*^*(L_{\rho_*\theta_3} \nu_2(\cdot))-c.p.(\nu_1,\nu_2)\big)\\ &=&
 		\nu_1\big(\rho_*^*(L_{\rho_*\theta_3} \nu_2(\cdot))-L_{\theta_3} \rho_*^* \nu_2(\cdot)\big)-c.p.(\nu_1,\nu_2).
 	\end{eqnarray*}
 	Similarly, we have the terms
 	\begin{eqnarray*}
 		&&[[\nu_1,\theta_2],\theta_3]+[[\theta_3,\nu_1],\theta_2]+[[\theta_2,\theta_3],\nu_1]\\ &=&
 		[\nu_1,\theta_2](L_{\theta_3}(\cdot))-L_{\rho_*\theta_3} [\nu_1,\theta_2](\cdot)-c.p.(\theta_2,\theta_3)+L_{\rho_*[\theta_2,\theta_3]_*} \nu_1-\nu_1(L_{[\theta_2,\theta_3]_*} (\cdot))\\ &=&\nu_1(L_{\theta_2}L_{\theta_3}(\cdot))-L_{\rho_* \theta_2} \nu_1(L_{\theta_3}(\cdot))-L_{\rho_*\theta_3} \nu_1(L_{\theta_2}(\cdot))+L_{\rho_*\theta_3} L_{\rho_*\theta_2} \nu_1-c.p.(\theta_2,\theta_3)\\ &&+L_{\rho_*[\theta_2,\theta_3]_*} \nu_1-\nu_1(L_{[\theta_2,\theta_3]_*} (\cdot))\\ &=&\nu_1\big([L_{\theta_2},L_{\theta_3}]-L_{[\theta_2,\theta_3]_*}(\cdot)\big)+L_{\rho \Phi(\theta_2,\theta_3)} \nu_1(\cdot),
 	\end{eqnarray*}
 	where in the last step we used the relation 
 	\begin{eqnarray}\label{anchor}
 	\rho_*[\theta_2,\theta_3]_*=[\rho_*\theta_2,\rho_*\theta_3]+\rho\Phi(\theta_2,\theta_3).
 	\end{eqnarray}

 	Note also that for $\alpha\in \Gamma(A^*)$ and $x\in \Gamma(A)$, by Equations \eqref{IM1} and \eqref{anchor},
 	we have 
 	\begin{eqnarray}\label{eq1}
 	\nonumber &&\langle \rho_*^*(L_{\rho_*\theta_3} \nu_2(x))-L_{\theta_3} \rho_*^* \nu_2(x),\alpha\rangle\\ &=&\nonumber \rho_*\theta_3\langle \nu_2(x),\rho_*\alpha\rangle-\langle \nu_2(x),[\rho_*\theta_3,\rho_*\alpha]\rangle-\rho_*\theta_3\langle \nu_2(x),\rho_*\alpha\rangle+\langle \nu_2(x),\rho_*[\theta_3,\alpha]_*\rangle\\ &=&\langle \nu_2(x),\rho\Phi(\theta_3,\alpha)\rangle=(d_A  \theta_2)(x,\Phi(\theta_3,\alpha)),%\\ &=&\rho x\Phi(\theta_2,\theta_3,\alpha)-\rho \Phi(\theta_3,\alpha) \langle \theta_2,x\rangle-\langle
 	%\theta_2,[x,\Phi(\theta_3,\alpha)]\rangle,
 	\end{eqnarray}
 	and
 	\begin{eqnarray}\label{eq2}
 	\nonumber&& \langle [L_{\theta_2},L_{\theta_3}]x-L_{[\theta_2,\theta_3]_*}x,\alpha\rangle \\ &=&\nonumber\rho_*\theta_2\langle L_{\theta_3} x,\alpha\rangle-\langle L_{\theta_3} x,[\theta_2,\alpha]_*\rangle-c.p.(\theta_2,\theta_3)-\rho_*[\theta_2,\theta_3]_*\langle x,\alpha\rangle+\langle x,[[\theta_2,\theta_3]_*,\alpha]_*\rangle\\ &=&\nonumber\rho_*\theta_2\rho_* {\theta_3} \langle x,\alpha\rangle+\langle x,[\theta_3,[\theta_2,\alpha]_*]_*\rangle-c.p.(\theta_2,\theta_3)-\rho_*[\theta_2,\theta_3]_*\langle x,\alpha\rangle+\langle x,[[\theta_2,\theta_3]_*,\alpha]_*\rangle\\ &=&\nonumber-\rho\Phi(\theta_2,\theta_3)\langle x,\alpha\rangle+\langle x,d_A \Phi(\theta_2,\theta_3,\alpha)+
 	\Phi(d_A  \theta_2,\theta_3,\alpha)-\Phi( \theta_2,d_A \theta_3,\alpha)+\Phi(\theta_2,\theta_3,d_A  \alpha)\rangle\\ &=&
 	(d_A \theta_2)(\Phi(\theta_3,\alpha),x)-(d_A \theta_3)(\Phi(\theta_2,\alpha),x)-\langle \alpha,[\Phi(\theta_2,\theta_3),x]\rangle.
 	\end{eqnarray}
 	Combining the above equalities, we can find the $\Hom(A,T^*M)$-component of the left hand side of Equation \eqref{Jacobi}:
 	\begin{eqnarray*}
 		&&\mathrm{pr}_{\Hom(A,T^*M)}([[(\nu_1,\theta_1),(\nu_2,\theta_2)],(\nu_3,\theta_3)]+c.p.)\\ &=&\big(\nu_1\big(\rho_*^*(L_{\rho_*\theta_3} \nu_2(\cdot))-L_{\theta_3} \rho_*^* \nu_2(\cdot)\big)-c.p.(\nu_1,\nu_2)+c.p.(3)\big)\\ &&+
 		\big(\nu_1\big([L_{\theta_2},L_{\theta_3}]-L_{[\theta_2,\theta_3]_*}(\cdot)\big)+L_{\rho \Phi(\theta_2,\theta_3)} \nu_1(\cdot)+c.p.(3)\big)\\ &=&-\nu_1([\Phi(\theta_2,\theta_3),\cdot])+L_{\rho \Phi(\theta_2,\theta_3)} \nu_1(\cdot)+c.p.(3)\\ &=&\iota_{\rho(\cdot)} d\nu_1(\Phi(\theta_2,\theta_3))+c.p.(3)\\ &=&\iota_{\rho(\cdot)}d[(\nu_1,\theta_2),(\nu_2,\theta_2),(\nu_3,\theta_3)]_3,
 	\end{eqnarray*}
 	where in the second-to-last step we have used \eqref{IM2}. Here  ``c.p.(3)'' means  the rest terms involving   $\nu_2,\theta_3,\nu_3$ and $\nu_3,\theta_1,\nu_1$. 
 	
 	The above lines are exactly the desired Equation \eqref{Jacobi}.
 	\item[~(3)~] \qquad  Third,  we verify a relation:
 	\[[(\nu_1,\theta_1),(\nu_2,\theta_2)]\triangleright \gamma-(\nu_1,\theta_1)\triangleright \big((\nu_2,\theta_2)\triangleright \gamma\big)+(\nu_2,\theta_2)\triangleright \big((\nu_1,\theta_1)\triangleright \gamma\big)=-[(\nu_1,\theta_1),(\nu_2,\theta_2),j\gamma]_3.\]
 	In fact, by \eqref{anchor} and \eqref{eq1}, we can compute the left hand side of the above equation: 
 	\begin{eqnarray*}
 		 & &\big(\nu_1\rho_*^*\nu_2+\nu_1(L_{\theta_2}(\cdot))-L_{\rho_*\theta_2} \nu_1(\cdot)-c.p.(2)\big)(\rho_*^*\gamma)+L_{\rho_*[\theta_1,\theta_2]_*}\gamma\\ &&-\big(\nu_1\rho_*^*(\nu_2\rho_*^*\gamma+L_{\rho_*\theta_2}\gamma)+L_{\rho_*\theta_1}(\nu_2\rho_*^*\gamma+L_{\rho_*\theta_2}\gamma)-c.p.(2)\big)\\ &=&
 		\nu_1(L_{\theta_2}\rho_*^*\gamma-\rho^*_*L_{\rho_*\theta_2} \gamma)-c.p.(2)+L_{\rho\Phi(\theta_1,\theta_2)} \gamma
 		\\ &=&\nu_1(\Phi(\theta_2,\rho^*\gamma))-\nu_2(\Phi(\theta_1,\rho^*\gamma))+d\Phi(\theta_1,\theta_2,\rho^*\gamma)+\iota_{\rho\Phi(\theta_1,\theta_2)} d\gamma,
 	\end{eqnarray*}
 which exactly match with the right hand side.
 	\item[~(4)~] \qquad  We finally check  compatibility of the $2$-bracket and the $3$-bracket, namely, the relation
 	\[-(\nu_4,\theta_4)\triangleright [(\nu_1,\theta_1),(\nu_2,\theta_2),(\nu_3,\theta_3)]_3+c.p.(4)=[[(\nu_1,\theta_1),(\nu_2,\theta_2)],(\nu_3,\theta_3),(\nu_4,\theta_4)]_3+c.p.(6).\]
 	In fact, its left hand side reads
 	\begin{eqnarray*}
 		 &&-\nu_4\rho_*^*\big(\nu_1(\Phi(\theta_2,\theta_3))+c.p.(3)+d\Phi(\theta_1,\theta_2,\theta_3)\big)\\ &&-L_{\rho_*\theta_4}\big(\nu_1(\Phi(\theta_2,\theta_3))+c.p.(3)+d\Phi(\theta_1,\theta_2,\theta_3)\big)+c.p.(4),
 	\end{eqnarray*}
 	while the right hand side reads 
 	\begin{eqnarray*}
 		\mathrm{RHS}&=&d\Phi([\theta_1,\theta_2]_*,\theta_3,\theta_4)+\big(\nu_1\rho_*^*\nu_2+\nu_1(L_{\theta_2}(\cdot))-L_{\rho_*\theta_2}\nu_1(\cdot)-c.p.(2)\big)(\Phi(\theta_3,\theta_4))\\ &&+\nu_3(\Phi(\theta_4,[\theta_1,\theta_2]_*))+\nu_4(\Phi([\theta_1,\theta_2]_*,\theta_3))+c.p.(6).
 	\end{eqnarray*}
 	So,  subtraction of the two sides  equals  
 	\begin{eqnarray*}
 		&&-\nu_4(d_*\Phi(\theta_1,\theta_2,\theta_3))+c.p.(4)\\ &&-\big(\nu_1(L_{\theta_2}\Phi(\theta_3,\theta_4))-\nu_2(L_{\theta_1}\Phi(\theta_3,\theta_4))+\nu_3(\Phi(\theta_4,[\theta_1,\theta_2]_*))+\nu_4(\Phi([\theta_1,\theta_2]_*,\theta_3))+c.p.(6)\big)\\ &&-\big(d(\rho_*\theta_4)(\Phi(\theta_1,\theta_2,\theta_3))+c.p.(4)+(d\Phi([\theta_1,\theta_2]_*,\theta_3,\theta_4)+c.p.(6))\big)\\ &=&\nu_4((d_*\Phi)(\theta_1,\theta_2,\theta_3,\cdot))+c.p.(4)+d((d_*\Phi)(\theta_1,\theta_2,\theta_3,\theta_4)),\end{eqnarray*}
 	which vanishes as $d_*\Phi=0$. 
 	
 	This completes the proof of $\Omega^1(M)\xrightarrow{j}\mathrm{IM}^1(A)$ being a weak Lie $2$-algebra.
 \end{itemize}

 \subsubsection{Proof of Proposition \ref{Prop:IMmorphism}}.
 
 We first verify that $\psi_0(\nu,\theta)\in \Der(A)$, namely, to check the conditions
 \[\psi_0(\nu,\theta)(fx)=f\psi_0(\nu,\theta)(x)+\psi_0(\nu,\theta)(f) x,\] \[\mbox{ and } \psi_0(\nu,\theta)[x,y]=[\psi_0(\nu,\theta)(x),y]+[x,\psi_0(\nu,\theta)(y)],\]
 for all $f\in \CinfM$ and $x,y\in \Gamma(A)$. 
 In fact, for the first one, we have
 \[\psi_0(\nu,\theta)(fx)=\rho_*^*\nu(fx)+L_{\theta} (fx)=f\rho_*^*\nu(x)+fL_{\theta} (x)+\rho_*(\theta)(f)x=f\psi_0(\nu,\theta)(x)+\rho_*(\theta)(f)x;\]  For the second one, we use \eqref{IM1}, \eqref{IM2}, \eqref{LiP} and \eqref{Ld}, and obtain 
 \begin{eqnarray*}
 	&&\psi_0(\nu,\theta)[x,y]\\ &=&\rho_*^*\nu[x,y]+L_\theta[x,y]\\ &=&
 	\rho_*^*(L_{\rho x} \nu(y)-\iota_{\rho y} d\nu(x))+[L_\theta x,y]+[x,L_\theta y]-L_{\iota_x d_A \theta} y+\iota_{\iota_y d_A \theta} d_*x\\ &=&d_* \langle \rho x,\nu(y)\rangle +\rho_*^*(\iota_{\rho x} d\nu(y)-\iota_{\rho y} d\nu(x))+[L_\theta x,y]+[x,L_\theta y]-L_{\rho^*\nu(x)} y+\iota_{\rho^*\nu(y)} d_*x\\ &=&[\rho_*^*\nu(x)+L_\theta x,y]+[x,\rho_*^*\nu(y)+L_\theta y]\\ &=&[\psi_0(\nu,\theta)(x),y]+[x,\psi_0(\nu,\theta)(y)].
 \end{eqnarray*}
 
 Next, following  Equation \eqref{LiP}, we have 
 \[-\psi_0(j\gamma)(x)=\rho_*^* \iota_{\rho x} d\gamma+L_{\rho^*\gamma} x=-[\rho_*^*\gamma,x]=-t(\rho_*^* \gamma)(x).\]
 This confirms that the diagram stated in the proposition is commutative. Then we check the relations
 \begin{eqnarray*}
 	\psi_0[(\nu,\theta),(\nu',\theta')]-[\psi_0(\nu,\theta),\psi_0(\nu',\theta')]&=&t\psi_2((\nu,\theta),(\nu',\theta')),\\ 
 	\rho_*^*((\nu,\theta)\triangleright \gamma)-\psi_0(\nu,\theta)(\rho^*_*\gamma)&=&\psi_2((\nu,\theta),j\gamma).
 \end{eqnarray*}
 In fact, by direct calculation, we have
 \begin{eqnarray*}
 	\psi_0[\nu,\nu']-[\psi_0(\nu),\psi_0(\nu')]&=&\rho_*^*\big(\nu\circ \rho_{*}^*\circ \nu'-\nu'\circ \rho_{*}^*\circ \nu)-[\rho_*^*\nu,\rho_*^*\nu']=0;\\ 
 	\psi_0[\theta,\theta']-[\psi_0(\theta),\psi_0(\theta')]&=&L_{[\theta,\theta']_*}(\cdot)-[L_\theta(\cdot),L_{\theta'}(\cdot)],\\
 	\psi_0[\nu,\theta']-[\psi_0(\nu),\psi_0(\theta')]&=&\rho^*_*(-L_{\rho_{*} \theta'} \nu(\cdot)+\nu(L_{\theta'}(\cdot)))-[\rho_*^*\nu(\cdot),L_{\theta'}(\cdot)]\\ &=&-\rho^*_*L_{\rho_{*} \theta'} \nu(\cdot)+L_{\theta'} \rho_*^*\nu(\cdot).
 \end{eqnarray*}
 Together with \eqref{eq1} and \eqref{eq2}, we have
 \begin{eqnarray*}
 	\psi_0[(\nu,\theta),(\nu',\theta')]-[\psi_0(\nu,\theta),\psi_0(\nu',\theta')]&=&[\Phi(\theta,\theta'),\cdot]=t\psi_2((\nu,\theta),(\nu',\theta')).
 \end{eqnarray*}
 Moreover, we have
 \begin{eqnarray*}
 	\rho_*^*((\nu,\theta)\triangleright \gamma)-\psi_0(\nu,\theta)(\rho_*^*\gamma)&=&\rho_*^*(\nu(\rho_*^*\gamma)+L_{\rho_*\theta} \gamma)-\rho_*^*\nu(\rho_*^*\gamma)-L_{\theta} (\rho_*^*\gamma)\\ &=&\rho_*^*L_{\rho_*\theta} \gamma-L_{\theta} (\rho_*^*\gamma)=-\Phi(\theta,\rho^*\gamma)\\ &=&\psi_2((\nu,\theta),j\gamma),
 \end{eqnarray*}
 where we have used   \eqref{eq1} again.

 Finally, it remains to prove
 \[\rho_*^*[(\nu_1,\theta_1),(\nu_2,\theta_2),(\nu_3,\theta_3)]_3=[\psi_0(\nu_1,\theta_1),\psi_2((\nu_2,\theta_2),(\nu_3,\theta_3))]-\psi_2([(\nu_1,\theta_1),(\nu_2,\theta_2)],(\nu_3,\theta_3))+c.p.\]
 Let us compare the two sides of this equation.  By definition and \eqref{IM1}, we have
 \begin{eqnarray*}
 	\mathrm{LHS}&=&
 	\rho^*_*\nu_1(\Phi(\theta_2,\theta_3))+c.p.+d_*(\Phi(\theta_1,\theta_2,\theta_3)),\\ 
 	\mathrm{RHS}&=&\rho_*^*\nu_1(\Phi(\theta_2,\theta_3))+L_{\theta_1}\Phi(\theta_2,\theta_3)-\Phi([\theta_1,\theta_2]_*,\theta_3)+c.p..
 \end{eqnarray*}
 Since $d_*\Phi=0$, it is easy to see that they are identical. This completes the proof of $(\psi_0,\rho_*^*,\psi_2)$ being a  Lie $2$-algebra homomorphism.

\subsubsection{More corollaries}

Recall that Lie bialgebroids are   special   quasi-Lie algebroids $(A,\dstar, \Phi )$ with $\Phi$ being trivial \cite{MX1}. So, we use the pair $(A,\dstar )$ to denote a Lie bialgebroid.

\begin{corollary}
Let $(A,\dstar )$ be a Lie bialgebroid   over the base manifold $M$.% and $\pi$ the induced Poisson structure on the base manifold $M$. 
\begin{itemize}
\item[\rm(i)] There is a strict Lie  $2$-algebra structure on the complex $\Omega^1(M)\xrightarrow{j} \mathrm{IM}^1(A)$, where $j(\gamma):=(-\iota_{\rho(~\cdot~)} d\gamma, -\rho^*\gamma)$, the Lie bracket on $\mathrm{IM}^1(A)$ is given by Equation \eqref{bracket},
%\begin{eqnarray}
%[(\nu,\theta),(\nu', \theta')] &=&\nonumber \big(\nu\circ \rho_{*}^*\circ \nu'-\nu'\circ \rho_{*}^*\circ \nu+L_{(\rho_{*} \theta)}\nu'(\cdot)-\nu'(L_\theta^{A^*}(\cdot))-L_{(\rho_{*} \theta')} \nu(\cdot)+\nu(L_{\theta'}^{A^*}(\cdot)),\\ &&\label{bracket}[\theta,\theta']_{*}\big),
%\end{eqnarray}
and the action of $\mathrm{IM}^1(A)$ on $\Omega^1(M)$ is defined by 
\[(\nu,\theta)\triangleright \gamma:=\nu(\rho_{*}^*\gamma)+L_{\rho_* \theta} \gamma.\]
\item[\rm(ii)] There is a strict Lie  $2$-algebra homomorphism $(\psi_0,\rho_*^*)$:
\begin{equation*}
 		\xymatrix{
 			\Omega^1(M)\ar@{^{}->}[d]_{j}  \ar[r]^{\rho_{*}^*} &  \Gamma(A) \ar@{^{}->}[d]^{t} \\
 			\mathrm{IM}^1(A)\ar[r]^{\psi_0} & \Der (A)},
 	\end{equation*}
	where  \[\psi_0(\nu,\theta)=\rho_{*}^*\circ \nu+L_{\theta}(\cdot),\qquad \forall (\nu,\theta)\in \mathrm{IM}^1(A).\]
\end{itemize}
		\end{corollary}

We finally consider the particular case of $M$ being a single point. Indeed, an IM $1$-form on a Lie algebra $\g$  is  an element $\theta\in \g^*$ such that $\ad_x^*\theta=0$  for all $x\in \g$. So we can identify $\mathrm{IM}^1(\g)$ with $(\g^*)^{\mathrm{ad}}$ ($\mathrm{ad}^*$-invariant elements). 

\begin{corollary}
Let $(\g,d_*,\Phi)$ be a quasi-Lie bialgebra. 
\begin{itemize}
\item[\rm(i)] There is a Lie  algebra structure $\mathrm{IM}^1(\g)=(\g^*)^{\mathrm{ad}}$, where the bracket is $[\cdot,\cdot]_*$.
\item[\rm(ii)] There is a weak Lie  $2$-algebra homomorphism $(\psi_0,0,\psi_2)$ between two strict Lie $2$-algebras:
\begin{equation*}
 		\xymatrix{
 			0\ar@{^{}->}[d]_{0}  \ar[r]^{0} &  \g \ar@{^{}->}[d]^{t} \\
 			(\g^*)^{\mathrm{ad}}\ar[r]^{\psi_0} & \Der (\g)},
 	\end{equation*}
	where  $\psi_0(\theta)=\ad_\theta^*(\cdot)$ and $\psi_2:\wedge^2 (\g^*)^{\mathrm{ad}}\to \g$ is given by 
	\[\psi_2(\theta,\theta')=\Phi(\theta,\theta').\]
\end{itemize}
\end{corollary}

\subsection{Relating linear $1$-forms and  vector fields on a quasi-Lie bialgebroid}
Let $A$ be a vector bundle over $M$. Denote by $\Omega^k_{\mathrm{lin}}(A)$ and $\mathfrak{X}_{\mathrm{lin}}^k(A)$, respectively, the spaces of linear $k$-forms \cite{BC} and linear $k$-vector fields \cite{ILX} on   $A$. We adopt the identifications $\Omega^\bullet_{\mathrm{lin}}(A)\cong \Gamma(\mathfrak{J}^\bullet A^*)$ and $\mathfrak{X}^\bullet_{\mathrm{lin}}(A)\cong \Gamma(\mathfrak{D}^\bullet A^*)$ (see   \cite{LL}).

 Consider a quasi-Lie bialgebroid structure  
 $(A,\dstar ,\Phi)$ underlying the vector bundle $A$. The operator $d_*$ gives rise to a $2$-bracket on  $\Gamma(A^*)$ (not a Lie bracket), and it corresponds to a linear bivector field $P_A\in \mathfrak{X}_{\mathrm{lin}}^2(A)$ on $A$. In a usual manner, this $P_A$ defines a $2$-bracket $[\cdot,\cdot]_{P_A}$ on $\Omega^1_{\lin}(A)$.   Also $d_*$ defines an anchor map $\Gamma(A^*)\to \XX(M)$ which can be lifted to a map  
  \[P_A^\sharp: \Omega^1_{\lin}(A)\to \mathfrak{X}^1_{\lin}(A).\] 
% While on  $\mathfrak{X}^1_{\lin}(A)$, we have the 
%Schouten bracket. 

%Let $(A,\dstar )$ be a Lie bialgebroid. Then the Lie algebroid structure on $A^*$ gives a linear Poisson structure $P_A\in \mathfrak{X}_{\mathrm{lin}}^2(A)$ on the vector bundle $A$.

Due to \cite{BC}, we have an  {inclusion} $\iota: \mathrm{IM}^1(A)\hookrightarrow{}\Omega^1_{\mathrm{lin}}(A)$ given by
\begin{eqnarray}\label{i}
\iota(\nu,\theta)=\Lambda_\nu+d\Lambda_\theta,\qquad (\Lambda_\mu)_x:=(dq_A)^*\nu(x), \qquad \forall x\in A,
\end{eqnarray}
where $\Lambda_\theta$ is defined in the same fashion as that of $\Lambda_\nu$ and $q_A:A\to M$ is the projection.  {We will verify   that $\mathrm{IM}^1(A)$ with the bracket given in \eqref{bracket} is a subalgebra of $(\Omega^1_{\mathrm{lin}}(A),[\cdot,\cdot]_{P_A})$ (see (i) of Proposition  \ref{subalgebra}).}

According to \cite{ILX}, $1$-differentials of $A$ are instances of linear $1$-vector fields on $A$. In other words, we have  
an inclusion $\kappa: \Der(A)\hookrightarrow \mathfrak{X}^1_{\mathrm{lin}}(A)$ determined by
\begin{eqnarray}\label{k1}
\kappa(\delta)( dq_A^*f)&=&q_A^*  \delta f ;\\\label{k2}
\kappa(\delta)(dl_{\xi } )_x&=&    \kappa(\delta)( dq_A^*\xi (x))  -\langle \delta x,  \xi\rangle,
\end{eqnarray}
for $\xi \in \Gamma(A^*)$ and $f\in C^\infty(M)$.
%\textcolor{red}{decide later, general $k$ or just $1$}

\begin{proposition}\label{subalgebra}
Let $(A,\dstar ,\Phi)$ be a quasi-Lie bialgebroid. 
\begin{itemize}
\item[\rm(i)]
We have the following commutative diagram:
%\begin{equation*}
 %		\xymatrix{
 %			\mathrm{IM}^1 (A)\ar[r]^{\psi_0}\ar[d]^{\iota}_{} & \Der  (A) \ar[d]_{\kappa}^{} \\
%			\Omega^1_{\mathrm{lin}}(A)\ar[r]^{P_A^\sharp} & \mathfrak{X}^1_{\mathrm{lin}}(A)\\ \Gamma(\mathfrak{J}^1 A^*) \ar[r]^{\phi_0}\ar[u]_{\alpha}^{\cong} & \Gamma(\mathfrak{D}^1 A^*)\ar[u]^{\beta}_{\cong}},
 %	\end{equation*}	
\begin{equation*}
 		\xymatrix{
 			\mathrm{IM}^1 (A)\ar[r]^{\iota}_{\subset}\ar[d]^{\psi_0} &\Omega^1_{\mathrm{lin}}(A)\ar[d]^{P_A^\sharp} &  \Gamma(\mathfrak{J}^1 A^*) \ar[l]^{\cong}_{\alpha}	
			 \ar[d]_{\phi_0}^{} \\
			 \Der  (A)\ar[r]^{\kappa} _{\subset}& \mathfrak{X}^1_{\mathrm{lin}}(A) &\Gamma(\mathfrak{D}^1 A^*) \ar[l]^{\cong}_{\beta}},
 	\end{equation*}		
	 where \[\psi_0(\nu,\theta)=\rho_*^*\nu(\cdot)+L_\theta(\cdot),\qquad \phi_0(\liftingd \xi)=[\xi,\cdot]_*,\qquad \xi\in \Gamma(A^*).\] 
\item[\rm(ii)] Regarding the $2$-brackets of the top objects and the natural Lie bracket of commutator of the bottom objects,   every horizontal map  preserves the relevant brackets.
\end{itemize}
\end{proposition}
 \begin{proof}
 \textit{(i)}  We use the equality $\psi_0(\nu,\theta)(x)=\rho_*^*\nu(x)+L_\theta x=\rho^*_*(\nu(x)+d\theta(x))+\iota_\theta d_*x$ (for all $x\in \Gamma(A)$), and compute the following relations: 
\begin{eqnarray*}
\kappa(\psi_0(\nu,\theta))(dq_A^*f)&=&q_A^*((\rho_*\theta)f);\\
\kappa(\psi_0(\nu,\theta))(dl_\xi)_x&=&\kappa(\psi_0(\nu,\theta))(dq_A^*\xi(x))-\langle \rho^*_*(\nu(x)+d\theta(x))+\iota_\theta d_*x,\xi\rangle\\ &=&\rho_*\theta(\xi(x))-\langle \nu(x)+d\theta(x),\rho_*\xi\rangle-\langle d_*x,\theta\wedge \xi\rangle\\ 
&=&-\langle \rho_*\xi,\nu(x)\rangle+\langle x,[\theta,\xi]_*\rangle.
\end{eqnarray*}
Then comparing with the following lines
 \begin{eqnarray*}
 P_A^\sharp(\iota(\nu,\theta))(dq_A^*f)&=&P_A((dq_A)^*\nu(\cdot)+dl_\theta,dq_A^*f)=q_A^*((\rho_* \theta)f),\\ 
 P_A^\sharp(\iota(\nu,\theta))(dl_\xi)_x&=&P_A((dq_A)^*\nu(x)+dl_\theta,dl_\xi)=P_A((dq_A)^*\nu(x),dl_\xi))+P_A(dl_\theta,dl_\xi)_x\\ &=&-\langle \rho_*\xi,\nu(x)\rangle+\langle x, [\theta,\xi]_*\rangle,
 \end{eqnarray*}
 one immediately proves $\kappa\circ \psi_0=P_A^\sharp\circ \iota$. 
 
 Given any $\liftingd \xi\in \Gamma(\jet^1 A^*)$, we have
 \begin{eqnarray*}
 \langle P_A^\sharp \alpha(\liftingd \xi), dl_\eta\rangle&=&P_A(dl_\xi,dl_\eta)=l_{[\xi,\eta]_*},\qquad \forall \xi,\eta\in \Gamma(A^*),\\ 
 \langle \beta\phi_0(\liftingd \xi), dl_\eta\rangle&=&\langle \beta([\xi,\cdot]_*),dl_\eta\rangle=l_{[\xi,\eta]_*}, \end{eqnarray*}
which clearly implies that $P_A^\sharp\circ  \alpha= \beta\circ \phi_0$.

 \textit{(ii)}  It is known from \cite[Theorem 2.1]{LL}  and \cite[Proposition 3.8]{ILX} that $\beta$ and $\kappa$ are Lie algebra isomorphisms. So we are left  to show the following relations:
\begin{eqnarray}\label{homoeq1}
\iota[(\nu,\theta),(\nu',\theta')]&=&[\iota(\nu,\theta),\iota(\nu',\theta')]_{P_A},\qquad \forall (\nu,\theta),(\nu',\theta')\in \mathrm{IM}^1(A),\\ 
\label{homoeq2}\alpha([\mu,\mu']_{\mathfrak{J}^1 A^*})&=&[\alpha(\mu),\alpha(\mu')]_{P_A},\qquad \forall \mu,\mu'\in \Gamma(\jet^1 A^*).
\end{eqnarray}
Let us denote $(\tilde{\nu},\tilde{\theta})=[(\nu,\theta),(\nu',\theta')]$, where, by \eqref{bracket}, $\tilde{\theta}=[\theta,\theta']_*$. Then Equation \eqref{homoeq1} is equivalent to 
\[(\Lambda_{\tilde{\nu}},d\Lambda_{[\theta,\theta']})=[\Lambda_\nu+d\Lambda_\theta,\Lambda_{\nu'}+d\Lambda_{\theta'}]_{P_A}.\]
By definition, we have $\Lambda_{\theta}=l_{\theta}\in C^\infty_{lin}(A)$; then by $[dl_{\theta},dl_{\theta'}]_{P_A}=dl_{[\theta,\theta']_*}$, we  get  $d\Lambda_{[\theta,\theta']}=[d\Lambda_\theta,d\Lambda_{\theta'}]_{P_A}$.
 Therefore, we can compute 
\begin{eqnarray*}
[\Lambda_\nu,\Lambda_{\nu'}]_{P_A}&=&[(dq_A)^*\nu,(dq_A)^*\nu']_{P_A}\\ &=&dP_A((dq_A)^*\nu,(dq_A)^*\nu')+\big(\iota_{P_A^\sharp (dq_A)^*\nu} d(dq_A)^*\nu'-c.p.(\nu,\nu')\big)\\ &=&0-(dq_A)^*(\nu'\rho_*^*\nu-\nu\rho_*^*\nu'),
\end{eqnarray*}
where we have used the fact that $(P_A^\sharp (dq_A)^*\nu)_x=-\rho_*^*\nu(x)\in A_{m}$ for $x\in A_m$, which  is easily verified  using  local coordinates. 

In the meantime, we find
\begin{eqnarray*}
[d\Lambda_\theta,\Lambda_{\nu'}]_{P_A}&=&[dl_\theta,(dq_A)^*\nu']_{P_A}\\ &=&L_{P_A^\sharp (dl_\theta)} (dq_A)^*\nu'
%\\ &=&dP_A(d l_\theta,(dq_A)^*\nu')+\iota_{P_A^\sharp (dl_\theta)} d(dq_A)^*\nu'
\\ &=&(dq_A)^*(L_{\rho_*\theta} \nu'(\cdot)-\nu'(L_\theta(\cdot)).
\end{eqnarray*}
Combining these equalities, we obtain the desired \eqref{homoeq1}.
For \eqref{homoeq2}, taking $\mu=\liftingd \xi$ and $\mu'=\liftingd \xi'$, we have
\begin{eqnarray*}
\alpha[\liftingd \xi,\liftingd \xi']_{\jet^1 A^*}=\alpha(\liftingd [\xi,\xi']_*)=dl_{[\xi,\xi']_*}=[dl_\xi,dl_{\xi'}]_{P_A}=[\alpha(\liftingd \xi),\alpha(\liftingd \xi')]_{P_A}. 
\end{eqnarray*}
This completes the proof.
\end{proof}

\subsection{Two universal lifting theorems}\label{subSec:unilifting}
In this part, we connect our constructions of weak Lie $2$-algebras, respectively, on the groupoid level and on the associate tangent Lie algebroid level.

We need two basic mappings.
\begin{itemize}
	\item The correspondence $\sigma:~\Omega^1_{\mult}(\Gpd)\to \mathrm{IM}^1(A)$ is given as in Equations \eqref{sigma1} and \eqref{sigma2}. More generally, we have the map  $\sigma:~\Omega^k_{\mult}(\Gpd)\to \mathrm{IM}^k(A)$ for all integers $k$; see \cite{BC} or Appendix \ref{Appendix} for more details.
	\item The map $\tau:~\mathfrak{X}^1_{\mult}(\Gpd)\to \Der (A)$  {  given in \cite{BCLX} is defined as follows --- For any $\Pi\in \mathfrak{X}^1_{\mult}(\Gpd)$, there is a unique $\tau(\Pi)\in \Der(A)$ subject to the relations} 
	\[\overrightarrow{\tau(\Pi)f}=[t^*f, \Pi],\qquad \overrightarrow{\tau(\Pi) x}=[\overrightarrow{x},\Pi],\qquad\forall f\in C^\infty(M),x\in \Gamma(A).\]
\end{itemize}

\begin{theorem}\label{Ptheta}
Let $(\Gpd,P,\Phi)$ be a  quasi-Poisson Lie groupoid and $(A,\dstar ,\Phi)$ the corresponding quasi-Lie bialgebroid. Then the maps $P^\sharp$ and  $\psi_0$ (given by Proposition \ref{Prop:IMmorphism}) together with $\sigma$  and $\tau$  defined above   
form a commutative diagram:
\begin{equation*}
 		\xymatrix{
 			\Omega^1_{\mult}(\Gpd)\ar[r]^{P^\sharp}\ar[d]_{\sigma} & \mathfrak{X}^1_{\mult}(\Gpd) \ar[d]^{\tau}  \\
			\mathrm{IM}^1(A) \ar[r]^{\psi_0} & \Der (A)}.
 	\end{equation*}	
Moreover, if $\Gpd$ is $s$-connected and simply connected, then  both $\sigma$ and $\tau$ are  isomorphisms.
\end{theorem}
\begin{proof} Take any  $\Theta\in \Omega^1_{\mult}(\Gpd)$ and suppose that $\sigma(\Theta)=(\nu,\theta)\in \mathrm{IM}^1(A)$. The commutativity relation  $\psi_0\circ \sigma=\tau\circ P^\sharp$ amounts to \[\overrightarrow{\rho_*^*\nu(x)+L_\theta x}=[\overrightarrow{x},P^\sharp\Theta],\qquad \forall x\in \Gamma(A).\]
To prove it, we need to check
 \[\rho^*_*\nu(x)+L_\theta x=[\overrightarrow{x}, P^\sharp\Theta]|_M.\]
 In fact, we have 
\begin{eqnarray*}
(L_{\overrightarrow{x}} P)^\sharp(\Theta)&=&L_{\overrightarrow{x}} (P^\sharp \Theta)-P^\sharp (L_{\overrightarrow{x}} \Theta)= [\overrightarrow{x},P^\sharp \Theta]-P^\sharp(d\iota_{\overrightarrow{x}} \Theta+\iota_{\overrightarrow{x}} d\Theta),
\end{eqnarray*}
and hence 
\begin{eqnarray*}
[\overrightarrow{x}, P^\sharp \Theta]|_M&=&[\overrightarrow{x},P]^\sharp(\Theta)|_M
+P^\sharp(d\iota_{\overrightarrow{x}}\Theta)+\iota_{\overrightarrow{x}} d\Theta)|_M\\ &=&\iota_\theta d_*x+\rho_*^*(d\theta(x)+\nu(x))\\ &=&\rho^*_*\nu(x)+L_\theta x.\end{eqnarray*}
 
\end{proof}

\begin{theorem} Under the same assumption and notation as in Theorem \ref{Ptheta}, the triple of maps $(\sigma,\mathrm{id},0)$ is a   strict Lie $2$-algebra  morphism of  weak Lie $2$-algebras:
\begin{equation*}
 		\xymatrix{
 			\Omega^1(M)\ar[r]^{\mathrm{id}}\ar[d]_{J}& \Omega^1(M) \ar[d]^{j} \\
			\Omega^1_{\mult}(\Gpd) \ar[r]_{\sigma}  & \mathrm{IM}^1(A)}.
 	\end{equation*}	
 	If   $\Gpd$ is $s$-connected and  simply connected, then    $(\sigma,\mathrm{id},0)$ is an isomorphism.
\end{theorem}
\begin{proof} 
We first show that, for   $\Theta$ and $\Theta'\in \Omega^1_{\mult}(\Gpd)$ mapping to,  respectively, $(\nu,\theta),(\nu',\theta')\in \mathrm{IM}^1(A)$ by $\sigma$,  the resulting $[\Theta,\Theta']_{P}\in \Omega^1_{\mult}(\Gpd)$ is mapped to $[(\nu,\theta),(\nu',\theta')]$  (defined in \eqref{bracket}).

By definition, we have
\[[\Theta,\Theta']_P=L_{P^\sharp \Theta} \Theta'-\iota_{P^\sharp \Theta'} d\Theta.\]
It follows from  Theorem \ref{Ptheta} that
the $1$-differential $\sigma(P^\sharp \Theta)=(\delta_0,\delta_1)\in \Der(A)$ is \[\delta_0=\rho_* \theta,\qquad  \delta_1(x)=\rho_*^*\nu(x)+L_\theta x.\]
A well-known fact is the IM $2$-form   $\sigma(d\Theta)$   $=(0,\nu)$ provided that $\sigma(\Theta)=(\nu,\theta)\in \mathrm{IM}^1(A)$.  
Applying a technical Lemma \ref{con} which is presented in the appendix,   
%for $P(\Theta,\Theta')\in \Omega^0_{\mult}(\Gpd)$, the corresponding $\sigma(P(\Theta,\Theta'))=$   
%$(\nu_{P(\Theta,\Theta')},0)\in \mathrm{IM}^0(A)$ is given by  
%\[\nu_{P(\Theta,\Theta')}(x)=\iota_{\rho_*\theta} \nu'(x)+L_{\rho_*\theta}\theta'(x)-\theta'(\rho_*^*\nu(x)+L_\theta x)=\langle \nu'(x),\rho_*\theta\rangle-\langle \nu(x),\rho_* \theta'\rangle+\langle [\theta,\theta']_*,x\rangle.\]
  for $\iota_{P^\sharp \Theta'} d\Theta\in \Omega^1_{\mult}(\Gpd)$, $\sigma(\iota_{P^\sharp \Theta'} d\Theta)$   $=(\nu_1,\theta_1)\in \mathrm{IM}^1(A)$ is given by 
\begin{eqnarray*}
\nu_{1}(x)&=&L_{\rho_*\theta'} \nu(x)-\nu(\rho_*^*\nu'(x)+L_{\theta'} x),\\ 
\theta_{1}(x)&=&-\iota_{\rho_*\theta'} \nu(x)=-\langle \nu(x),\rho_*\theta'\rangle.
\end{eqnarray*}
And $\sigma(L_{P^\sharp \Theta} d\Theta')=(\nu_{2},\theta_{2})\in \mathrm{IM}^1(A)$ is given by 
\begin{eqnarray*}
\nu_{2}(x)&=&L_{\rho_*\theta} \nu'(x)-\nu'(\rho_*^*\nu(x)+L_{\theta} x),\\ 
\theta_{2}(x)&=&L_{\rho_*\theta} \theta'(x)-\theta'(\rho_*^*\nu(x)+L_{\theta} x)=-\langle \rho_* \theta',\nu(x)\rangle+\langle [\theta,\theta']_*,x\rangle.
\end{eqnarray*}
Thus, assuming $\sigma([\Theta,\Theta']_P)=(\tilde{\nu},\tilde{\theta})$, we have
\begin{eqnarray*}
\tilde{\nu}(x)&=&\nu_2(x)-\nu_1(x)=L_{\rho_*\theta} \nu'(x)-\nu'(\rho_*^*\nu(x)+L_\theta x)-L_{\rho_*\theta'} \nu(x)+\nu(\rho_*^*\nu'(x)+L_{\theta'} x),\\
\tilde{\theta}(x)&=&\theta_2(x)-\theta_1(x)=\langle [\theta,\theta']_*,x\rangle.
\end{eqnarray*}
Comparing with \eqref{bracket}, we have proved \begin{eqnarray}\label{goal}
\sigma([\Theta,\Theta']_P)=(\tilde{\nu},\tilde{\theta})=[(\nu,\theta),(\nu',\theta')].
\end{eqnarray}

Then for $\gamma\in \Omega^1(M)$, we have $J\gamma=s^*\gamma-t^*\gamma\in \Omega_{\mult}^1(\Gpd)$. Suppose that   $\sigma(J\gamma)=(\nu,\theta)\in \mathrm{IM}^1(A) $ where \begin{eqnarray*}
\langle \nu(x),Y\rangle &=&d(s^*\gamma-t^*\gamma)(x,Y)=-(d\gamma)(\rho x, Y),\qquad x\in \Gamma(A), Y\in \mathfrak{X}^1(M);\\ 
\theta(x)&=&(s^*\gamma-t^*\gamma)(x)=-\gamma(\rho x).
\end{eqnarray*}
Hence, we find
\begin{eqnarray}\label{goal2}
\sigma(J\gamma)=\sigma(s^*\gamma-t^*\gamma)=(-\iota_{\rho(\cdot)} d\gamma, -\rho^*\gamma)=j(\gamma).
\end{eqnarray}
Therefore,  it remains to prove that
\begin{eqnarray}\label{goal3}
[\Theta_1,\Theta_2,\Theta_3]_3=[(\nu_1,\theta_1),(\nu_2,\theta_2),(\nu_3,\theta_3)]_3,\qquad \forall \Theta_i\in \Omega^1_{\mult}(\Gpd),
\end{eqnarray}
where $(\nu_i,\theta_i)=\sigma(\Theta_i)\in \mathrm{IM}^1(A)$. In fact,   we have 
\begin{eqnarray*}
s^*[\Theta_1,\Theta_2,\Theta_3]_3&=&d\overrightarrow{\Phi}(\Theta_1,\Theta_2,\Theta_3)+\big(\iota_{\overrightarrow{\Phi}(\Theta_1,\Theta_2)} d\Theta_3+c.p.\big)\\ &=&s^*d\Phi(\theta_1,\theta_2,\theta_3)+s^*\big(\nu_3(\Phi(\theta_1,\theta_2))+c.p.\big)\\ &=&s^*[(\nu_1,\theta_1),(\nu_2,\theta_2),(\nu_3,\theta_3)]_3,
\end{eqnarray*}
which justifies \eqref{goal3} (as $s^*$ is injective).
In conclusion, Equations \eqref{goal}-\eqref{goal3} imply that $(\sigma,\mathrm{id},0)$ is a Lie $2$-algebra isomorphism.
\end{proof}

In summary, if a quasi-Poisson groupoid 
  $(\Gpd,P,\Phi)$ is $s$-connected and  simply connected, then regarding the associated quasi-Lie bialgebroid $(A,\dstar ,\Phi)$, we have the following commutative diagrams:
	\begin{equation*}
		\xymatrixcolsep{5pc}\xymatrix@!0{
			& \Omega^1 (M) \ar[rr]^{\rho_*^*} \ar@{^{ }->}'[d][dd]^{j}  &  & \Gamma(A) \ar@{^{}->}[dd]^{t}       \\
			\Omega^1(M) \ar[ur]^{=}\ar[rr]^{\quad \quad \quad p^\sharp}\ar@{^{ }->}[dd]_{J}  &  & \Gamma(A) \ar[ur]_{=}\ar@{^{}->}[dd]_{T} \\
			&   \mathrm{IM}^1(A) \ar'[r][rr]^{P_A^\sharp \quad\quad\quad}&  & \Der (A).           \\
			\Omega^1_{\mathrm{mult}}(\Gpd) \ar[rr]^{P^\sharp}   \ar[ur]^{\cong} &   &
			\mathfrak{X}^1_{\mathrm{mult}}(\Gpd) \ar[ur]_{\cong}        }
	\end{equation*} 
 Here, the front  and back faces are weak  Lie $2$-algebra morphisms as described by Propositions \ref{Lie2homo} and \ref{Prop:IMmorphism} (observing that $p^\sharp=\rho_*^*$), respectively.

% \tohandle{---} 
 
% DIRAC=. Bursztyn, H., Crainic, M.: Dirac geometry, quasi-Poisson actions and D/G-valued moment maps.
% J. Differ. Geom. 82, 501–566 (2009). ArXiv:0710.0639

 %MOMENT=Mikami, K., Weinstein, A.: Moments and reduction for symplectic groupoid actions. Publ. RIMS,
 %Kyoto Univ. 24, 121–140 (1988)
 
 %MORITA=Xu, P.: Momentum maps and Morita equivalence. J. Differ. Geom. 67, 289–333 (2004)

%\tohandle{---} 
\appendix
\section{A technical fact}\label{Appendix}  
Let $\Gpd$ be a Lie groupoid over $M$ and $A$ the tangent Lie algebroid of it. 
A basic mapping $\sigma:~\Omega^k_{\mult}(\Gpd)\to \mathrm{IM}^k(A)$ for all integers $k$ is introduced in \cite{BC}. Here we briefly recall this map. Indeed, an  {\bf IM $k$-form } of a Lie algebroid $A$ is a pair $(\nu,\theta)$, where $\nu: A\to \wedge^k T^*M$ and $\theta: A\to \wedge^{k-1}T^*M$ are bundle maps satisfying the constraints
\begin{eqnarray*}
\iota_{\rho(x)}\theta(y)&=&-\iota_{\rho(y)}\theta(x),\\
\theta([x,y])&=&L_{\rho(x)} \theta(y)-\iota_{\rho(y)}d \theta(x)-\iota_{\rho(y)}\nu(x),\\\mbox{ and }~
\nu([x,y])&=&L_{\rho(x)} \nu(y)-\iota_{\rho(y)}d \nu(x),
\end{eqnarray*}
for $x,y\in \Gamma(A)$. In particular,  an IM $1$-form is a pairs satisfying \eqref{IM1} and \eqref{IM2}. 
Given any $\omega\in \Omega^k_{\mult}(\Gpd)$, the corresponding IM $k$-form $\sigma(\omega)=(\nu,\theta)$ is defined by the following relations
\begin{eqnarray*}
 	  	 \langle\nu(x),U_1\wedge \cdots \wedge U_k\rangle&=&d\omega  (x,U_1,\cdots,U_k), \\
 	  	 \mbox{ and }\quad \langle\theta(x),U_1\wedge \cdots \wedge U_{k-1}\rangle&=&\omega  (x,U_1,\cdots,U_{k-1}), 
 	  	 \end{eqnarray*} 	  	
 	  	for $x\in \sections{A}$ and $U_i\in \mathfrak{X}^1(M)$.  The multiplicativity property of $\omega$ ensures that $(\nu,\theta)$ fulfills the aforementioned conditions of an IM $k$-form   of $A$.

It has been shown in \cite{CLL}*{Lemma 3.8} that for %$\alpha\in \Omega^1_{\multi}(\Gpd), \Pi\in \mathfrak{X}^k_{\mult}(\Gpd)$ and 
$X\in \mathfrak{X}^1_{\mult}(\Gpd)$ and $\Theta\in \Omega^k_{\mult}(\Gpd)$, we have the contraction
%$\iota_\alpha\Pi\in \mathfrak{X}^{k-1}(\Gpd)$ and 
$\iota_X \Theta\in \Omega^{k-1}_{\mult}(\Gpd)$ and the Lie derivative $L_X\Theta\in \Omega^k_{\mult}(\Gpd)$ since the de Rham differential preserves  multiplicativity properties. Now we would like to find the IM-forms corresponding to $\iota_X \Theta$ and $L_X\Theta$ via $\sigma$. 

 Recall the map $\tau:~\mathfrak{X}^1_{\mult}(\Gpd)\to \Der (A)$ defined in Section \ref{subSec:unilifting}.

\begin{lemma}\label{con}
	 For $X\in \mathfrak{X}^1_{\mult}(\Gpd)$, $\Theta\in \Omega^k_{\mult}(\Gpd)$, suppose that $\tau(X)=(\delta_0,\delta_1)\in \Der (A)$, $\sigma(\Theta)=(\nu,\theta)\in \mathrm{IM}^k(A)$,     $\sigma(\iota_X \Theta)=(\tilde{\nu},\tilde{\theta})\in \mathrm{IM}^{k-1}(A)$, and $\sigma(L_X \Theta)=(\hat{\nu},\hat{\theta})\in \mathrm{IM}^{k}(A)$. Then we have 
	\begin{eqnarray*}
		\tilde{\nu}(x)=\iota_{\delta_0}(\nu(x))+L_{\delta_0}(\theta(x))-\theta(\delta_1 x),\qquad 
		\tilde{\theta}(x)=-\iota_{\delta_0}(\theta(x)),\qquad \forall x\in \Gamma(A),
	\end{eqnarray*}
	and 
	\begin{eqnarray*}
		\hat{\nu}(x)=L_{\delta_0} (\nu(x))-\nu(\delta_1 x),\qquad 
		\hat{\theta}(x)=L_{\delta_0} (\theta(x))-\theta(\delta_1 x).
	\end{eqnarray*}\end{lemma}
\begin{proof} 
	The  proof is simply straightforward computations --- For $U_i\in TM$, we have
	\begin{eqnarray*}
		\langle \tilde{\theta}(x), U_1\wedge\cdots \wedge U_{k-2}\rangle&=&(\iota_X \Theta)(x,U_1,\cdots, U_{k-2})\\ &=&-\Theta(x,X|_M,U_1,\cdots,U_{k-2})\\ &=&-\langle \iota_{\delta_0} \theta(x), U_1\wedge\cdots \wedge U_{k-2}\rangle,
	\end{eqnarray*}
	and
	\begin{eqnarray*}
		\langle \tilde{\nu}(x), U_1\wedge\cdots \wedge U_{k-1}\rangle&=&d(\iota_X \Theta)(x,U_1,\cdots, U_{k-1})\\ &=&(L_X \Theta-\iota_X d\Theta)(x,U_1,\cdots,U_{k-1})\\ &=&X|_M\Theta(x,U_1,\cdots,U_{k-1})-\Theta([X,\overrightarrow{x}]|_M,U_1,\cdots,U_k)\\ &&-\sum_i \Theta(x,\cdots,[X|_M,U_i],\cdots)+\langle \iota_{\delta_0}\nu(x),U_1\wedge \cdots\wedge U_{k-1})\\ &=&\langle L_{\delta_0} \theta(x)-\theta(\delta_1x)+\iota_{\delta_0}\nu(x),U_1\wedge \cdots\wedge U_{k-1}\rangle.
	\end{eqnarray*}
 These are  the desired formulas of $\tilde{\nu}$ and $\tilde{\theta}$. 
 
 Based on the well-known fact that the IM $k$-form $\sigma(d\Theta)=(0,\nu)$ if $\sigma(\Theta)=(\nu,\theta)\in \mathrm{IM}^{k-1}(A)$, we can  determine the IM $k$-forms of $d\iota_X \Theta$ and $\iota_X d\Theta$ as follows:  	
\[\sigma(d\iota_X \Theta)=(0,\tilde{\nu}),\qquad \sigma(\iota_X d\Theta)=\big(L_{\delta_0}(\nu(\cdot))-\nu(\delta_1 (\cdot), -\iota_{\delta_0}(\nu(\cdot))\big).\]
 So the IM $k$-form of $L_X \Theta$ is as described.
\end{proof}

\end{document}